\documentclass[11pt,a4paper,reqno]{amsart}
\usepackage{amssymb,amsmath,amsthm}
\usepackage[utf8]{inputenc}
\usepackage[T1]{fontenc}
\usepackage{enumerate}
\usepackage[all]{xy}
\usepackage{mathrsfs}
\usepackage{array}
\usepackage{fullpage}
\usepackage{comment}
\usepackage{xcolor}

\def\A{\mathbb{A}}
\def\Z{\mathbb{Z}}
\def\N{\mathbb{N}}
\def\Q{\mathbb{Q}}

\def\s1{{\rm Spec}^{(1)}}

\newtheorem{thm}{Theorem}[section]
\newtheorem{lem}[thm]{Lemma}
\newtheorem{prop}[thm]{Proposition}
\newtheorem{cor}[thm]{Corollary}
\newtheorem{claim}{Claim}

\theoremstyle{definition}
\newtheorem{defn}[thm]{Definition}

\newtheorem{exa}[thm]{Example}
\newtheorem{rmk}[thm]{Remark}

\title{Integrally Hilbertian rings and  \break the polynomial Schinzel hypothesis}
\author{Angelot Behajaina}
\email{angelot.behajaina@univ-lille.fr}
\address{Univ. Lille, CNRS, UMR 8524, Laboratoire Paul Painlev\'e, F-59000 Lille, France}

\author{Pierre D\`{e}bes}
\email{pierre.debes@univ-lille.fr}
\address{Univ. Lille, CNRS, UMR 8524, Laboratoire Paul Painlev\'e, F-59000 Lille, France}

\author{Joachim König}
\email{jkoenig@knue.ac.kr}
\address{Department of Mathematics Education, Korea National University of Education, Cheongju, South Korea}


\begin{document}

\begin{abstract}
The classical Hilbert specialization property is a field-theoretic tool ensuring that polynomial irreducibility over a field is preserved under specialization of some of the variables. We develop an integral counterpart  by introducing the notion of {integrally Hilbertian rings}, where specialization takes place inside a ring and irreducibility is required over the ring. A core part shows how new obstacles to irreducibility such as 
%as the possibility of 
coefficient divisors or fixed divisors can be dealt with over Krull domains, a large class of rings including UFDs, Dedekind domains, etc. 
%and Noetherian integrally closed domains. 
As a result, we obtain a general criterion for integral hilbertianity, along with 
%We establish a general criterion showing that a Hilbertian ring which is a Krull domain is integrally Hilbertian. This yields 
many examples, \hbox{e.g.} all rings of integers of number fields.
%, their localizations, their integral closures in finite extensions.  
%In particular, we prove that for every integral domain $\mathcal{Z}$, the polynomial ring $\mathcal{Z}[U]$ is integrally Hilbertian, even when $\mathcal{Z}$ is not integrally closed.
%Our approach relies 
%on a refinement of the coprime Schinzel hypothesis and 
%on the introduction of 
%\emph{locally Schinzel rings}, characterized by 
%two structural properties that are shown to hold for Krull domains. 
Polynomial rings over arbitrary domains are other examples. As an application, we prove a polynomial variant of the Schinzel Hypothesis
on prime values of polynomials with integer coefficients: if $\mathcal{Z}$ is an integrally Hilbertian ring, the hypothesis becomes a true statement if the ring of integers ${\mathbb Z}$ is replaced by the polynomial ring $\mathcal{Z}[U]$ and ``prime'' by ``irreducible''.
%satisfies a Schinzel-type property. 
This result generalizes previous works and fits in a unified framework for Schinzel-type phenomena that we introduce. We
%, We generalize previous results obtained in the UFD case and offer 
further obtain an additional conclusion that has some noteworthy consequences  for the classical Schinzel Hypothesis itself.
%over $\mathbb{Z}$.
%, including a reduction to polynomials of arbitrarily large degree.
%Our work provides a unified integral framework for Hilbert specialization and Schinzel-type phenomena, extending classical field-theoretic methods to a broad class of arithmetic domains.
\end{abstract}

%\begin{abstract}
%We prove that all rings of integers  of number fields are ``integrally Hilbertian''. That is, 
%they satisfy an ``integral'' version of the classical Hilbert irreducibility property,
%to the effect that, under appropriate  irreducibility assumptions on given multivariate polynomials with coefficients in such a ring ${\mathcal Z}$, one can specialize in the ring ${\mathcal Z}$ some of the variables  in such a way that irreducibility over the ring ${\mathcal Z}$ is preserved for the specialized polynomials. We identify an intermediate type of ring, which we call ``Schinzel ring'', that is central in the study of this new Hilbert property. Dedekind domains and Unique Factorization Domains are shown to be Schinzel rings, leading to other examples of integrally Hilbertian rings. A main application is a polynomial version,  for the ring ${\mathcal Z}[Y_1,\ldots,Y_n]$, of the Schinzel Hypothesis about primes in value sets of polynomials.
%Some consequences for the ring of integers itself are also deduced.
%\end{abstract}

\maketitle

\section{Introduction} \label{sec:intro}

The \emph{Hilbert specialization property}, by which irreducibility of polynomials is preserved under specialization of some of the variables (Definition \ref{def:Hilb-ring-field}), is one of the few general and powerful tools in Arithmetic Geometry, making it possible, in certain situations, to specialize some parameters while keeping the algebraic structure unchanged. 
%\cite{FJ23}. 
For example, when a field ${\mathcal Q}$ has the property, \hbox{i.e.}, is a \emph{Hilbertian field}, any finite separable field extension of $\mathcal{Q}(T)$ yields, via appropriate specializations  in $\mathcal{Q}$ of the variable $T$, extensions of $\mathcal{Q}$ with the same degree and the same Galois group. Hilbert spe\-cialization can also be helpful in affine geometry to produce irreducible affine ${\mathcal Q}$-varieties
%hypersurfaces of the affine space ${\mathbb A}^n_{\mathcal Q}$ 
from an algebraic $\mathcal{Q}(T)$-family of varieties parametrized by parameters $T$, no\-tably when Bertini-type results, which require irreducibility of the family {over the algebraic closure of} ${\mathcal Q}(T)$, \hbox{do not apply.}
\vskip 1mm

Hilbert specialization is a \emph{field} process with polynomials defined over a \emph{field} and irreducibility requested over that \emph{field}. As initiated in \cite{BDKN22}, we aim to make it an \emph{integral} tool with specialization \emph{in a ring} {and} irreducibility required \emph{over the ring}. The difficulty in passing from fields to rings is that irreducibility over a ring is sensitive to coefficient divisibility and ``fixed divisors'' phenomena (as illustrated before Definition \ref{def:int-Hilb}); these have no analog over fields.
\vskip 1mm

Our proposed integral version of Hilbertian field theory goes as follows. Definition \ref{def:int-Hilb} of \emph{integrally Hilbertian rings} encapsulates the problem. We then have three main results. Theorem \ref{thm:main1} is a criterion for a ring to be integrally Hilbertian, leading to many examples, \hbox{e.g.} all rings of integers of number fields (Corollary \ref{cor:main1}). Theorem \ref{thm:main1-2} provides another fundamental example: any polynomial ring over an integral domain. Applications concern the celebrated Schinzel Hypothesis on prime values of polynomials with integer coefficients. Theorem \ref{thm:schinzel-main} establishes a polynomial variant, typically for rings like ${\mathbb Z}[U]$ instead of ${\mathbb Z}$. This result improves on \cite[Theorem 1.1]{BDN20b}; in particular, we have an additional assertion that has some noteworthy implications for the original Schinzel Hypothesis itself, over ${\mathbb Z}$ (Corollary \ref{cor:SH}).

%\penalty -3000

%, which improves on a previous version proved in \cite{BDN20b}. Beside the bigger generality and the better bounds, we establish an additional conclusion that has some noteworthy implications, even to the original Schinzel hypothesis over ${\mathbb Z}$ (Corollary \ref{cor:SH}). 
\vskip 1mm 
We now present our work in more detail.

\subsection{Integrally Hilbertian rings} Start with these classical definitions.

\begin{defn} \label{def:Hilb-ring-field} A domain ${\mathcal Z}$ with fraction field $\mathcal{Q}$, of characteristic $0$ or imperfect  (\hbox{i.e.} of characteristic $p>0$ and ${\mathcal Z}\not={\mathcal Z}^p$), is a \emph{Hilbertian ring} if the following \emph{Hilbert specialization property} holds. For
any two tuples of indeterminates $\underline{T}=(T_1,\ldots,T_k)$ (the \emph{parameters}) and $\underline{Y}=(Y_1,\ldots,Y_n)$ (the \emph{variables}), with $k,n\geq 1$, and any irreducible polynomials $P_1,\ldots,P_s \in {\mathcal Q}[\underline{T},\underline{Y}]$, of degree $\geq 1$ in $\underline{Y}$, the subset of ${\mathcal Z}^k$ consisting of all $\underline{t}$ such that $P_1(\underline{t},\underline{Y}),\dots,P_s(\underline{t},\underline{Y})$ are irreducible in $\mathcal{Q}[\underline{Y}]$ is Zariski-dense. A \emph{Hilbertian field} is a field that is a {Hilbertian ring}. 
\end{defn}

Hilbertian fields include the rational field ${\mathbb Q}$, rational function fields $\kappa(u_1,\ldots,u_r)$ in $r\geq 1$ variables over an arbitrary field $\kappa$, and their finite extensions. 
%(more generally all \emph{fields with a product formula} (\S \ref{sssec:FDP}) of characteristic $0$ or imperfect \cite[Theorem 4.6]{BDN20b}). 
Hilbertian rings include all domains ${\mathcal Z}$ such that the fraction field ${\mathcal Q}$ is a field of the preceding type. \cite{FJ23} is a classical reference on Hilbertian fields and rings. The \emph{imperfectness assumption} at the beginning of Definition \ref{def:Hilb-ring-field} is here for consistency with \cite{FJ23} and is not restrictive (Remark \ref{remiiredZvsQ}). 
\vskip 1mm

%We only consider Hilbertian rings ${\mathcal Z}$ and Hilbertian fields ${\mathcal Q}$ satisfying this \emph{imperfectness assumption}: ${\mathcal Z}$ (or ${\mathcal Q}$) is of characteristic $0$ or is imperfect (\hbox{i.e.} ${\mathcal Z}\not={\mathcal Z}^p$ if $p={\rm char}({\mathcal Z})>0$); see Remark \ref{remiiredZvsQ}. For simplicity we include it a priori in our definitions of Hilbertian fields and Hilbertian rings.

A key obstruction in the intended integral setting is the existence of fixed divisors. Recall that given an integral domain ${\mathcal Z}$ with fraction field ${\mathcal Q}$ and a polynomial $P \in {\mathcal Z}[\underline{T},\underline{Y}]$, a nonunit $a \in {\mathcal Z}$, $a\not=0$, is called a \emph{fixed divisor} of $P$ \hbox{w.r.t.} $\underline{T}$ if
$P(\underline{t},\underline{Y}) \equiv 0 \pmod{a}$ for every $\underline{t} \in {\mathcal Z}^k$. For example $2\in \Z$ is a fixed divisor of $P=(T^2-T)Y+(T^2-T-2)$
\hbox{w.r.t.} ${T}$; consequently, no specialized polynomial $P(t,Y)$ with $t\in {\mathbb Z}$
% $t\in {\mathbb Z} \setminus \{0,1\}$ 
is irreducible in $\Z[Y]$. 
\vskip 1mm

The following property of a ring is a central theme of the paper. 

\begin{defn} \label{def:int-Hilb}
A domain ${\mathcal Z}$ is said to be \emph{integrally Hilbertian} if the following condition is satisfied. For any polynomials $P_1,\ldots,P_s \in \mathcal{Z}[\underline{T},\underline{Y}]$, irreducible  in ${\mathcal Q}[\underline{T},\underline{Y}]$, of degree $\geq 1$ in $\underline{Y}$, and such that  the product $P_1 \cdots P_s$ has no fixed divisor \hbox{w.r.t.} $\underline{T}$, there is a Zariski-dense subset ${H} \subset {\mathcal Z}^k$ such that for every $\underline{t} \in {H}$, 
\vskip 0,5mm

\noindent
{\rm (a)} \emph{the polynomials $P_1(\underline{t},\underline{Y}),\dots,P_s(\underline{t},\underline{Y})$ are irreducible in $\mathcal{Q}[\underline{Y}]$}, 
\vskip 0,5mm

\noindent
{\rm (b)} \emph{the product $\prod_{i=1}^s P_i(\underline t,\underline Y)$ has no nonunit divisor in ${\mathcal Z}$}.
\end{defn}

This gives in particular that, as intended, for every $\underline{t} \in {H}$, 
\vskip 0,5mm

\noindent
(c) {\emph{the polynomials $P_1(\underline{t},\underline{Y}),\dots,P_s(\underline{t},\underline{Y})$ are irreducible in $\mathcal{Z}[\underline{Y}]$ and in $\mathcal{Q}[\underline{Y}]$.}
\vskip 0,7mm

\noindent
Assertion (c) can in fact equivalently replace (a)\hskip 2pt \&\hskip 2pt (b) if \emph{$s=1$ (one polynomial) or ${\mathcal Z}$ is a \emph{UFD}}.\footnote{As is usual, UFD stands for Unique Factorization Domain, and similarly, PID for Principal Ideal Domain.}
\vskip 1mm

Condition (b) shows the difference between the Hilbertian ring notion -- only the field condition (a) is required, fixed divisors are irrelevant -- and the integrally Hilbertian property -- one must simultaneously control divisors (a) of positive degree and (b) of zero degree.
%irreducibility (condition (a)) and divisibility (condition (b)) .
\vskip 1mm

Remark \ref{rmk:no-fixed-divisor} explains how the ``no fixed divisor assumption'' can concretely be guaranteed. Furthermore, under some common condition on ${\mathcal Z}$ (\hbox{e.g.} being a Krull domain as discussed below), one can get rid of fixed divisors of the product $P_1\cdots P_s$ and so reduce to the assumption, at the cost of localizing ${\mathcal Z}$ to ${\mathcal Z}[1/\varphi]$ for some nonzero $\varphi \in {\mathcal Z}$ (Lemma \ref{lem:fixed-divisors-removal}).
\vskip 1mm

As observed in Proposition \ref{intH=>Hring}, an integrally Hilbertian ring must be a Hilbertian ring. Conversely, the following result provides a condition for a Hilbertian ring to be integrally Hilbertian.

\begin{thm} \label{thm:main1} Let ${\mathcal Z}$ be a Hilbertian ring that is also a Krull domain. Then ${\mathcal Z}$ is an integrally Hilbertian ring.
\end{thm}

Definition of Krull domains is recalled in \S \ref{ssec:krulldom}. They include:
\vskip 0,2mm

\noindent
(a) Unique Factorization Domains,
\vskip 0,5mm

\noindent
(b) rings that are Noetherian and integrally closed (in particular Dedekind domains).
\vskip 1mm

Special case (a) was established in  \cite{BDKN22}. 
Case (b) and the general Krull case are new. 
\vskip 1mm

An advantage of the Krull assumption 
is that it is preserved by localization (by any multiplicative subset), and by taking integral closures in finite extensions of the fraction field (Remark \ref{rmk:Krull-domain}).
%${\mathcal Q}$ \cite[\S 41.B]{Mat80}.
The same is true with the Hilbertian ring property (Remark \ref{remiiredZvsQ}). Whence this complement:
%to Theorem \ref{thm:main1}:
\vskip 1,5mm

\noindent
{\bf Theorem \ref{thm:main1} \hbox{\rm (continued)}} \emph{If ${\mathcal Z}$ is a Hilbertian ring and a Krull domain,
not only ${\mathcal Z}$ 
%is an integrally Hilbertian ring, 
but also 
any localization of ${\mathcal Z}$, and any integral closure of ${\mathcal Z}$ in a finite extension of ${\mathcal Q}$, is integrally Hilbertian.}
\vskip 1,5mm

%\noindent 
In particular, 
%under the assumptions of (*), the integrally Hilbertian property may be helpful 
for polynomials $P_1,\ldots,P_s$ \emph{with} some fixed divisors, the integrally Hilbertian conclusion  holds over  ${\mathcal Z}[1/\varphi]$ for some nonzero $\varphi \in {\mathcal Z}$.  
%Indeed, by Lemma \ref{lem:fixed-divisors-removal}, there is a nonzero element $\varphi \in {\mathcal Z}$ such that the no fixed divisor assumption holds over the extension ${\mathcal Z}[1/\varphi]$. So the conclusion of Definition \ref{def:int-Hilb} holds for  $P_1,\ldots,P_s$ over  ${\mathcal Z}[1/\varphi]$.
\vskip 1mm

As the ring of integers ${\mathbb Z}$ and the polynomial ring $\kappa[\underline x]$, with $\kappa$ any field and $\underline x$ any nonempty set of indeterminates, are  Hilbertian rings, and  are UFD, we immediately deduce from Theorem \ref{thm:main1} these examples of integrally Hilbertian rings.

\begin{cor}\label{cor:main1} The ring of integers ${\mathcal Z}$ of any number field is {integrally Hilbertian}. If $\kappa$ is any field, then the integral closure ${\mathcal Z}$ of $\kappa[\underline x]$ in any finite extension $E$ of $\kappa(\underline x)$ is integrally Hilbertian.
\end{cor}

\begin{rmk} 
\label{rmk:number-fields} %Replacing prime elements by prime ideals is a natural idea when considering general Dedekind domains. One however faces the following difficulty: a polynomial $P \in \mathcal{Z}[\underline{T},\underline{Y}]$ may have no fixed divisor $a\in {\mathcal Z}$ but still have a fixed \emph{prime ideal divisor} ${\mathcal P}$ (\hbox{w.r.t.} $\underline T$), \hbox{i.e.} $P(\underline{t},\underline{Y}) \equiv 0 \pmod{{\mathcal P}}$ for every $\underline{t} \in {\mathcal Z}^k$; see Example \ref{exa:examfixide}. We managed to overcome the difficulty by extending the UFD approach from \cite{BDKN22} to the general Krull situation (instead of using the Dedekind idea). 
For rings of integers of number fields, Corollary \ref{cor:main1} improves on \cite[Theorem 1.1]{BDKN22} where ${\mathcal Z}$ is assumed to be of class number $1$. 
%There is, for them, an alternative approach consisting in evaluating the density of the subset $H$ from Definition \ref{def:int-Hilb} of ``good'' specializations. See \cite{BD22} for  ${\mathcal Z}={\mathbb Z}$ and \cite{BD26} in general.
%for general rings of integers where techniques of geometry of numbers are used.
%to overcome the difficulty mentioned above.
%\vskip 1mm
%\noindent
%(b) 
There is a gap between the PID and the Dedekind cases, and, more generally, between the UFD and non UFD cases in Theorem \ref{thm:main1}. The traditional Dedekind idea to replace prime elements by prime ideals faces the following difficulty:
%Replacing prime elements by prime ideals is a natural idea to compensate the loss of Gauss's lemma when considering Dedekind domains. One however faces the following difficulty: 
a polynomial $P \in \mathcal{Z}[\underline{T},\underline{Y}]$ 
may have no fixed divisor $a\in {\mathcal Z}$ but still have a fixed \emph{prime 
ideal divisor} ${\mathcal P}$ \hbox{w.r.t.} $\underline T$ (\hbox{i.e.} 
$P(\underline{t},\underline{Y}) \equiv 0 \pmod{{\mathcal P}}$ 
for every $\underline{t} \in {\mathcal Z}^k$); see Example \ref{exa:examfixide}.
\end{rmk}

The Dedekind idea (as discussed above) led to partial results in \cite{BDKN22}. Here instead we identify two structural properties that are
key in the UFD case and manage to extend them in some form -- properties (NVA) and (SF) of \S \ref{ss:stronSchpro} -- to the Krull case.
%\vskip 1mm
A central property of Krull rings established along the way is the so-called  \emph{coprime Schinzel Hypothesis} from \cite{BDKN22}, which is about coprime values of coprime polynomials (see Remark \ref{rmk:int-H-coprimeS}) and goes back to a fundamental result of Schinzel over $\Z$ \cite{schinzel2002}. It appears here as a more flexible and stronger variant, which we call being \emph{locally Schinzel} (Definition \ref{def:schinzel-ring}). The proof of Theorem \ref{thm:main1} can then be summarized as follows: for a Hilbertian ring ${\mathcal Z}$,
\vskip 1mm

\centerline{Krull $\Rightarrow$  (NVA) \& (SF) $\Rightarrow$  locally Schinzel $\Rightarrow$  integrally Hilbertian.}

\vskip 1mm
\noindent
We refer to Lemmas \ref{lem:interrest} and \ref{lem:chinremthe} for first implication (true up to some equivalence), to Proposition \ref{lem:mainaxiomatic} for the second one, and to Proposition \ref{prop:schinzel+Hil=intH} for the last one.\footnote{
In the special case of ring of integers of number fields, passing from ${\mathbb Z}$ to the general non PID case can be achieved in a different manner, based on an alternate approach which eva\-luates the density of the subset $H$ from Definition \ref{def:int-Hilb} of ``good'' specializations. See \cite{BD22} for the ``easy'' case ${\mathcal Z}={\mathbb Z}$ and \cite{BD26} for the general case for which more refined results from geometry of numbers are used.
} 

%This 
%property 
%guarantees that a {Hilbertian ring} be integrally Hilbertian (Proposition \ref{prop:schinzel+Hil=intH}), and is guaranteed by properties (NVA) and (SF) (Proposition \ref{lem:mainaxiomatic}). 
%Theorem \ref{thm-integ-Hilb} shows that \emph{Krull domains} sa\-tis\-fy   (NVA) and (SF), hence are locally Schinzel,
%\emph{are locally Schinzel rings}, 
%and so are integrally Hilbertian if they also are Hilbertian rings. 

\vskip 1mm

Polynomial rings $\mathcal{Z}={\mathcal R}[U]$ over a domain $\mathcal{R}$ are not Krull rings in general, and it is unclear whether they satisfy (NVA) and (SF). 
%It is not clear whether the polynomial ring ${\mathcal Z}[U]$ is a locally Schinzel ring in general. 
Yet we can show the following. 

\begin{thm} \label{thm:main1-2}
Let ${\mathcal R}$ be an arbitrary domain. Then the polynomial ring ${\mathcal R}[U]$ is locally Schinzel and integrally Hilbertian.
\end{thm}

Theorem \ref{thm:main1-2} is reminiscent of the analogous result that $K(U)$ is a Hilbertian field for any field $K$. When ${\mathcal R}$ is not integrally closed, which implies that ${\mathcal R}[U]$ is not integrally closed either, Theorem \ref{thm:main1-2} shows that implication ``\emph{integrally Hilbertian} $\Rightarrow$ \emph{integrally closed}'' does not hold.

\subsection{Schinzel type applications} Our applications relate to the famous conjectural 
\vskip 1,5mm

 \noindent
 {\bf Schinzel Hypothesis}.  \emph{If $P_1,\ldots,P_s\in \Z[T]$ are irreducible polynomials such that the product $\prod_{i=1}^s P_i$ has no fixed divisor in $\Z$} (the local condition), \emph{then there exist infinitely many integers $m\in \Z$ such that $P_1(m),\ldots,P_s(m)$ are simultaneously prime}. 
\vskip 1,5mm

We offer a generalizing definition in \S \ref{ssec:schinzel-rings}.
%The field ${\mathcal Q}$ is a priori assumed to be a \emph{field with a product formula}, relative to a given nonempty set $S$ of primes $\mathfrak{p}$ of ${\mathcal Q}$ (\S \ref{sssec:FDP}). 
%Definition is recalled in \S \ref{sssec:FDP}. For more on fields with a product formula, see \cite[\S 17.3]{FJ23}
%%%%%%%%%%%%%%%%%%%%%%%%%%%%%%%
%%%%%%%%%%%%%%%%%%%%%%%%%%%%%%%
%%%%%%%%%%%%%%%%%%%%%%%%%%%%%%%
%%%%%%%%%%%%%%%%%%%%%%%%%%%%%%%
%Our basic examples will be the field of rationals ${\mathbb Q}$ and the rational function fields $\kappa(\underline Y)$ in $n\geq 1$ variables over an arbitrary field $\kappa$. Other examples include finite extensions of fields with a product formula. 
%Associated to a product formula comes a natural height on ${\mathcal Q}$ -- the \emph{Weil height} -- which we denote by $H_S$. When ${\mathcal Q} = {\mathbb Q}$, $H_S$ is the usual absolute value. On the field $\kappa(\underline Y)$, each variable $Y_i$ induces a partial degree Weil height: $H_{i}(\cdot) = 2^{\deg_{u_i}(\cdot)}$, $i=1,\ldots, n$.
Roughly speaking, a \emph{Schinzel ring}\footnote{While the name ``Schinzel ring'' obviously comes from the motivating Schinzel Hypothesis, the name ``locally Schinzel ring'' that was previously introduced refers to the fact that, as we will see, the defining property of these rings is concerned with the \emph{local condition} assumed in the Schinzel Hypothesis.} is defined to be a domain ${\mathcal Z}$, equipped with some \emph{Weil heights}, for which, given polynomials $P_1,\ldots,P_s \in {\mathcal Z}[\underline T]$, irreducible in $\mathcal{Q}[\underline T]$, the \emph{local condition} -- absence of fixed divisors \hbox{w.r.t.} $\underline T$ for $P_1\cdots P_s$ -- is the only obstruction to producing irreducible specializations $P_1(\underline t), \ldots, P_s(\underline t)$ at points $\underline t\in {\mathcal Z}^k$ of arbitrarily large heights;
%, \hbox{w.r.t.} some given ``Weil heights'' on ${\mathcal Q}$ (assumed to be a field with a product formula); 
see Definition \ref{def:schinzel}.

%\begin{defn} \label{def:schinzel}
%Let ${\mathcal Z}$ be domain such that the fraction field ${\mathcal Q}$ is equipped with several sets $S_1,\ldots,S_n$ satisfying the product formula. Let $H_1,\ldots,H_n$ be the corresponding Weil heights. The ring  ${\mathcal Z}$ is called a \emph{Schinzel ring \hbox{w.r.t. the heights $H_1,\ldots,H_n$}} if the following holds. Let $\underline P$ be a finite set of polynomials $P_1,\ldots,P_s \in {\mathcal Z}[\underline T]$,  irreducible in $\mathcal{Q}[\underline T]$
%and such that the product $P_1\cdots P_s$ has no fixed divisor in ${\mathcal Z}$ \hbox{w.r.t.} $\underline T$.
%Let $A$ be a positive real number. Then the following subset of ${\mathcal Z}^k$ is Zariski-dense:
%\vskip 1mm

%\centerline{\hskip 6mm ${\mathcal S}_{\mathcal{Z}}(\underline P,A)
%= \left\{\underline t = (t_1,\ldots,t_k)\in {\mathcal Z}^k \hskip 1pt \left\vert\hskip 1pt \begin{array}{c}
%     P_1(\underline t), \ldots, P_s(\underline t) \ \hbox{\rm are irreducible in ${\mathcal Z}$}, \hbox{and} \hfill \hfill \\
%     H_i(t_j) \geq A,\ \hbox{for all } j=1,\ldots,k \ \hbox{and}\  i=1,\ldots,n. \hfill \hfill \\
%    \end{array}\right. \right\} 
%    $.}
%\end{defn}

%Obviously, for ${\mathcal Z}={\mathbb Z}$ and $k=1$, one can drop the height condition in the definition of ${\mathcal S}_{\mathcal{Z}}(\underline P,A)$. This is not the case for ${\mathcal Z} = \kappa[Y_1,\ldots,Y_n]$: containing polynomials of arbitrarily large degree is strictly stronger than being infinite.
%\vskip 1mm

The original Schinzel Hypothesis corresponds to the special case ${\mathcal Z}={\mathbb Z}$ and $k=1$; see Theorem \ref{thm:schinzel1=>Schinzelk} for the equivalence between one and several parameters. The polynomial ring ${\mathcal Z}[Y_1,\ldots,Y_n]$, equipped with the Weil heights induced by the partial degree $\deg_{Y_i}(\cdot)$ in each variable $Y_1,\ldots,Y_n$, %induces a Weil height $H_i$, 
is another natural test for the Schinzel ring notion.
%In fact we will show the following, which already appeared in \cite[Lemma 2.3]{KK25} for $s=1$ and ${\mathcal Z}={\mathbb Z}$.
 %\begin{thm}\label{thm:schinzel:1-equiv-k} The special case $k=1$ of Definition \ref{def:schinzel} is equivalent to the full case $k\geq 1$ if ${\mathcal Z}$ is a Krull domain.
%\end{thm}
We obtain the following.

\begin{thm} \label{thm:schinzel-main}
Let ${\mathcal Z}$ be an integrally Hilbertian ring and $Y_1,\ldots,Y_n$ be $n\geq 1$ va\-riables. Then the polynomial ring ${\mathcal Z}[Y_1,\ldots,Y_n]$ is a Schinzel ring \hbox{w.r.t.} the partial degree Weil heights $H_1,\ldots,H_n$.   
\end{thm}

More explicitly, this rephrases as follows: 
\vskip 1mm

\noindent
(**) \emph{given $s\geq 1$ polynomials $P_1,\ldots,P_s \in {\mathcal Z}[\underline{Y}][\underline T]$,  irreducible in $\mathcal{Q}(\underline{Y})[\underline T]$
and such that the product $P_1\cdots P_s$ has no nonunit divisor in ${\mathcal Z}[\underline Y]$, the subset ${\mathcal S}\subset {\mathcal Z}[\underline Y]^k$ consisting of all $k$-tuples $(M_1(\underline Y),\ldots,M_k(\underline Y))$ 
of polynomials}
\vskip 0,5mm
%\noindent
- \emph{each of degree bigger than any prescribed constant $A>0$ in each variable $Y_1,\ldots, Y_n$ },
\vskip 0,5mm 
%\noindent
-  \emph{and such that  
$P_i(\underline Y, M_1(\underline Y),\ldots,M_k(\underline Y))$ is irreducible in ${\mathcal Z}[\underline Y]$ ($i=1,\ldots,s$)},
\vskip 0,5mm

\noindent
\emph{is Zariski-dense.}
\vskip 1mm

Theorem \ref{thm:main3} shows a more precise version: 
the polynomials $M_i(\underline Y)$ can be prescribed  to be of degree in $Y_j$ any given integer $d_{ij}$ ($i=1,\ldots,k$, $j=1,\ldots,n$), provided that these integers $d_{ij}$ are suitably large. Precise bounds will be given, and for fixed suitable $d_{ij}$, the set of tuples $(M_1(\underline Y),\ldots,M_k(\underline Y))$, with these degrees, will be shown to be Zariski-dense, \hbox{in some natural sense.}
\vskip 1mm

\begin{rmk} \label{rmk:schinzel}
(a) Theorem \ref{thm:schinzel-main} improves on \cite[Theorem 1.1]{BDN20b} where ${\mathcal Z}$ is a UFD. The classical consequences of the Schinzel Hypothesis: the Dirichlet Theorem, the Twin Prime Problem, etc., as well as an analog of the Goldbach Problem hold for the ring ${\mathcal Z}[\underline Y]$ (as stated in \cite{BDN20b}) under our more general assumption that ${\mathcal Z}$ is integrally Hilbertian. 
\vskip 1mm

%\noindent
%(c) The original version, over $\Z$,  of the Schinzel Hypothesis corresponds to the situation (not covered by Theorem \ref{thm:main2}) where  ${\mathcal Z}[\underline Y]$ is replaced by $\Z$. 
%\vskip 1mm

\noindent
(b) There has been a parallel activity on the ``polynomial Schinzel Hypothesis'' in the case that ${\mathcal Z}$ is a finite field (\cite{BenWi2005} \cite{Pol2008} \cite{BS2012} \cite{BSJa2012} \cite{Entin2016}). Specific ingredients for finite fields, \hbox{e.g.} the Lang-Weil estimates, are used. Our Hilbertian environment allows a different specialization approach, which, thanks to our \emph{integrally} Hilbertian property, can be performed \emph{over the ring}.
\end{rmk}

It follows from Theorem \ref{thm:main1-2} and Theorem \ref{thm:schinzel-main} that if ${\mathcal Z}$ is an arbitrary domain, then the polynomial ring ${\mathcal Z}[Y_0,Y_1,\ldots,Y_n]$ is a Schinzel ring \hbox{w.r.t.} the partial degree Weil heights $H_1,\ldots,H_n$. More explicit forms of these results in fact show the following, which improves on \cite[Theorem 1.2]{BDN20b} where  ${\mathcal Z}$ is a field.

\begin{cor} \label{cor:Z[Y_0,...,Y_n]}
Let ${\mathcal Z}$ be an arbitrary domain. Then the polynomial ring ${\mathcal Z}[Y_0, Y_1,\ldots,Y_n]$ with $n\geq 1$ is a Schinzel ring \hbox{w.r.t.} the partial degree Weil heights $H_0, H_1,\ldots,H_n$. 
\end{cor}

Obviously, the polynomial ring ${\mathcal Z}[Y_0]$ is not a Schinzel ring if ${\mathcal Z}$ is an algebraically closed field. So the assumption $n\geq 1$ cannot be relaxed.  Another (subtler) counterexample is the polynomial ring ${\mathbb F}_2[Y_0]$; from an example of Swan \cite[pp. 1102-1103]{swan62}, the polynomial $M(Y_0)^8+Y_0^3$ is reducible for all polynomials $M\in {\mathbb F}_2[Y_0]$. 
\vskip 1mm

Compared to \cite{BDN20b}, beside the bigger generality of Theorem \ref{thm:schinzel-main}, 
%we also have better bounds on the integers $d_{ij}$. We further obtain 
we have the additional conclusion that the local condition can be preserved by specialization. 
%See Theorem \ref{thm:hypschistronggen} for a more precise statement. 
\vskip 1mm

%\begin{thm}\label{thm:schinzel-main}%{thm:hypschistronggen1} 

\noindent
{\bf Theorem \ref{thm:schinzel-main} (continued).}} \emph{Retain assumptions and notation from {\rm (**)} above. Assume further that $\mathcal{Z}$ is 
%a near UFD or 
a Krull domain and that the product $P_1\cdots P_s$ has no fixed divisor \hbox{w.r.t.} the $(k+n)$-tuple $(\underline T,\underline{Y})$. 
%and that for  $i=1,\dots,k$, a \underbar{nonzero} tuple $\underline d_i = (d_{i1},\ldots, d_{in}) \in {\mathbb N}^n$ is given that is suitably large.
%, in some sense. 
Then the subset ${\mathcal S}_0$ of the set ${\mathcal S}$ from statement \hbox{\rm (**)} consisting of $k$-tuples $(M_1(\underline Y),\ldots,M_k(\underline Y)) \in {\mathcal S}$ satisfying the additional property that the polynomial $(P_1\cdots P_s)(\underline Y, M_1(\underline Y),\ldots,M_k(\underline Y))$ has no fixed divisor \hbox{w.r.t.} $\underline Y$,  is still Zariski-dense in ${\mathcal Z}[\underline Y]^k$.}
%\end{thm}
\vskip 2mm

Theorem \ref{thm:hypschistronggen} provides a more precise version, showing further that, \emph{for $k=n=1$} (one parameter, one variable), \emph{one can request that the polynomial $M(Y)$ be monic and of arbitrary degree $d\geq 1$}. A noteworthy consequence for the original Schinzel Hypothesis is the following.

\begin{cor} \label{cor:SH} Assume that the original Schinzel Hypothesis over $\Z$ is true. 
Let $P_1,\ldots,P_s \in \Z[T]$ be some irreducible polynomials such that the product $P_1\cdots P_s$ has no fixed divisor \hbox{w.r.t.} $T$. Then there exist monic polynomials $M\in \mathbb{Z}[T]$ of arbitrary degree
%\footnote{That the degree $d$ can be arbitrarily chosen in this special situation follows from Theorem \ref{thm:main3} and not from Theorem \ref{thm:hypschistronggen1} (which requests that $d$ be suitably big).} 
$d\geq 1$ such that for infinitely many $t \in {\mathbb N}$, the integers $P_1(M(t)),\ldots,P_s(M(t))$ are prime numbers.
\end{cor}

Namely, the special case $k=n=1$ of Theorem \ref{thm:hypschistronggen}, applied to the polynomials $P_1,\ldots,P_s \in \Z[T]\subset \Z[Y][T]$ (of degree $0$ in $\underline Y$) provides a monic polynomial $M(T)\in {\mathbb Z}[T]$ of arbitrary degree $d\geq 1$ such that
the polynomials $P_1(M(T)),\ldots,P_s(M(T))$ satisfy the assumptions of the Schinzel Hypothesis. Applying the Schinzel Hypothesis yields the required conclusion.
\vskip 1mm

Corollary \ref{cor:SH} shows in particular that, to prove the Schinzel Hypothesis, one may restrict to  polynomials of arbitrarily large degree — a somewhat striking reduction.
\vskip 2mm

%Theorem \ref{thm:main2} below shows that 
%\vskip 1mm

%\noindent
%(*) \emph{if ${\mathcal Z}$ is an integrally Hilbertian ring, then the polynomial ring ${\mathcal Z}[\underline Y]$ is a Schinzel ring \hbox{w.r.t.} the partial degree Weil heights $H_1,\ldots,H_r$}.}
%\vskip 1mm

The rest of the paper is organized as follows.
\vskip 1mm

In Section \ref{sec:presen}, we introduce the general setup. Definitions  of integrally Hilbertian rings and locally Schinzel rings are given and related (Proposition \ref{prop:schinzel+Hil=intH}).
We then state Theorem \ref{thm-integ-Hilb}, which generalizes Theorem \ref{thm:main1}.
\vskip 1mm

In Section \ref{sec:proof-main-lemma}, we prove Theorem \ref{thm-integ-Hilb}. The proof is reduced to that of Lemma \ref{lem:main} that asserts that Krull rings (and other rings) are locally Schinzel. We then introduce the previously mentioned conditions (NVA) and (SF), and use them to prove Lemma \ref{lem:main}.
%, using these propertieswhich consists in showing that Krull rings (and other rings) satisfy conditions (NVA) and (SF). 
%We identify two general conditions for a ring ${\mathcal Z}$ to be locally Schinzel (Definition \ref{def:schproide}) that lead to a proof of Theorem \ref{thm-integ-Hilb} (and of Theorem \ref{thm:main1}).
\vskip 1mm

Section \ref{sec:poly-rings} is devoted to the proof of Theorem \ref{thm:main1-2} that polynomial rings are \hbox{integrally Hilbertian.}
\vskip 1mm

%\label{ss:thm13}
%{sec:applications}
In Section \ref{sec:applications}, we define Schinzel rings and discuss our main results in this context, mainly about polynomial rings being Schinzel rings.
\vskip 1mm

%Section \ref{ss:thm13} is devoted of Theorem \ref{thm:hypschistronggen} generalizing Theorem \ref{thm:schinzel-main} (continued) and to the proof of Theorem \ref{thm:schinzel:1-equiv-k}.
%\vskip 1mm

Section \ref{sec:proofs} provides the proofs of the results of Section \ref{sec:applications}, including that of Theorem \ref{thm:schinzel-main} and its corollaries.

%Section \ref{ssec:schinzel-rings} aims to introduce Schinzel rings (Definition \ref{def:schinzel}). This notion recasts in a common environment the celebrated Schinzel Hypothesis (thanks to Theorem \ref{thm:schinzel1=>Schinzelk}) and our results on polynomial rings. These are Theorem \ref{thm:main3} and Theorem \ref{thm:hypschistronggen}, which are more precise forms of Theorem \ref{thm:schinzel-main}.
%They are stated and commented on in \S \ref{SchH-general} and in \S \ref{ssec:hypschistronggen}. Their proofs are postponed to Section \ref{sec:proofs}.

%Section \ref{sec:applications} is devoted Theorem  to our Schinzel type applications: 
%Theorem \ref{thm:main2} and the more general 
%Theorem \ref{thm:main3}, Theorem \ref{thm:hypschistronggen} (respectively generalizing Theorem \ref{thm:schinzel-main}, Theorem \ref{thm:schinzel-main}) and Theorem \ref{thm:schinzel1=>Schinzelk}.
\vskip 2mm

\noindent
{\bf Acknowledgements.} The first two authors acknowledge the support of the CDP C2EMPI, as well as the French State under the France-2030 programme, the University of Lille, the Initiative of Excellence of the University of Lille, the European Metropolis of Lille for their funding and support of the R-CDP-24-004-C2EMPI project. The third author was supported by the National Research Foundation of Korea (NRF Basic Research Grant RS-2023-00239917).

\vskip 2mm

Fix for the whole paper a domain ${\mathcal Z}$ and denote its fraction field by ${\mathcal Q}$. Also fix two tuples of indeterminates: $\underline{T}=(T_1,\dots,T_k)$ ($k \geq 0$), and $\underline{Y}=(Y_1,\dots,Y_n)$ ($n \geq 0$), with $k=0$ or $n=0$ meaning that the corresponding tuple is empty. We call the $T_i$ the \emph{parameters}; they are to be specialized, unlike the $Y_i$ which we call the \emph{variables}.

\section{Locally Schinzel rings and integrally Hilbertian rings}\label{sec:presen}

After some preliminary definitions in \S \ref{ssec:preliminaries}, integrally Hilbertian rings are introduced in \S \ref{ssec:int-Hilb-ring}, and locally Schinzel rings in \S \ref{ssec:Schinzel_ring}. In \S \ref{ssex:main_results)}, we state and comment on Theorem \ref{thm-integ-Hilb}, which gen\-eralizes
Theorem \ref{thm:main1}. \S \ref{ssec:counter-examples} provides examples of Hilbertian but \hbox{not integrally Hilbertian rings.}
%and its proof is reduced to that of Lemma \ref{lem:main} in \S \ref{ssec:reduction}. 

\subsection{Preliminaries} \label{ssec:preliminaries}

\subsubsection{Ring theory}\label{ssec:krulldom}

Our main criterion (Theorem \ref{thm-integ-Hilb}) for a domain ${\mathcal Z}$ to be locally Schinzel or integrally Hilbertian includes ${\mathcal Z}$ being either a \emph{Krull domain} or a \emph{near UFD}. We recall below basic facts about these classes of rings. 
\vskip 0,5mm

Given a domain ${\mathcal Z}$, a proper ideal $\mathfrak{a}\subset \mathcal{Z}$ is \emph{primary} if for every $x,y \in \mathcal{Z}$, $x y \in \mathfrak{a}$ implies $x \in \mathfrak{a}$ or $y^m$ for some integer $m \geq 1$; and set 
$$
{\rm Spec}^{(1)}(\mathcal Z)=\{ \mathfrak{p} \in {\rm Spec} \hskip 1pt \mathcal Z \mid {\rm ht}(\mathfrak{p})=1\},
$$
where ${\rm ht}(\cdot)$ denotes the height.

\begin{defn}\label{def:krulring}
A domain $\mathcal Z$ is called a \emph{Krull domain} if the following conditions hold:
\begin{enumerate}
\item $A_{\mathfrak{p}}$ is a discrete valuation ring for all $\mathfrak{p} \in \s1(\mathcal Z)$; 
\item Every non-zero principal ideal $\langle a  \rangle$ of $\mathcal Z$ is the intersection of finitely many height-one primary ideals.
\end{enumerate}
\end{defn}

\begin{rmk} \label{rmk:Krull-domain}
(a) Every UFD is a Krull domain: merely note that every height-one prime ideal in a UFD is generated by a prime element.
\vskip 1mm

\noindent
(b) Every integrally closed Noetherian domain is a Krull domain. This follows from \cite[\S 17.H, Theorem 37, Theorem 38]{Mat80}.
%\textcolor{red}{[TODO: more details, check references, right edition?]}
\vskip 1mm

\noindent
(c) It is clear from the definition that being a Krull domain is preserved by localization (by any multiplicative subset). From the Mort--Nagata Integral Closure Theorem \cite[\S 41.B]{Mat80}, so it is by taking integral closure in finite extensions of the fraction field ${\mathcal Q}$. 
\end{rmk}

Assume that $\mathcal Z$ is a Krull domain. For each $\mathfrak{p} \in \s1(\mathcal Z)$, denote by $v_{\mathfrak{p}}$ the corresponding normalized valuation on $\mathcal{Z}_{\mathfrak{p}}$, which extends uniquely to ${\rm Frac}(\mathcal Z )$. For $n \in \N$, let 
$$
\mathfrak{p}^{(n)}=\mathfrak{p}^n\mathcal{Z}_{\mathfrak{p}} \cap \mathcal{Z}
$$ be the \emph{$n$-th symbolic power} of $\mathfrak{p}$.

%\begin{exa}\label{ex:krullring}
%\begin{itemize}
%\item Every integrally closed Noetherian domain is a Krull domain. This follows from \cite[Theorem 37, Theorem 38]{Mat80}.
%\item Every UFD is a Krull domain\footnote{Every height-one prime ideal in a UFD is generated by a prime element.}.
%\end{itemize}
%\end{exa}

\begin{lem}[Weak Approximation Theorem]\label{lem:weakappro}
Let ${\mathcal Z}$ be a Krull domain. Let $\mathcal{S}=\{\mathfrak{p}_1,\dots,\mathfrak{p}_r\} \subset \s1(\mathcal{Z})$ be a finite set, let $x_1,\dots,x_r \in \mathcal{Z}$, and let $n_1,\dots,n_r \in \N$. Then there exists $x \in \mathcal{Z}$ such that
$$
v_{\mathfrak{p}_i}(x-x_i)=n_i\quad \textrm{for all } 1 \leq i \leq r.
$$ 
\end{lem}

\begin{proof}
This result should be standard, but we include a proof due to the lack of a direct reference. By \cite[Page 289, V)]{Mat80}, the valuations $v_{\mathfrak{p}_1},\dots,v_{\mathfrak{p}_r}$ are independent. Classically then, an element $x$ satisfying the required conclusion can be chosen in $\mathcal{Q}$.  By \cite[Proposition 4.4]{Jar91} and the construction given in its proofs, one can indeed choose $x$ to lie in $\mathcal{Z}$. 
\end{proof}

See \cite[Section 41, Chapter 13]{Mat80} for more on Krull domains.

\begin{defn} \label{def:near-UFD}
A \emph{near UFD} is a domain such that 
\vskip 0,8mm

\noindent
(\ref{def:near-UFD}-1) \emph{every nonzero element has finitely many prime divisors (modulo units), and} 
\vskip 1mm

\noindent
(\ref{def:near-UFD}-2) \emph{every nonunit has at least one prime divisor} (in particular irreducibles are prime).
\vskip 1mm

\noindent
We also call \emph{PDF ring} (for ``Prime Divisor Finite'') a domain
that solely satisfies (\ref{def:near-UFD}-1).
\end{defn}

For more on near UFDs, see \cite[\S 2.3]{BDKN22} where this definition is introduced. 

%UFDs are near UFDs. The converse holds for Noetherian rings and Krull domains\footnote{Indeed, both rings satisfy the ascending chain conditions for principal ideals. This implies that any nonzero element is a product of prime elements.} but is not true in general.  
%\vskip 1mm

 \begin{rmk}\label{rmk:finmanydiv}
(a) Krull domains, Noetherian rings and near UFDs are PDF rings. It is obvious for near UFDs. For the first two, note first that both satisfy the Ascending Chain Condition on Principal ideals (ACCP) (obvious for Noetherian, whereas the Krull property is shown in \cite[Proposition 2.2]{AAZ1990} to imply the so-called ``bounded factorization'' property, which in turn implies ACCP).  Secondly, ACCP implies PDF: namely, ACCP easily gives that each nonzero element $a$ is a product of finitely many irreducibles; any prime divisor of $a$ then must divide one of these irreducibles.

%\noindent
%\textcolor{red}{Alternate exposition: A Krull domain $\mathcal{Z}$ is a PDF. Indeed, let $a$ be a nonzero element. We claim that, up to units, $a$ has only finitely many divisors. To see this, for any divisor $b|a$, we have $v_{\mathfrak{p}}(b) \leq v_{\mathfrak{p}}(a)$ for all $\mathfrak{p} \in {\rm Spec}^{(1)}\mathcal{Z}$. Since $v_{\mathfrak{p}}(a) = 0$ for all but finitely many $\mathfrak{p}$, there are only finitely many possible tuples 
%$
%\bigl(v_{\mathfrak{p}}(b)\bigr)_{\mathfrak{p} \in {\rm Spec}^{(1)}(\mathcal{Z})}
%%$. 
%By \cite[Page 289, III)]{Mat80}, there are finitely many $b$ up to units.} 
\vskip 1mm

\noindent
(b) Clearly, UFDs are near UFDs. Conversely near UFDs satisfying ACCP must be UFD. 
\vskip 1mm

\noindent
(c) \emph{The two classes of Krull domains and near UFDs are not contained in each other}. Indeed it follows from (b) that rings that are both Krull and near UFDs are UFD. As there are Krull domains that are not UFD (every ring of integers of a number field of class number $>1$), and near UFDs that are not UFD (\hbox{e.g.} \cite[Example 2.4]{BDKN22}), 
the claim follows. (Similarly the two classes of Noetherian rings and near UFDs are not contained in each other; and the same is true for Noetherian rings and Krull domains).
\end{rmk}

%\begin{defn} \label{def:near-UFD}
%A \emph{near UFD} is a domain such that 
%\vskip 0,8mm

%\noindent
%(\ref{def:near-UFD}-1) \emph{every nonzero element has finitely many prime divisors (modulo units), and} 
%\vskip 1mm

%\noindent
%(\ref{def:near-UFD}-2) \emph{every nonunit has at least one prime divisor}.
%\end{defn}

%UFDs are near UFDs. The converse holds for Noetherian rings and Krull domains\footnote{Indeed, both rings satisfy the ascending chain conditions for principal ideals. This implies that any nonzero element is a product of prime elements.} but is not true in general. A noteworthy property is that irreducible elements of a near UFD are prime. For more on near UFDs, see \cite[\S 2.3]{BDKN22} where this definition is introduced. 
%\vskip 1mm

%We also call PDF \emph{ring} (for ``Prime Divisor Finite'') a domain
%that solely satisfies (\ref{def:near-UFD}-1). 

%\begin{rmk}\label{rmk:finmanydiv}
%\textcolor{red}{A Krull domain $\mathcal{Z}$ is a PDF. Indeed, let $a$ be a nonzero element. We claim that, up to units, $a$ has only finitely many divisors. To see this, for any divisor $b|a$, we have $v_{\mathfrak{p}}(b) \leq v_{\mathfrak{p}}(a)$ for all $\mathfrak{p} \in {\rm Spec}^{(1)}\mathcal{Z}$. Since $v_{\mathfrak{p}}(a) = 0$ for all but finitely many $\mathfrak{p}$, there are only finitely many possible tuples 
%$
%\bigl(v_{\mathfrak{p}}(b)\bigr)_{\mathfrak{p} \in {\rm Spec}^{(1)}(\mathcal{Z})}
%$. 
%By \cite[Page 289, III)]{Mat80}, there are finitely many $b$ up to units.}  
%\end{rmk}

\subsubsection{Fixed divisor}

\begin{defn}
Let $P \in {\mathcal Z}[\underline{T},\underline{Y}]$ be a polynomial. A proper ideal ${\frak p}\subset {\mathcal Z}$ is called a \emph{fixed divisor} of $P$ \hbox{w.r.t.} $\underline{T}$ if
$P(\underline{t},\underline{Y}) \equiv 0 \pmod{{\frak p}}$
for every $\underline{t} \in {\mathcal Z}^k$.
The set of all fixed divisors of $P$ is denoted by $\mathcal{F}_{\underline{T}}(P)$. 
\end{defn}

When $k=0$, a fixed divisor ${\frak p}$ is merely an (ideal) divisor of $P(\underline{Y})$ in ${\mathcal Z}$ (in the sense that $P(\underline{Y}) \equiv 0 \pmod{{\frak p}}$). For clarity, we use a different notation, ${\rm Div}_{\mathcal{Z}}(P)$, for the set of all {proper} \emph{ideal divisors} of $P$. Also note that for $P\in {\mathcal Z}[\underline T,\underline Y]$, we have
${\rm Div}_{\mathcal{Z}}(P) \subset \mathcal{F}_{\underline{T}}(P)$.

\begin{lem} \label{lem:fixed-divisors-removal}
Assume that $\mathcal{Z}$ is a Krull domain or a near UFD. Let $
P(\underline{T},\underline{Y}) \in \mathcal{Z}[\underline{T},\underline{Y}]$ be a nonzero polynomial. Then there exists a nonzero element $\varphi \in \mathcal{Z}$ such that $P$ has no fixed  divisors \hbox{w.r.t.} $\underline{T}$ among the proper principal ideals of $\mathcal{Z}[1/\varphi]$.
\end{lem}

For near UFDs, the result is proved in \cite[\S 4.3.1]{BDKN22}. The case of Krull domains is handled in Remark \ref{rmk:fixed-div-finite}.

\begin{rmk}[How to find fixed divisors]\label{rmk:no-fixed-divisor} 
%Fixed divisors can be concretely be found as follows.
%This assumption can concretely be checked or guaranteed in several ways. 
Observe first that a nonunit $a\in {\mathcal Z}$ is a fixed divisor of some polynomial $P\in {\mathcal Z}[\underline T,\underline Y]$ \hbox{w.r.t.} ${\underline T}$ if and only if each coefficient $c(\underline T)$ of $P$ (viewed as a polynomial in ${\underline Y}$) does, \hbox{i.e.} satisfies $c(\underline t) \equiv 0 \pmod{a}$ for every ${\underline t}\in {\mathcal Z}^k$. Then, for example with ${\mathcal Z} = {\mathbb Z}$, 
one can reduce to the case $a=p$ is a prime; and the previous condition holds if and only if $c(\underline T)$ lies in the ideal $\langle p, T_1^p-T_1,\ldots,T_k^p-T_k \rangle \subset {\mathcal Z}[\underline T]$ \cite[Lemme 1]{terjanian66}. In particular, such primes $p$ can be bounded in terms of the partial degrees of $P$ \cite[\S 3.1]{BDKN22}. Evaluation of $c(\underline{T})$ at some points $\underline t\in {\mathcal Z}^k$ may then, in practice, yield the no fixed divisor assumption; an obvious situation is when $c(\underline{T})$ is a unit in ${\mathcal Z}$. 
\end{rmk}

\subsubsection{Primality type} \label{ssec:primality-typ}
In some situations, 
%notably when ${\mathcal Z}$ is a Krull domain (for example, when $\mathcal Z$ is a Dedekind domain), 
there may be some interest to obtain that a specialized polynomial $P(\underline t,\underline Y)$, instead of having ``no nonunit divisor in ${\mathcal Z}$'', has ``no \emph{prime ideal divisor} ${\frak p}\subset {\mathcal Z}$'' -- \hbox{i.e.}, that \emph{$P(\underline t,\underline Y)\not\equiv 0$ modulo every prime ideal ${\frak p}\subset {\mathcal Z}$}. But then one needs to assume that $P$ has no ``\emph{fixed prime ideal divisor}'' -- \hbox{i.e.}, that \emph{for no prime ideal ${\frak p}\subset {\mathcal Z}$, we have 
$P(\underline{t},\underline{Y}) \equiv 0 \pmod{{\frak p}}$ for every $\underline{t} \in {\mathcal Z}^k$}. This creates a variant of the notion of integrally Hilbertian rings. Other possibly valuable variants appear from other choices of sets of ideals. 
\vskip 1mm

Our main variants will be for the following sets of ideals. Our original variant is the first one.

\begin{itemize}
\item ${\rm Nonunit}\hskip 1pt {\mathcal Z}$: the set of all \emph{nonzero principal proper} ideals $\langle a\rangle \subset {\mathcal Z}$ (generated by nonunits $a$, $a\not=0$).
\item ${\rm Spec}^\ast\hskip 1pt {\mathcal Z}$: the set of all \emph{nonzero prime} ideals of ${\mathcal Z}$.
\item ${\rm Irred}\hskip 1pt {\mathcal Z}$: the set of all \emph{principal} ideals $\langle a\rangle$ generated by \emph{irreducible} elements $a\in {\mathcal Z}$.
\item ${\rm Prime}\hskip 1pt {\mathcal Z}$: the set of all \emph{nonzero principal prime} ideals $\langle a\rangle\subset {\mathcal Z}$ (generated by prime elements $a$).
\end{itemize}

\noindent
When there is no risk of confusion on the ring ${\mathcal Z}$, we omit ${\mathcal Z}$ in the notation; for example, we write ${\rm Prime}$ for ${\rm Prime}\hskip 1pt {\mathcal Z}$.
Also, as is usual, we often identify principal ideals $\langle a\rangle$ with elements $a\in {\mathcal Z}$ modulo units of ${\mathcal Z}$. 

\begin{rmk} \label{rem:PID-context}
When ${\mathcal Z}$ is a PID, the three sets ${\rm Spec}^\ast$,
%${\rm Spec}\hskip 1pt {\mathcal Z}\setminus\{0\}$, 
${\rm Irred}$ and ${\rm Prime}$ are equal and every ideal in ${\rm Nonunit}$ uniquely factors as a nonempty product of ideals in ${\rm Prime}$. 
\end{rmk}

\begin{exa}\label{exa:examfixide} Here is an example where ${\mathcal Z}$ is the ring of integers of a number field and $P(T,Y)\in {\mathcal Z}[T,Y]$ is a polynomial that has no fixed divisor in ${\mathcal P}= {\rm Nonunit}$ but has some in ${\mathcal P}= {\rm Spec}^\ast$. Take $\mathcal Z=\Z[\sqrt{-6}]$. It is the ring of integers of $\Q(\sqrt{-6})$. Consider the polynomial
\vskip 1mm

\centerline{$P(T,Y)=(T^2-T)Y+(T^2-T+2) \in \mathcal Z[T,Y]
$.}
\vskip 1mm

\noindent
Clearly
$
(2)= \mathfrak{p}^2
$ for $\mathfrak{p}=(2,\sqrt{-6})$. Since the norm of $\mathfrak{p}$ is $2$, $\mathfrak{p}$ is a fixed prime ideal divisor of $P$ \hbox{w.r.t.} $T$. Assume on the contrary that $\pi$ is a nonunit fixed divisor of $P$ \hbox{w.r.t.} $T$. Then $\pi$ divides $(-1)^2-(-1)=2$, and so $\pi$ is associate with $2$, as $2$ is irreducible. However $2$ does not divide
\vskip 1mm

\centerline{$(1+\sqrt{-6})^2-(1+\sqrt{-6})=-6+\sqrt{-6},
$}
\vskip 1mm

\noindent
a contradiction.
\end{exa}

More generally, we call \emph{primality type} any subset ${\mathcal P}$ of proper ideals of ${\mathcal Z}$. In the core Sections \ref{sec:presen} and  \ref{sec:proof-main-lemma}, our definitions and results are stated \hbox{w.r.t.} a given primality type ${\mathcal P}$.
The primality type ${\mathcal P} = {\rm Nonunit}$ remains the main case of interest. It is also ${\mathcal P} = {\rm Nonunit}$ that is tacitly meant when there is no explicit reference to a primality type (as in Sections \ref{sec:intro}, \ref{sec:poly-rings}, \ref{sec:applications}).

\subsection{Integrally Hilbertian ring} \label{ssec:int-Hilb-ring} Fix an integral domain ${\mathcal Z}$ and a primality type ${\mathcal P}$. Given $s$ polynomials $P_1,\ldots,P_s \in {\mathcal Z}[\underline{T},\underline{Y}]$, consider the following conditions, which serve both in the hypotheses and in the conclusions of our main results. 
\vskip 2mm

\noindent
(Irred$/{\mathcal Q}(\underline T)$) \hskip 10mm The polynomials $P_1,\ldots,P_s$ are irreducible in ${\mathcal Q}(\underline T)[\underline Y]$. 
\vskip 1mm

\noindent
(Prim$/{\mathcal Q}[\underline T]$) \hskip 11mm The polynomials $P_1,\ldots,P_s$ are primitive \hbox{w.r.t.} ${\mathcal Q}[\underline T]$.\footnote{For a UFD ${\mathcal Z}$, a polynomial $P\in {\mathcal Z}[\underline Y]$ is \emph{primitive} \hbox{w.r.t.} ${\mathcal Z}$ if its coefficients are coprime, \hbox{i.e.}, if they have no nonunit common divisor.}
\vskip 1mm

\noindent
$\hbox{\rm (NoFixDiv}/{\mathcal Z}[\underline T])_{\mathcal P}$ \hskip 0,5mm
The product $P_1\cdots P_s$ has no fixed divisor in ${\mathcal P}{\mathcal Z}$ \hbox{w.r.t.} ${\underline T}$. 
\vskip 2mm

The three depend on the $k$-tuple ${\underline T}$ of parameters. The third condition $\hbox{\rm (NoFixDiv}/{\mathcal Z}[\underline T])_{\mathcal P}$ also depends on the primality type ${\mathcal P}$. The answers to our main questions: is a ring ${\mathcal Z}$ integrally Hilbertian? a Schinzel ring? as subsequently defined, depend on ${\mathcal P}$.

Condition (Irred$/{\mathcal Q}(\underline T)$) includes $\deg_{\underline Y}(P_i)\geq 1$, $i=1,\ldots,s$. Conditions (Irred$/{\mathcal Q}(\underline T)$) and (Prim$/{\mathcal Q}[\underline T]$) joined together are equivalent to the polynomials $P_1,\ldots,P_s$ being irreducible in ${\mathcal Q}[\underline T,\underline Y]$ and of degree $\geq 1$ in ${\underline Y}$. When $k=0$, the second condition (Prim$/{\mathcal Q}$) merely means that 
the polynomials $P_1,\ldots,P_s\in {\mathcal Q}[\underline Y]$ are nonzero.

%They are summarized in Theorem \ref{thm-integ-Hilb}.

%(relative to the given primality type) are summarized in Theorem \ref{thm-integ-Hilb}. They obviously depend on the hypotheses on the ring ${\mathcal Z}$, but interestingly enough also on the primality type. 

\begin{defn} \label{def:integrally-hilbertian}
The ring ${\mathcal Z}$ is said to be \emph{integrally Hilbertian \hbox{w.r.t.} the primality type ${\mathcal P}$} if for any integers $k,n,s \geq 1$ and any polynomials $P_1,\ldots,P_s \in {\mathcal Z}[\underline{T},\underline{Y}]$ satisfying the three hypotheses (Irred$/{\mathcal Q}(\underline T)$), (Prim$/{\mathcal Q}[\underline T]$) and $\hbox{\rm (NoFixDiv}/{\mathcal Z}[\underline T])_{\mathcal P}$, the subset
  \vskip 2mm

\noindent
(\ref{def:integrally-hilbertian})  \hskip 3mm ${\mathcal H}_{\mathcal{P}}(P_1,\ldots,P_s)
= \left\{\underline m\in {\mathcal Z}^k \hskip 1pt \left\vert\hskip 1pt \begin{array}{c}
     \hbox{\rm the polynomials $P_1(\underline m,\underline Y),\ldots,P_s(\underline m,\underline Y)\in {\mathcal Z}[\underline Y]$ } \\
    \hskip 1mm \hbox{\rm satisfy conditions (Irred$/{\mathcal Q}$) and $\hbox{\rm (NoFixDiv}/{\mathcal Z})_{\mathcal P}$.} \\
    \end{array}\right. \right\} 
    \footnote{The polynomials $P_1(\underline m,\underline Y),\ldots,P_s(\underline m,\underline Y)$ are in ${\mathcal Z}[\underline Y]$; they have no more parameters $T_i$. As specified above, condition $\hbox{\rm (NoFixDiv}/{\mathcal Z})_{\mathcal P}$ means that they have no ideal divisor in ${\mathcal P}$.}$

\vskip 2mm
\noindent
is Zariski-dense in $\A^k({\mathcal Q})$.
\end{defn}

\begin{rmk}[Irreducibility in \hbox{${\mathcal Z}[\underline Y]$} of the specialized polynomials] \label{remiiredZvsQ1}
If a $k$-tuple $\underline m$ is in the set ${\mathcal H}_{\mathcal{P}}(P_1,\ldots,P_s)$ and the primality type  is ${\mathcal P} = {\rm Nonunit}$, then the specialized polynomials are irreducible in ${\mathcal Z}[\underline Y]$: indeed, from (Irred$/{\mathcal Q}$), they cannot have a nontrivial divisor in ${\mathcal Z}[\underline Y]$ of positive degree, and due to $\hbox{\rm (NoFixDiv}/{\mathcal Z})_{\mathcal P}$, they cannot have a nonunit divisor $a\in{\mathcal Z}$. The same conclusion holds with ${\mathcal P} = {\rm Irred}$ if ${\mathcal Z}$ additionally assumed to be Noetherian (indeed, in a Noetherian ring, every nonunit is divisible by an irreducible). It does too for ${\mathcal P} = {\rm Prime}$ if ${\mathcal Z}$ is a near UFD.
The same question is not clear for ${\mathcal P} = {\rm Spec}^\ast$. If ${\mathcal Z}$ is a UFD, then irreducibility over ${\mathcal Z}$ is \emph{equivalent} to irreducibility over ${\mathcal Q}$ joint with nonexistence of nonunit divisors in ${\mathcal Z}$.
\end{rmk}

The definition of an integrally Hilbertian ring given in Definition \ref{def:int-Hilb} corresponds to the situation of Definition \ref{def:integrally-hilbertian} that the primality type is ${\mathcal P}={\rm Nonunit}$, and in turn, this original definition corresponds to the condition from \cite{BDKN22} that the so-called \emph{Hilbert--Schinzel specialization property} holds for any integers $k,n,s \geq 1$.
\vskip 1mm

Proposition \ref{intH=>Hring} and Proposition \ref{prop:schinzel+Hil=intH} relate the notions of \emph{integrally Hilbertian rings} and of \emph{Hilbertian rings}.
The preliminary Remark \ref{remiiredZvsQ}(a) is aimed at dissipating some  possible confusion due to some subtelty in the literature. Remark \ref{remiiredZvsQ}(b) provides some details on a ``well-known to experts'' property of Hilbertian rings used in Section \ref{sec:intro}.

\begin{rmk}[Complements on Hilbertian rings and Hilbertian fields] \label{remiiredZvsQ} \hskip 1mm
\vskip 0,5mm

\noindent
(a) 
%In positive characteristic $p>0$, a slightly different variant of our definition of Hilbertian fields exists: 
\cite[\S 13.1]{FJ23} defines the Hilbert specialization property satisfied by Hilbertian fields and Hilbertian rings as the special case of our definition (Definition \ref{def:Hilb-ring-field}) for which $n=1$ and $P_1,\ldots,P_s$ are separable as polynomials in $Y$. Our a priori stronger version of the property is equivalent under the \emph{imperfectness assumption} that ${\mathcal Q}$ is of characteristic $0$ or imperfect; for Hilbertian fields, we refer to \cite[Prop.13.4.3]{FJ23}, and for Hilbertian rings to \cite[Prop.4.2]{BDN20b}.
%, \hbox{i.e.} ${\mathcal Z}\not={\mathcal Z}^p$ if $p={\rm char}({\mathcal Z})>0$ \cite[Prop.13.4.3]{FJ23}. 
For simplicity, we incorporated this assumption in our definition of Hilbertian rings and fields, thus avoiding any risk of confusion. There is no loss in doing so: Proposition \ref{intH=>Hring} shows  that integrally Hilbertian rings satisfy the imperfectness assumption.
\vskip 1mm

\noindent
(b) As asserted in Section \ref{sec:intro}, \emph{if ${\mathcal Z}$ is a Hilbertian ring, then the same holds for}
\vskip 0,5mm

\noindent
(b-1) \emph{any localization by some multiplicative subset $S\subset {\mathcal Z}$}, and,
\vskip 0,5mm

\noindent
(b-2) \emph{any integral closure ${\mathcal Z}^\prime_L$ of ${\mathcal Z}$ in a finite extension ${\mathcal L}/{\mathcal Q}$ of the fraction field.} 
\vskip 1mm

\noindent
Assertion (b-1) is easy since ${\rm Frac}(S^{-1}{\mathcal Z})={\mathcal Q}$ and ${\mathcal Z}\subset S^{-1}{\mathcal Z}$. 
Concerning (b-2), it is  clear if ${\mathcal L}/{\mathcal Q}$ is separable from \cite[Cor.13.2.3]{FJ23}. The case that ${\mathcal L}/{\mathcal Q}$ is inseparable can be deduced from an adjusted version of \cite[Prop.13.3.6(b)]{FJ23} that replaces the phrase ``\emph{Hilbert subsets of $K$ (\hbox{\rm resp.} of $L$) are non empty}'' by ``\emph{Hilbert subsets of $K$ (\hbox{\rm resp.} of $L$) contain elements of ${\mathcal Z}$ (\hbox{\rm resp.} of ${\mathcal Z}^\prime_L$)}'' in the assumption (\hbox{resp.} in the conclusion) of the statement. The main change in the proof is to observe that the element $\gamma\in {\mathcal L}$ from the proof of \cite[Lemma 13.3.5]{FJ23} can be found in ${\mathcal Z}^\prime_L$. Other places of the proof of \cite[Prop.13.3.6(b)]{FJ23} use elements of Hilbert sets of $K$. These should be picked in ${\mathcal Z}$, as our modified assumption allows. 
%\textcolor{red}{TODO: check or find a better reference}. 
\end{rmk}

\begin{prop} \label{intH=>Hring}
Let ${\mathcal Z}$ be an integrally Hilbertian ring \hbox{w.r.t.} a primality type ${\mathcal P}$. Then ${\mathcal Z}$ is of characteristic $0$ or imperfect, 
%\emph{(\hbox{i.e.} ${\mathcal Z}^p\neq{\mathcal Z}$ if $p={\rm char}({\mathcal Z})>0$)}, 
and is a Hilbertian ring.
\end{prop}

\begin{proof}
If $p={\rm char}({\mathcal Z})>0$, the polynomial $Y^p-T$ is irreducible in ${\mathcal Q}[T,Y]$ and has no fixed divisor \hbox{w.r.t.} ${\underline T}$ (with respect to any primality type). Hence, by the integrally Hilbertian property, there exists $t\in {\mathcal Z}$ such that $t\notin {\mathcal Q}^p$. Hence ${\mathcal Z}$ is imperfect.

To show that ${\mathcal Z}$ is a Hilbertian ring, let $P_1,\ldots,P_s \in \mathcal{Q}[\underline{T},Y]$ be $s\geq 1$ irreducible polynomials of degrees $\geq 1$ in the single variable $Y$, and separable in $Y$ (Remark \ref{remiiredZvsQ}(a) explains why reduction to such polynomials is legitimate). Furthermore, from \cite[Lemma 13.1.6]{FJ23}, one may assume that $s=1$. Set $P=P_1$ and let $d=\deg_Y(P)$. Let $c(\underline{T}) \in \mathcal{Z}[\underline{T}]$ denote the leading coefficient of $P$ when viewed as a polynomial in $Y$. Consider the polynomial 
$$
Q(\underline{T},Y)=c^{d-1}P(\underline{T},Y/c) \in \mathcal{Z}[\underline{T},Y].
$$ Clearly, $Q$ is monic, and hence has no fixed divisors w.r.t $\underline{T}$. Moreover, $Q$ remains irreducible in $\mathcal{Q}[\underline{T},Y]$. Since $\mathcal{Z}$ is integrally Hilbertian, there exists a Zariski-dense set of $\underline{m} \in \mathcal{Z}^k$ such that $Q(\underline{m},Y)$, and so $P(\underline{m},Y)$ too, is irreducible in $\mathcal{Q}[Y]$.    
\end{proof}

\subsection{Locally Schinzel ring} \label{ssec:Schinzel_ring} 
We introduce a new property: being a locally Schinzel ring (Definition \ref{def:schinzel-ring}) and show that this property guarantees that a Hilbertian ring
%\footnote{Definition of a Hilbertian ring is recalled at the beginning of Section \ref{sec:intro}.}  
be integrally Hilbertian (Proposition \ref{prop:schinzel+Hil=intH}). 
In the definition, we slightly abuse terminology  and call {\it arithmetic progression} of ${\mathcal Z}$ every subset of ${\mathcal Z}$ of the form $\omega {\mathcal Z} + \alpha$ with $\omega, \alpha \in {\mathcal Z}$, $\omega \not= 0$.

\begin{defn} \label{def:schinzel-ring}
The ring ${\mathcal Z}$ is said to be a \emph{locally Schinzel ring \hbox{w.r.t.} the primality type ${\mathcal P}$}  if the following conclusion holds for any integers $k,n,s \geq 1$ and any polynomials $P_1,\ldots,P_s \in {\mathcal Z}[\underline{T},\underline{Y}]$ satisfying the two hypotheses (Prim$/{\mathcal Q}[\underline T]$) and $\hbox{\rm (NoFixDiv}/{\mathcal Z}[\underline T])_{\mathcal P}$ (from Section \ref{ssec:int-Hilb-ring}):
  \vskip 1mm

\noindent
(\ref{def:schinzel-ring})  \hskip 1mm Letting $\underline{T}'=(T_2,\dots,T_k)$, there is an arithmetic progression $\tau=(\omega \ell + \alpha)_{\ell \in {\mathcal Z}} \subset {\mathcal Z}$ ($\omega,\alpha \in {\mathcal Z},\omega \neq 0$) such that, for all but finitely many
%\footnote{\textcolor{red}{JK: Of course one could write ``for all" via taking a coarser arithmetic progression, but I assume you deliberately want to avoid this?} \textcolor{blue}{You're right of course, but I have a slight preference for not doing it.}} 
$t_1 \in \tau$, the polynomials
$$
P_1(t_1,\underline{T}',\underline{Y}),\dots,P_s(t_1,\underline{T}',\underline{Y}) \in {\mathcal Z}[\underline{T}',\underline{Y}]
$$
satisfy (Prim$/{\mathcal Q}[\underline T^\prime]$) and $\hbox{\rm (NoFixDiv}/{\mathcal Z}[\underline T^\prime])_{\mathcal P}$.
\end{defn}

Theorem \ref{thm-integ-Hilb} gives concrete examples of locally Schinzel rings.

\begin{rmk} \label{rmk:int-H-coprimeS}
Being a locally Schinzel ring relates to the following property called \emph{coprime Schinzel Hypothesis} in \cite{BDKN22}: 
\vskip 1mm

\noindent
(*) \emph{For any $k\geq 1$, $\ell \geq 2$ and any nonzero polynomials $Q_1,\ldots,Q_\ell \in {\mathcal Z}[\underline{T}]$, coprime in ${\mathcal Q}[\underline T]$ and such that no nonunit $p\in {\mathcal Z}$ divides all values $Q_{1}(\underline{m}),\dots,Q_{\ell}(\underline{m})$ with $\underline m\in {\mathcal Z}^k$, there exists $\underline m= (m_1,\ldots,m_{k})\in {\mathcal Z}^{k}$ such that $Q_1(\underline m), \ldots, Q_\ell(\underline m)$ have no common divisor in ${\rm Nonunit}\hskip 1pt\mathcal{Z}$.}
\vskip 0,5mm

\vskip 1mm

\noindent
Locally Schinzel rings (\hbox{resp.} integrally Hilbertian rings) satisfy the coprime Schinzel Hypothesis: given $Q_1,\ldots,Q_\ell$ as above, apply the locally Schinzel property one parameter after another (\hbox{resp.} the integrally Hilbertian property all parameters at once) to the polynomial $Q_1Y_1\cdots + Q_\ell Y_\ell$.  
\end{rmk}

The point of the locally Schinzel property is that it links the Hilbertian ring and integrally Hilbertian ring notions. 

\begin{prop} \label{prop:schinzel+Hil=intH}
Let ${\mathcal Z}$ be a domain 
%\footnote{This imperfectness condition is necessary for  the ring ${\mathcal Z}$ to be a Hilbertian or to be integrally Hilbertian.}. 
and  ${\mathcal P}$ be a primality type. 
%${\rm Spec},{\rm Irred},{\rm Prime},{\rm Nonunit}$. 
If ${\mathcal Z}$ is both a locally Schinzel ring \hbox{w.r.t.} ${\mathcal P}$ and a Hilbertian ring, then it is integrally Hilbertian \hbox{w.r.t.} ${\mathcal P}$. 
\end{prop}

%\AB{AB: In this part, we also assume $D \neq D^p$ if ${\rm char}(p)>0$? If so, $D$ is automatically infinite.}

\begin{proof} Let $P_1,\ldots,P_s \in {\mathcal Z}[\underline{T},\underline{Y}]$ be $s$ polynomials satisfying the three conditions (Irred$/{\mathcal Q}(\underline T)$), (Prim$/{\mathcal Q}[\underline T]$) and $\hbox{\rm (NoFixDiv}/{\mathcal Z}[\underline T])_{\mathcal P}$. Set $\underline{T}'=(T_2,\dots,T_k)$. As ${\mathcal Z}$ is a locally Schinzel ring, there is an arithmetic progression $\tau=(\omega \ell + \alpha)_{\ell \in {\mathcal Z}} \subset {\mathcal Z}$ ($\omega,\alpha \in {\mathcal Z},\omega \neq 0$) such that, for all $\ell \in {\mathcal Z}$ but in a finite set $F$, the polynomials
$$
P_1(\omega \ell +\alpha,\underline{T}',\underline{Y}),\dots,P_s(\omega \ell +\alpha,\underline{T}',\underline{Y}) \in {\mathcal Z}[\underline{T}',\underline{Y}]
$$
satisfy (Prim$/{\mathcal Q}[\underline T^\prime]$) and $\hbox{\rm (NoFixDiv}/{\mathcal Z}[\underline T^\prime])_{\mathcal P}$. Consider then the polynomials 
$$
P_1(\omega T_1 +\alpha,\underline{T}',\underline{Y}),\dots,P_s(\omega T_1 +\alpha,\underline{T}',\underline{Y}) \in {\mathcal Z}[\underline{T},\underline{Y}].
$$

\noindent
They are irreducible in ${\mathcal Q}(\underline{T})[\underline{Y}]$. By \cite[Proposition 4.2]{BDN20b}, it follows from ${\mathcal Z}$ being a Hilbertian ring (including the imperfectness assumption) that there is an infinite subset $H\subset {\mathcal Z}\setminus F$, such that for each $\ell \in H$, the polynomials

$$
P_1(\omega \ell +\alpha,\underline{T}',\underline{Y}),\dots,P_s(\omega \ell +\alpha,\underline{T}',\underline{Y}) \in {\mathcal Z}[\underline{T}',\underline{Y}]
$$

\noindent
are irreducible in ${\mathcal Q}(\underline{T}')[\underline{Y}]$, \hbox{i.e.}, they satisfy condition (Irred$/{\mathcal Q}(\underline T^\prime)$). Thus they satisfy the three conditions (Irred$/{\mathcal Q}(\underline T^\prime)$), (Prim$/{\mathcal Q}[\underline T^\prime]$), $\hbox{\rm (NoFixDiv}/{\mathcal Z}[\underline T^\prime])_{\mathcal P}$ and the same argument can be applied to them. Repeating inductively the argument  to each of the parameters $T_2,\ldots,T_k$ leads to the requested condition that the set ${\mathcal H}_{\mathcal{P}}(P_1,\ldots,P_s)$ is Zariski-dense. \end{proof}

\subsection{Main results} \label{ssex:main_results)}
The following statement provides our main examples of locally Schinzel rings and of integrally Hilbertian rings. 

\begin{thm} \label{thm-integ-Hilb} 
{\rm (a)} For  a ring ${\mathcal Z}$ as in first row and a primality type ${\mathcal P}$ as in first column, we have the following answers to the question of whether ${\mathcal Z}$ is a locally Schinzel ring \hbox{w.r.t.} ${\mathcal P}$:
\vskip 3mm

\begin{center}
\begin{tabular}{ |c|c|c|c|c| } 
  \hline
   & {\rm Krull} &  {\rm near UFD} & {\rm PDF}\\ 
    \hline
   ${\rm Nonunit}$ & Yes & Yes  & No\\
  \hline
   ${\rm Irred}$&  Yes & Yes  & No\\ 
  \hline
  ${\rm Prime}$&  Yes & Yes & Yes \\
  \hline

\end{tabular}
\end{center}
\vskip 3mm 

%\begin{center}
%\begin{tabular}{ |c|c|c|c|c| } 
%  \hline
%   & {\rm Krull} & {\rm Dedekind} & {\rm near UFD} & {\rm half near UFD}\\ 
%  \hline
%   ${\rm Spec}$& ? & Yes    & ? & ?\\ 
%  \hline
%   ${\rm Irred}$& Yes & Yes & Yes  & No\\ 
%  \hline
%  ${\rm Prime}$& Yes & Yes & Yes & Yes \\
%  \hline
%   ${\rm Nonunit}$& Yes & Yes & Yes  & No\\
%  \hline
%\end{tabular}
%\end{center}
%\vskip 3mm 

\noindent
{\rm (b)} Assume further that ${\mathcal Z}$ is a Hilbertian ring (of characteristic $0$ or imperfect). Then the table above provides the same answers to the question of whether ${\mathcal Z}$ is integrally Hilbertian \hbox{w.r.t.} ${\mathcal P}$.
\end{thm}

%The question marks in the table mean that the corresponding answers are unknown to us.
Theorem \ref{thm:main1} from Section \ref{sec:intro} is the special case of Theorem \ref{thm-integ-Hilb} corresponding to the box ``Krull-Nonunit'' in statement (b).
%; indeed the ring of integers of any number field is a Hilbertian ring by \cite[Theorem 14.4.1]{FJ23}, and is a Krull domain.
A Krull domain is in fact a locally Schinzel ring \hbox{w.r.t.} any of the 3 primality types Nonunit, Irred, Prime.

\begin{rmk}
(a) It follows from Theorem \ref{thm-integ-Hilb}
and Remark \ref{rmk:int-H-coprimeS} that Krull domains satisfy the coprime Schinzel hypothesis. This is a progress compared to \cite{BDN20b} where, even in the special case of Dedekind domains, condition (*) from Remark \ref{rmk:int-H-coprimeS} defining the coprime Schinzel property was only known for $k=1$ (one parameter).
\vskip 1mm

%\noindent
%(b) Recall that 
%if ${\mathcal Q}$ is a {\it field with a product formula} and is of characteristic $0$ or imperfect, then ${\mathcal Z}$ is a Hilbertian ring \cite[Theorem 4.6]{BDN20b}. 
%Thus this more concrete condition can replace the assumption that ${\mathcal Z}$ is a Hilbertian ring in Theorem \ref{thm-integ-Hilb}(b).

%Typical examples of fields with a product formula include the field $\Q$ and the rational function field $k(x)$ in one variable over a field $k$, and their finite extensions (for more on fields with a product formula, see \cite[\S 17.3]{FJ23}). 

%\textcolor{red}{[TO BE CHECKED. TO BE KEPT? :
%Combining this with Theorem \ref{thm-integ-Hilb}(b) yields this generalization of Corollary \ref{thm:main1}.} 

%\textcolor{red}{\begin{cor} \label{thm:main2+} Let ${\mathcal Z}$ be a Krull domain or a near UFD such that ${\mathcal Q}$ has a product formula and is of characteristic $0$ or imperfect. Then ${\mathcal Z}$ is integrally Hilbertian \hbox{w.r.t.} any of the primality types {\rm Nonunit}, {\rm Irred}, {\rm Prime}. ]
%(and also {\rm } in the near UFD case).
%\end{cor}}
%\vskip 1mm

\noindent
(b) Concerning the primality type ${\rm Spec}^\ast$, the answer to the question in (a) is unclear for our three types of rings (Krull, near UFD, PDF). From Proposition \ref{lem:mainaxiomatic} and Example \ref{exa:dednvasf}, the answer to (a) is ``Yes'' for Dedekind domains; hence,  by Proposition \ref{prop:schinzel+Hil=intH}, the answer to (b) is also ``Yes''.
\end{rmk}

%Besides rings of integers of number fields, examples include, given any field $k$, the integral closure ${\mathcal Z}$ of the polynomial ring $k[x]$ in any finite extension $E$ of $k(x)$.
%It is not clear, however, whether the same is true with the indeterminate $x$ replaced by several indeterminates $\underline x$ (unless $E=k(\underline x)$), and more generally, if the integrally Hilbertian property is stable by integral closure.

\subsection{Rings that are Hilbertian but neither integrally Hilbertian nor locally Schinzel rings} \label{ssec:counter-examples}

In \cite[Proposition 2.10]{BDKN22}, the ring $\Z[\sqrt{5}]$ was identified as an example of a domain not satisfying the coprime Schinzel Hypothesis. Therefore it cannot be an integrally Hilbertian ring nor a locally Schinzel ring (Remark \ref{rmk:int-H-coprimeS}). The following shows that such rings, in fact, are abundant. 

\begin{lem}
Let $\mathcal{Z}$ be a domain with non-associate irreducible elements $p,q$ such that the ideal $(p)\cap (q)$ is contained in a unique maximal ideal. Then $\mathcal{Z}$ is neither integrally Hilbertian nor a locally Schinzel ring \hbox{w.r.t.} the primality type {\rm Nonunit}.
\end{lem}
\begin{proof}
This follows immediately from \cite[Proposition 2.8]{BDKN22} and Remark \ref{rmk:int-H-coprimeS}.
\end{proof}

\begin{cor} 
\label{cor:local}
A local domain with more than one irreducible element (up to associates) cannot be integrally Hilbertian nor a locally Schinzel ring \hbox{w.r.t.} the primality type {\rm Nonunit}.
\end{cor}

\begin{cor}
\label{cor:many_nonintH}
Every quadratic number field possesses subrings which are neither integrally Hilbertian nor  locally Schinzel \hbox{w.r.t.}  ${\mathcal P} = {\rm Nonunit}$, but are PDF rings and Noetherian.
\end{cor}

\begin{proof}
Given any squarefree integer $d$ and a prime $p$ which is inert in $\mathbb{Q}(\sqrt{d})$, \cite[Theorem 3.1]{CS12} constructs, via localization of some order of $\mathbb{Q}(\sqrt{d})$, a subring of $\mathbb{Q}(\sqrt{d})$ which is local and has $p+1$ non-associate irreducible elements. From Corollary \ref{cor:local}, such a ring is neither integrally Hilbertian nor a locally Schinzel ring. %They are however Hilbertian rings (as any subring of ${\mathbb Q}(\sqrt{d})$). Hence they cannot be locally Schinzel rings (Proposition \ref{prop:schinzel+Hil=intH}).
Obviously, they are also PDF rings, and they are Noetherian.
\end{proof}

\section{Proof of Theorem \ref{thm-integ-Hilb}}
\label{sec:proof-main-lemma}
\S \ref{ssec:reduction} is a first reduction stage, reducing the proof to that of Lemma \ref{lem:main} below.  \S \ref{ss:stronSchpro} introduces the strategy of the proof of this lemma. Definition \ref{def:schproide} pins down two general properties of a primality type $\mathcal{P}$ of the integral domain ${\mathcal Z}$ that, first, guarantee that $\mathcal{Z}$ is a locally Schinzel ring \hbox{w.r.t.} $\mathcal{P}$ (Proposition \ref{lem:mainaxiomatic}, proved in \S \ref{ssec:prop3.5}), and second, are satisfied (up to equivalence) in all situations of Lemma \ref{lem:main} (Proposition \ref{lem:schcases}, proved in \S \ref{ssec:prop3.6}).
%\vskip 0,5mm

%Let $\mathcal{Z}$ be an integral domain %of characteristic $0$ or imperfect, 
%and $\mathcal{Q}$ be its fraction field. Consider a primality type $\mathcal{P}$ of $\mathcal{Z}$. As before, fix two tuples of indeterminates  $\underline{T}=(T_1,\dots,T_k)$ ($k \geq 0$), and $\underline{Y}=(Y_1,\dots,Y_n)$ ($n \geq 1$).

\subsection{Reduction of the proof of Theorem \ref{thm-integ-Hilb}} 
\label{ssec:reduction}
Lemma \ref{lem:main} corresponds to the sole answers ``Yes'' in the table of statement (a) of Theorem \ref{thm-integ-Hilb}. 
%We reduce the proof of Theorem \ref{thm-integ-Hilb} to showing this lemma.

\begin{lem} \label{lem:main}
A ring $\mathcal{Z}$ is a locally Schinzel ring \hbox{w.r.t.} $\mathcal{P}$ in each of the following situations:
\begin{enumerate}
\item $\mathcal{Z}$ is a Krull domain and $\mathcal{P} \in \{{\rm Nonunit}, {\rm Irred}, {\rm Prime}\}$.
%\item $\mathcal{Z}$ is a Dedekind domain and $\mathcal{P}\in\{{\rm Spec},{\rm Irred} , {\rm Prime} , {\rm Nonunit} \}$.\label{cond1lemmain}
\item $\mathcal{Z}$ is a PDF ring and $\mathcal{P}={\rm Prime}$. \label{cond2lemmain}
\item $\mathcal{Z}$ is a near UFD and $\mathcal{P} \in \{{\rm Nonunit}, {\rm Irred}, {\rm Prime}\}$. \label{cond3lemmain}
\end{enumerate}
\end{lem}

\begin{proof}[Proof of Theorem \ref{thm-integ-Hilb} assuming Lemma \ref{lem:main}] The situations from Lemma \ref{lem:main} correspond to the answers ``Yes'' from the table of Theorem \ref{thm-integ-Hilb}. Thus Lemma \ref{lem:main} directly yields these answers ``Yes'' for statement (a) of Theorem \ref{thm-integ-Hilb}, and also the answers ``Yes'' for statement (b), by Proposition \ref{prop:schinzel+Hil=intH}.

The answers ``No'' concerning PDF rings \hbox{w.r.t} the primality type Nonunit follow directly from Corollary \ref{cor:many_nonintH}. As the rings there are also Noetherian, the same is equivalently true \hbox{w.r.t.} to the primality type ${\mathcal P} = {\rm Irred}$ (the main point of the equivalence being that in a Noetherian ring, every nonunit is divisible by an irreducible).
\end{proof}

\subsection{Properties (NVA) and (SF)}\label{ss:stronSchpro}
Definition \ref{def:schproide} 
%(``the strong Schinzel ring'') 
requests some preliminary definitions. First, we say that two primality types $\mathcal{P}$ and $\mathcal{P}'$ of $\mathcal{Z}$ are \emph{equivalent} if the following holds: {every ideal in $\mathcal{P}$ is contained in some ideal in $\mathcal{P}'$, and every ideal in $\mathcal{P}'$ is contained in some ideal in $\mathcal{P}$}. Note that, for such $\mathcal{P}$ and $\mathcal{P}'$, the ring $\mathcal{Z}$ is a locally Schinzel ring \hbox{w.r.t.} $\mathcal{P}$ if and only if $\mathcal{Z}$ is a locally Schinzel ring \hbox{w.r.t.} $\mathcal{P}'$; and the same holds for the integrally Hilbertian property.  

\begin{exa}
The primality types ${\rm Nonunit} \hskip 1pt \mathcal{Z}$ and ${\rm Irred} \hskip 1pt \mathcal{Z}$ are equivalent if $\mathcal{Z}$ is Noetherian or a Krull domain. Indeed in both these cases, the ascending chain condition for principal ideals is satisfied, which guarantees that every nonunit is divisible by an irreducible element.
\end{exa}

\begin{defn}
The \emph{support} of a nonzero ideal $\mathfrak{a}$ of $\mathcal{Z}$ is:
$$
{\rm Supp}(\mathfrak{a})=\{ \mathfrak{p} \in {\rm Spec}^\ast\hskip 1pt \mathcal{Z} \mid \mathfrak{a} \subset \mathfrak{p}\}.
$$
For a nonzero $a \in \mathcal{Z}$, we use the notation ${\rm Supp}(a)$ for ${\rm Supp}(\langle a \rangle)$. Given a subset $\mathcal{S} \subset {\rm Spec}^\ast \hskip 1pt \mathcal{Z}$, denote by $\Sigma({\mathcal{S}})$ the set of all ideals $\mathfrak{a}$ of $\mathcal{Z}$ such that ${\rm Supp}(\mathfrak{a}) \subset \mathcal{S}$. More generally, for $\mathcal{S} \subset \mathcal{E} \subset {\rm Spec}^\ast\hskip 1pt\mathcal{Z}$, denote by $\Sigma(\mathcal{S})_{\mathcal{E}}$ the set of all ideals $\mathfrak{a}$ of $\mathcal{Z}$ such that ${\rm Supp}(\mathfrak{a}) \cap \mathcal{E} \subset \mathcal{S}$.
\end{defn}
If $\mathcal{Z}$ is a Dedekind domain, $\Sigma(\mathcal{S})$ (resp., $\Sigma(\mathcal{S})_{\mathcal{E}}$) is the set of all ideals $\mathfrak{a}$ for which the prime ideals (resp., the prime ideals in $\mathcal{E}$) involved in the prime ideal factorization of $\mathfrak{a}$ are in $\mathcal{S}$.
\begin{rmk}\label{rmk:sigminter}
For any subsets $\mathcal{S} \subset {\rm Spec}^\ast \hskip 1pt \mathcal{Z}$ and $\mathcal{E} \subset {\rm Spec}^\ast \hskip 1pt \mathcal{Z}$, we have $\Sigma(\mathcal{S}) \subset \Sigma(\mathcal{S} \cap \mathcal{E})_{\mathcal{E}}$.
\end{rmk}
\begin{defn}\label{def:schproide}
A domain $\mathcal{Z}$ equipped with a primality type $\mathcal{P}$ is said to satisfy:

%The ring $\mathcal{Z}$ is called a \emph{strong Schinzel ring} \hbox{w.r.t.} a primality type $\mathcal{P}$ if $\mathcal{P}$ is equivalent to a primality type $\mathcal{P}'$ satisfying the two following conditions:
\begin{enumerate}
\item \label{eq:nonvanpr} the \emph{nonvanishing approximation property} -- (NVA)$\null_{\mathcal P}$ for short, if for any nonzero polynomial $P \in \mathcal{Z}[\underline{T},\underline{Y}]$, and any finite subset $\mathcal{S} \subset \mathcal{P}$, the following holds. If a family $(\underline{u}_{\mathfrak{a}})_{\mathfrak{a} \in \mathcal{S}}$ of $k$-tuples $\underline{u}_{\mathfrak{a}} \in \mathcal{Z}^k$ indexed by $\mathcal{S}$ satisfies: 
$$
P(\underline{u}_{\mathfrak{a}},\underline{Y}) \not \equiv 0 \pmod{\mathfrak{a}}\,\,\text{for all}\,\, \mathfrak{a} \in \mathcal{S},
$$
then there exists $\underline{v} \in \mathcal{Z}^k$ such that:
$$
P(\underline{v},\underline{Y}) \not \equiv 0 \pmod{\mathfrak{a}}\,\,\text{for all}\,\, \mathfrak{a} \in \mathcal{S},
$$
\item the \emph{support finiteness property} -- (SF)$\null_{\mathcal P}$ for short, if  there exists a subset $\mathcal{E} \subset {\rm Spec}^\ast\hskip 1pt \mathcal{Z}$ such that:\label{eq:exismaide}
\begin{enumerate}
\item for any nonzero $a \in \mathcal{Z}$, the set $
{\rm Supp}(a) \cap \mathcal{E}
$ is finite;
\item for any finite subset $\mathcal{S} \subset \mathcal{E}$, the set $
\Sigma(\mathcal{S})_{\mathcal{E}} \cap \mathcal{P}$
is finite.
\end{enumerate}
\end{enumerate}
\end{defn}
\begin{exa}\label{exa:dednvasf}
Let $\mathcal{Z}$ be a Dedekind domain. Using the Chinese remainder theorem, it is easily  checked that $\mathcal{Z}$ satisfies (NVA)$\null_{\mathcal P}$ for ${\mathcal P} = {\rm Spec}^\ast\hskip 1pt \mathcal{Z}$; and property 
(SF)$\null_{\mathcal P}$ is also satisfied  for $\mathcal{E}={\rm Spec}^\ast\hskip 1pt \mathcal{Z}$. Property 
(SF)$\null_{\mathcal P}$ also holds for ${\mathcal P}= {\rm Irred}$ (and $\mathcal{E}={\rm Spec}^\ast\hskip 1pt \mathcal{Z}$). But it does not for ${\mathcal P}= {\rm Nonunit}$: powers $a^n$ of a nonunit $a$ have the same support; hence condition (2)(b) fails.
\end{exa}
We can state the two key intermediate results from which Lemma \ref{lem:main} follows.

\begin{prop} \label{lem:mainaxiomatic}
Let $\mathcal{Z}$ be a domain equipped with a primality type $\mathcal{P}$. If properties  {\rm (NVA)}$\null_{\mathcal P}$ and {\rm (SF)}$\null_{\mathcal P}$ are satisfied, then ${\mathcal Z}$ is a locally Schinzel ring \hbox{w.r.t.} ${\mathcal P}$. \end{prop}

\begin{prop}\label{lem:schcases}
In each of the situations of Lemma \ref{lem:main}, the primality type $\mathcal{P}$ is equivalent to a primality type $\mathcal{P}^\prime$ of ${\mathcal Z}$ such that {\rm (NVA)}$\null_{{\mathcal P}^\prime}$ and {\rm (SF)}$\null_{{\mathcal P}^\prime}$ are satisfied. 
\end{prop}

\subsection{Proof of Proposition \ref{lem:mainaxiomatic}} \label{ssec:prop3.5}

\subsubsection{Preliminary lemmas} 
We establish some results needed for the proof of Proposition \ref{lem:mainaxiomatic}. For any integer $B \geq 1$, define: 
$$
\Gamma_B=\left\lbrace \mathfrak{p} \in {\rm Spec}^\ast\hskip 1pt\mathcal{Z}\mid |\mathcal{Z}/\mathfrak{p}| \leq B \right\rbrace.
$$
\begin{lem}\label{lem:interrestgen}
For any integer $B \geq 1$, there exists a nonzero $a \in \mathcal{Z}$ such that $\Gamma_B \subset {\rm Supp}(a)$. 
\end{lem}
\begin{proof}
The proof is the same as that of \cite[Lemma 3.9]{BDKN22} by replacing primes with prime ideals. For the sake of completeness, we reproduce the argument. Fix an integer $B\geq 1$. For every prime power $q=\ell^r \leq B$, take an element $m_q \in \mathcal{Z}$ such that $m_q^q-m_q \neq 0$. Let $a$ be the product
of all $m_q^q-m_q$ with $q$ running over all prime powers $q \leq B$. 

Consider now a prime ideal $\mathfrak{p} \in \Gamma_B$. The integral domain $\mathcal{Z}/\mathfrak{p}$, being
finite, is a field. Hence $|\mathcal{Z}/\mathfrak{p}|$ is a prime power $q=\ell^r$, and $q \leq B$. Since $m_q^q-m_q \equiv 0 \pmod{\mathfrak{p}}$, the prime ideal $\mathfrak{p}$ contains $a$, i.e. $\mathfrak{p} \in {\rm Supp}(a)$. Therefore $\Gamma_B \subset {\rm Supp}(a)$.
\end{proof}
Lemma \ref{lem:zero} locates the support of fixed divisors of polynomials. Recall that, for a nonzero polynomial $P \in \mathcal{Z}[\underline{T},\underline{Y}]$, the set $\mathcal{F}_{\underline{T}}(P)$ denotes the set of all fixed ideal divisors of $P$ \hbox{w.r.t.} $\underline{T}$. Moreover, for a nonzero $P \in \mathcal{Z}[\underline{Y}]$, the set ${\rm Div}_{\mathcal{Z}}(P)$ denotes the set of all ideal divisors of $P$.
\begin{lem}\label{lem:zero}
Let $P \in \mathcal{Z}[\underline{T},\underline{Y}]$ be a nonzero polynomial. Set $\Delta=\max_{1 \leq i \leq k}\deg_{T_i}(P)$ and $\Omega={\rm Div}_{\mathcal{Z}}(P) \cap {\rm Spec}^\ast\hskip 1pt \mathcal{Z}$. Then we have these inclusions: 
\begin{enumerate}
\item $(\mathscr{F}_{\underline{T}}(P) \cap {\rm Spec}^\ast\hskip 1pt \mathcal{Z}) \setminus {\rm Div}_{\mathcal{Z}}(P) \subset \Gamma_{\Delta}$. \label{lem:zeropart1}
\item $\mathscr{F}_{\underline{T}}(P) \subset \Sigma(\Omega \cup \Gamma_{\Delta})$.  \label{lem:zeropart2}
\end{enumerate}
\end{lem}
\begin{proof}
(\ref{lem:zeropart1}) We wish to show that if $\mathfrak{p}$ is a fixed divisor of $P$ \hbox{w.r.t.} $\underline{T}$ in ${\rm Spec}^\ast\hskip 1pt \mathcal{Z}$ such that $\mathfrak{p}$ does not divide $P$, then $|\mathcal{Z}/\mathfrak{p}| \leq \Delta$. The proof follows from the same reasoning as in \cite[Lemma 3.1 (a)]{BDKN22}, with prime elements replaced by prime ideals.
\vskip 1mm
\noindent
(\ref{lem:zeropart2}) Let $\mathfrak{a} \in \mathscr{F}_{\underline{T}}(P)$. We wish to show that ${\rm Supp}(\mathfrak{a}) \subset \Omega \cup \Gamma_{\Delta}$. Let $\mathfrak{p} \in {\rm Supp}(\mathfrak{a})$, \hbox{i.e.} $\mathfrak{p} \supset \mathfrak{a}$. If $\mathfrak{p} \in {\rm Div}_{\mathcal{Z}}(P)$, then $\mathfrak{p} \in \Omega$. Now assume that $ \mathfrak{p} \notin {\rm Div}_{\mathcal{Z}}(P)$. Since $ \mathfrak{a} \in \mathscr{F}_{\underline{T}}(P)$, we have $\mathfrak{p} \in \mathscr{F}_{\underline{T}}(P)$. By \eqref{lem:zeropart1}, we obtain $\mathfrak{p} \in \Gamma_{\Delta}$.
\end{proof}

\subsubsection{Proof of Proposition \ref{lem:mainaxiomatic}}
Let $\mathcal{E} \subset {\rm Spec}^\ast \hskip 1pt \mathcal{Z}$ be the subset associated with (SF)$\null_{\mathcal P}$. 

Let $P_1,\ldots,P_s \in {\mathcal Z}[\underline{T},\underline{Y}]$ be $s$ polynomials satisfying (Prim$/{\mathcal Q}[\underline T]$) and $\hbox{\rm (NoFixDiv}/{\mathcal Z}[\underline T])_{\mathcal P}$. Set $P=P_1 \cdots P_s$. Consider $P_1,\ldots,P_s$ as polynomials in $\underline{T}^\prime=(T_2,\dots,T_k)$ and $\underline{Y}$ and with coefficients in $\mathcal{Z}[T_1]$. Note that they are all primitive \hbox{w.r.t.} $\mathcal{Q}[T_1]$. If follows that $P$ is primitive \hbox{w.r.t.} $\mathcal{Q}[T_1]$ as well. Thus, the coefficients $Q_1,\dots,Q_r \in \mathcal{Z}[T_1]$ ($r \geq 1$) of $P$, considered as a polynomial in $\underline{T}',\underline{Y}$, are coprime in $\mathcal{Q}[T_1]$. As in \cite[\S 2.1]{BDKN22}, choose a nonzero element 
$$\delta\in \Biggl( \sum_{j=1}^{r} Q_{j}\mathcal{Z}[T_1] \Biggl) \cap \mathcal{Z}.
$$   
For $\Delta=\max_{1 \leq i \leq k}\deg_{T_i}(P)$, set: 
 \begin{equation}\label{eq:ererer}
\mathcal{S}=\mathcal{P} \cap \Sigma\left(\hskip 1pt [\hskip 1pt {\rm Supp}(\delta) \cup {\rm Div}_{\mathcal{Z}}(P) \cup \Gamma_\Delta] \cap \mathcal{E}\hskip 2pt \right)_{\mathcal{E}}. 
 \end{equation}
 By the support finiteness property, ${\rm Supp}(\delta) \cap \mathcal{E}$ and ${\rm Div}_{\mathcal{Z}}(P) \cap \mathcal{E}$ are finite; for the latter, note that it is contained in ${\rm Supp}(a) \cap \mathcal{E}$ for any nonzero coefficient $a$ of $P$ in $\mathcal{Z}$. Moreover, by Lemma \ref{lem:interrestgen} and again the support finiteness property, $\Gamma_{\Delta} \cap \mathcal{E}$ is finite. Therefore $[{\rm Supp}(\delta) \cup {\rm Div}_{\mathcal{Z}}(P) \cup \Gamma_\Delta] \cap \mathcal{E}$ is finite too. It follows that $\mathcal{S}$ is finite. Choose a nonzero element $\omega \in \prod_{\mathfrak{p} \in \mathcal{S}} \mathfrak{p} \subset \bigcap_{\mathfrak{p} \in \mathcal{S}} \mathfrak{p}$ -- we take $\omega=1$ if $\mathcal{S}=\emptyset$.
\begin{claim}\label{clm:1lem11we}
There exist arithmetic progressions
$$
\tau_1= (\omega \ell + v_{1})_{\ell\in \mathcal{Z}}\,,\, \tau_2= (\omega \ell + v_{2})_{\ell\in \mathcal{Z}}\,, \dots,\, \tau_k= (\omega \ell + v_{k})_{\ell\in \mathcal{Z}} \subset \mathcal{Z}
$$ such that
\begin{equation}
P(\underline u,\underline{Y}) \not\equiv 0 \pmod{\mathfrak{p}},
\end{equation}
for all $\underline{u}=(u_1,\dots,u_k) \in \tau_1 \times \dots \times \tau_k$ and all $\mathfrak{p} \in \mathcal{S}$.
\end{claim}
\begin{proof}[Proof of Claim \ref{clm:1lem11we}]
From $\hbox{\rm (NoFixDiv}/{\mathcal Z}[\underline T])_{\mathcal P}$, for every $\mathfrak{p} \in \mathcal{S}$,
there exists 
$\underline u_{\mathfrak{p}} \in \mathcal{Z}^k$ such that 
\begin{equation}\label{eq:pcongr1we}
P(\underline u_{\mathfrak{p}},\underline{Y})\not\equiv 0 \pmod{\mathfrak{p}}.
\end{equation}
By the nonvanishing approximation property, there exists $\underline{v}=(v_1,\dots,v_k) \in \mathcal{Z}^k$ such that
$$
P(\underline{v},\underline{Y}) \not\equiv 0 \pmod{\mathfrak{p}}\,\,\, \textrm{for all}\,\, \mathfrak{p} \in \mathcal{S}.
$$ 

Set $\tau_1 \times \dots \times \tau_k=(\omega \ell + v_{1})_{\ell\in \mathcal{Z}} \times \dots \times (\omega \ell + v_{k})_{\ell\in \mathcal{Z}}$. Let $\underline{u}=(u_1,\dots,u_k) \in \tau_1 \times \dots \times \tau_k$ and $\mathfrak{p} \in \mathcal{S}$. Since $\omega \in \bigcap_{\mathfrak{p} \in \mathcal{S}}\mathfrak{p} $, we have $\underline{u}\equiv \underline{v} \pmod{\mathfrak{p}}$, and thus
\begin{equation}
P(\underline u,\underline{Y})\equiv P(\underline{v},\underline{Y})\not\equiv 0 \pmod{\mathfrak{p}}.
\end{equation}
\end{proof}
Consider the following polynomials, where $V_1$ is a new variable:
$$\widetilde P_i (V_1,\underline{T}^\prime, \underline{Y})= P_i(\omega V_1+v_1, \underline{T}^\prime, \underline{Y}) \in \mathcal{Z}[V_1,\underline{T}^\prime,\underline{Y}],\hskip 2mm  (i=1,\dots,s),
$$ and set $\widetilde P = \prod_{i=1}^s \widetilde P_i$. To complete the proof of Proposition \ref{lem:mainaxiomatic}, it remains to establish the following.
\begin{claim}\label{clm:existinfmanywe}
For all but finitely many $\ell_1 \in \mathcal{Z}$, the polynomials $$\widetilde{P}_1(\ell_1,\underline{T}^\prime,\underline{Y}),\dots,\widetilde{P}_s(\ell_1,\underline{T}^\prime,\underline{Y})$$ satisfy \hbox{\rm (Prim$/{\mathcal Q}[\underline T^\prime]$)} and $\hbox{\rm (NoFixDiv}/{\mathcal Z}[\underline T^\prime])_{\mathcal P}$.
\end{claim}
\begin{proof}[Proof of Claim \ref{clm:existinfmanywe}]

As polynomials in $\mathcal{Q}[V_1,\underline{T}^\prime][\underline{Y}]$, the polynomials $\widetilde P_1,\ldots, \widetilde P_s$ satisfy (Prim$/{\mathcal Q}[V_1,\underline T^\prime]$). We next use \cite[Proposition 3.1]{BD22}, which is an analog for coprimality of the Bertini--Noether theorem  (originally for irreducibility): for all but finitely many $\ell_1\in \mathcal{Z}$, the polynomials $$\widetilde{P}_1(\ell_1,\underline{T}^\prime,\underline{Y}),\dots,\widetilde{P}_s(\ell_1,\underline{T}^\prime,\underline{Y})$$ satisfy (Prim$/{\mathcal Q}[\underline T^\prime]$).

Fix $\ell_1\in \mathcal{Z}$ such that the above property holds. Assume on the contrary that 
$$
\widetilde{P}_1(\ell_1,\underline{T}',\underline{Y}),\dots,\widetilde{P}_s(\ell_1,\underline{T}',\underline{Y})
$$ do not satisfy $\hbox{\rm (NoFixDiv}/{\mathcal Z}[\underline T^\prime])_{\mathcal P}$, that is, there exists an ideal $\mathfrak{p}$ in $\mathcal{F}_{\underline{T}'}(\widetilde{P}(\ell_1,\underline{T}^\prime,\underline{Y})) \cap \mathcal{P}$. Then $$\widetilde{P}(\ell_1, \underline{t}^\prime, \underline{Y}) \equiv 0 \pmod{\mathfrak{p}},$$ for all $\underline{t}^\prime \in \mathcal{Z}^{k-1}$. 

On the one hand, we claim that $\mathfrak{p} \notin \mathcal{S}$. Otherwise, since $\omega\ell_1+v_1 \equiv v_1 \pmod{\mathfrak{p}}$, Claim \ref{clm:1lem11we} implies that: 
$$
0\equiv \widetilde{P}(\ell_1,v_2,\dots,v_k,\underline{Y})\equiv P(\underline{v},\underline{Y}) \not \equiv 0 \pmod{\mathfrak{p}},
$$ a contradiction.

On the other hand, we claim that $\mathfrak{p} \in \mathcal{S}$. Indeed, with $\Delta'=\max_{2 \leq i \leq k}\deg_{T_i}(\widetilde{P}(\ell_1, \underline{T}',\underline{Y}))$, Lemma \ref{lem:zero} and Remark \ref{rmk:sigminter} lead to 
\begin{equation}\label{eq:frere}
    \mathscr{F}_{\underline{T}'}(\widetilde{P}(\ell_1,\underline{T}',\underline{Y})) \cap \mathcal{P} \subset \Sigma\left(\left({\rm Div}_{\mathcal{Z}}(\widetilde{P}(\ell_1,\underline{T}',\underline{Y}) \cup \Gamma_{\Delta'}\right) \cap \mathcal{E}\right)_{\mathcal{E}}\cap \mathcal{P}.
\end{equation}
\noindent
If $\mathfrak{q} \in {\rm Div}_{\mathcal{Z}}(\widetilde{P}(\ell_1,\underline{T}',\underline{Y})) \cap \mathcal{E}$, then $Q_{i}(\omega \ell_1+v_1) \equiv 0 \pmod{\mathfrak{q}}$ for $i=1,\ldots,r$. Therefore $\mathfrak{q} \in {\rm Supp}(\delta)$ and so 
${\rm Div}_{\mathcal{Z}}(\widetilde{P}(\ell_1,\underline{T}',\underline{Y})) \cap \mathcal{E} \subset {\rm Supp}(\delta) \cap \mathcal{E}
$. Combining this with the inclusion $
\Gamma_{\Delta'} \subset \Gamma_{\Delta}$, it now follows from \eqref{eq:ererer} and \eqref{eq:frere} that
$$
\mathscr{F}_{\underline{T}'}(\widetilde{P}(\ell_1,\underline{T}',\underline{Y})) \cap \mathcal{P} \subset  \Sigma\left(({\rm Supp}(\delta) \cup {\rm Div}_{\mathcal{Z}}(P) \cup \Gamma_\Delta) \cap \mathcal{E}\right)_{\mathcal{E}} \cap \mathcal{P}=\mathcal{S}.
$$ Thus $\mathfrak{p} \in \mathcal{S}$, thus contradicting the preceding paragraph. 

This completes the proof of Claim \ref{clm:existinfmanywe} and of Proposition \ref{lem:mainaxiomatic}.
\end{proof}

\subsection{Proof of Proposition \ref{lem:schcases}} \label{ssec:prop3.6}
\subsubsection{Preliminary lemmas} To prove Proposition \ref{lem:schcases}, we need the following two results.

\begin{lem}\label{lem:interrest}
Assume that $\mathcal{Z}$ is a Krull domain and let $\mathcal{E}=\s1(\mathcal{Z})$.
\begin{enumerate}
\item The primality type ${\mathcal P}={\rm Irred}$ satisfies the support finiteness property for $\mathcal{E}$. \label{lem:finirrproide}  
\item For any integer $B \geq 1$, the set $\Sigma(\Gamma_B \cap \mathcal{E})_{\mathcal{E}} \cap {\rm Irred}\hskip 1pt \mathcal{Z}$ is finite.
\label{cor:finSigmaM}
\end{enumerate}
\end{lem}
\begin{proof}
\begin{enumerate}
\item Since $\mathcal Z$ is a Krull domain, by \cite[Page 289, III)]{Mat80}, the set ${\rm Supp}(a) \cap \mathcal{E}$ is finite for every nonzero element $a \in \mathcal Z$. Hence condition (a) of Definition \ref{def:schproide}(2) holds.
For condition (b), fix a finite subset $\mathcal{S}=\{ \mathfrak{p}_1,\dots,\mathfrak{p}_r\}\subset  \mathcal{E}$. We wish to show that the set 
$$
\Sigma(\mathcal{S})_{\mathcal{E}} \cap \mathcal{P}:=\{ (\theta) \in \mathcal{P} \mid {\rm Supp}(\theta) \cap \mathcal{E} \subset \mathcal{S} \}
$$ is finite. Define
$$
\mathcal{K}=\{ v_{\mathcal{S}}((\theta)):=(v_{\mathfrak{p}_1}(\theta),\dots,v_{\mathfrak{p}_r}(\theta)) \in \mathbb{N}^r \mid (\theta) \in \Sigma(\mathcal{S})_{\mathcal{E}} \cap \mathcal{P}\}.
$$ By \cite[Page 289, III)]{Mat80}, this is naturally in bijection with $\Sigma(\mathcal{S})_{\mathcal{E}} \cap \mathcal{P}$; Indeed, for any $(\theta) \in \Sigma(\mathcal{S})_{\mathcal{E}} \cap \mathcal{P}$, we have $(\theta)=\bigcap_{i=1}^r \mathfrak{p}_i^{(v_{\mathfrak{p}_i(\theta)})}$. Thus, it suffices to prove the following.

\begin{claim} \label{claim:finiteK}
The set $\mathcal{K}$ is finite.
\end{claim}

\begin{proof}[Proof of Claim \ref{claim:finiteK}] 
For any $\mathfrak{p} \in \mathcal{E} \setminus \mathcal{S}$ and any $(\theta) \in \Sigma(\mathcal{S})_{\mathcal{E}} \cap \mathcal{P}$, we have $v_{\mathfrak{p}}(\theta)=0$. We equip $\mathcal{K}$ with the partial order defined by $$(a_1,\dots,a_r)\leq (b_1,\dots,b_r)\quad\textrm{if and only if} \quad a_i \leq b_i \quad \textrm{for all }1 \leq i \leq r.
$$ By \cite[Page 288, I)]{Mat80}, for $(\theta), (\tau) \in \Sigma(\mathcal{S})_{\mathcal{E}} \cap \mathcal{P}$, we have 
\begin{equation}\label{eq:uncomkrull}
v_{\mathcal{S}}(\theta) \leq v_{\mathcal{S}}(\tau) \,\,\textrm{ if and only if }\,\, \theta \mid \tau \,\,\textrm{ if and only if }\,\, (\tau)=(\theta).
\end{equation} 
Consequently, distinct elements of $\mathcal{K}$ are pairwise noncomparable; in this setting, ${\bf c}$ and ${\bf d}$ are \emph{noncomparable} if ${\bf c}\not\leq {\bf d}$ and ${\bf d} \not\leq {\bf c}$. 

On the other hand, with the partial order on $\mathbb{N}^r$:
$${\bf c}=(c_1,\dots,c_r) \leq {\bf d}=(d_1,\dots,d_r)\,\, \textrm{if and only if}\,\,c_1 \leq d_1,\dots,c_r \leq d_r,$$ it is well-known that $(\mathbb{N}^r, \leq)$ does not contain an infinite family of pairwise noncomparable elements, see, e.g., \cite[Theorem 1.1(1)]{AA2020}.
%\footnote{ To prove that, one can proceed by induction on $r \geq 1$. For $r=1$, the statement is clear. Suppose it holds for $r-1 \geq 1$. Assume on the contrary that it does not hold for $r$. Consider an infinite set $X \subset \mathbb{N}^r$ of noncomparable elements. Note that, for $a \geq 0$ and $1 \leq i \leq r$, from the induction hypothesis, $X \cap \{x_i \leq a\}=\bigcup_{k=0}^a (X \cap \{x_i=k\})$ is finite. Fix ${\bf a} \in X$. By the previous observation, $X \cap \{ x_1 > a_1\}$ is infinite. This implies (replacing $X$ with $X \cap \{x_1 > a_1\}$) that $X \cap \{x_1 > a_1\} \cap \{ x_2 >a_2 \}$ is also infinite. Continuing so, we see that $X \cap \{x_1 > a_1 \} \cap \dots \cap \{x_r > a_r\}$ is infinite. Therefore, there exists ${\bf a}' \in X$ such that ${\bf a}' > {\bf a}$, a contradiction. See also \cite[Problem 83]{She02}.}. 

Consequently $\mathcal{K}$ is finite.
\end{proof} 

\item Let $B \geq 1$ be an integer. By Lemma \ref{lem:interrestgen}, there exists a nonzero $a \in \mathcal{Z}$ such that $\Gamma_B \cap \mathcal{E} \subset {\rm Supp}(a) \cap \mathcal{E}$, so $\Gamma_B \cap \mathcal{E}$ is finite. Combining it with \eqref{lem:finirrproide}, we deduce that $\Sigma(\Gamma_B \cap \mathcal{E})_{\mathcal{E}} \cap {\rm Irred}\hskip 1pt \mathcal{Z}$ is finite.
\end{enumerate}
\end{proof}

\begin{rmk}
%[Proof of Proposition \ref{lem:fixed-divisors-removal}] 
\label{rmk:fixed-div-finite}
At this stage, we can provide a proof of the remaining Krull case of Proposition \ref{lem:fixed-divisors-removal}. Assume that $\mathcal{Z}$ is a Krull domain, and set  $\mathcal{E}=\s1(\mathcal{Z})$.
Let $P \in \mathcal{Z}[\underline{T},\underline{Y}]$ be a nonzero polynomial.  
One may assume that $P$ has a fixed divisor in ${\rm Nonunit} \hskip 1pt \mathcal{Z}$ \hbox{w.r.t.} $\underline{T}$; otherwise, take $\varphi = 1$. Choose $\underline{m} \in \mathcal{Z}^k$ such that $P(\underline{m},\underline{Y})$ is nonzero. Pick a nonzero coefficient $a$ of $P(\underline{m},\underline{Y})$ and consider the set $\mathcal{S} = {\rm Supp}(a) \cap \mathcal{E}$. By the support finiteness property (Lemma \ref{lem:interrest}), the set 
$\mathcal{S}$ is finite, and so is the set
%set  
%\[
%\mathcal{S} = {\rm Supp}(a) \cap \mathcal{E}.
%\]
  
\vskip 1mm

\centerline{$\Theta = \Sigma(\mathcal{S})_{\mathcal{E}} \cap {\rm Irred} \hskip 1pt \mathcal{Z}.$}
\vskip 1mm

We claim that we can take $\varphi$ to be a generator of the principal ideal $\prod_{\mathfrak{p} \in \Theta} \mathfrak{p}$. Indeed, assume on the contrary that $P$ has a fixed divisor $\langle u/\varphi^\ell \rangle \in {\rm Nonunit} \hskip 1pt \mathcal{Z}[1/\varphi]$ \hbox{w.r.t.} $\underline{T}$, where $u \in \mathcal{Z}$ and $\ell \geq 0$. In particular,  
\[
P(\underline{m},\underline{Y})\equiv 0 \pmod{u/\varphi^\ell},
\]
so $u$ divides $a \cdot \varphi^\ell$. Then  
$$
{\rm Supp}(u) \cap \mathcal{E} \subset ({\rm Supp} (a) \cap \mathcal{E}) \bigcup \left(\left( \bigcup_{\mathfrak{p} \in \Theta} {\rm Supp}(\mathfrak{p}) \right) \cap \mathcal{E}\right) = \mathcal{S}.
$$
Consider a decomposition $u=\pi_1 \cdots \pi_e$ into a product of irreducible elements. Since ${\rm Supp}(\langle \pi_i \rangle) \cap  \mathcal{E}$ ($1 \leq i \leq e$) is also contained in $\mathcal{S}$, we have $\langle \pi_1 \rangle,\dots,\langle \pi_e \rangle \in \Theta$. Therefore, each element $\pi_i$ divides $\varphi$ ($i=1,\ldots,e$) and so $u$ divides $\varphi^e$, implying that $u$ is a unit in $\mathcal{Z}[1/\varphi]$ -- a contradiction.  

Consequently, $P$ has no fixed divisor in ${\rm Nonunit}\hskip 1pt \mathcal{Z}[1/\varphi]$ \hbox{w.r.t.} $\underline{T}$.  
\end{rmk}

\begin{lem}\label{lem:chinremthe}
Suppose that ${\mathcal Z}$ and $\mathcal{P}$ satisfy one of the following:
\begin{enumerate}[(a)]
\item $\mathcal{Z}$ is a PDF ring and $\mathcal{P} ={\rm Prime}$.\label{item:casea}
\item $\mathcal{Z}$ is a Krull domain and $\mathcal{P}= {\rm Irred}$.
\label{item:caseb} 
\end{enumerate}
Then $\mathcal{P}$ satisfies the nonvanishing approximation property.
\end{lem}
\begin{proof}
Let $P\in \mathcal{Z}[\underline{T},\underline{Y}]$ be a nonzero polynomial, let $\mathcal{S} \subset \mathcal{P}$ be a finite set, and let $(\underline{u}_{\mathfrak{a}})_{\mathfrak{a} \in \mathcal{S}} \subset \mathcal{Z}^k$ be such that: $$
P(\underline{u}_{\mathfrak{a}},\underline{Y})\not\equiv 0 \pmod{\mathfrak{a}}\,\,\,\,\text{for all}\,\, \mathfrak{a} \in \mathcal{S}.
$$

\noindent
\emph{Case (\ref{item:casea})}. The proof follows from that of \cite[Lemma 4.2]{BDKN22}. For the convenience of the reader, we repeat the argument. Denote by \( \mathcal{S}_1 \subset \mathcal{S} \) the subset of prime ideals \( \mathfrak{p} \) that are maximal ideals. We apply \cite[Lemma 3.8]{BDKN22}. From above, \( P \) is nonzero modulo each \( \mathfrak{a} \in \mathcal{S} \setminus \mathcal{S}_1 \), and we have \( \mathfrak{a} \not\subset \mathfrak{a}'\) for any distinct \( \mathfrak{a}, \mathfrak{a}' \in \mathcal{S}\) (as $\mathfrak{a}$ and $\mathfrak{a}'$ are principal ideals generated by prime elements). Thus, \cite[Lemma 3.8]{BDKN22} provides \( \underline{v} = (v_1, \dots, v_k) \in \mathcal{Z}^k \) such that
\[
\underline{v} \equiv \underline{u}_{\mathfrak{a}} \pmod{\mathfrak{a}} \quad \text{for all prime } \mathfrak{a} \in \mathcal{S}_1,
\]
and
\[
P(\underline{v}, \underline{Y}) \not\equiv 0 \pmod{\mathfrak{a}} \quad \text{for all } \mathfrak{a} \in \mathcal{S} \setminus \mathcal{S}_1.
\]
These congruences imply that
\[
P(\underline{v}, \underline{Y}) \not\equiv 0 \pmod{a} \quad \text{for all } \mathfrak{a} \in \mathcal{S}.
\]
\vskip 1mm

\noindent
\emph{Case (\ref{item:caseb})}. Let $
\mathcal{E}=\s1(\mathcal{Z})$.
%We will prove that there exists $\underline{v} \in \mathcal{Z}^k$ such that
%$$
%P(\underline{v},\underline{Y}) \not \equiv 0 \pmod{\mathfrak{a}}\quad \textrm{for all }\mathfrak{a} \in \mathcal{S}.
%$$
Since $\mathcal{S}$ is finite, by \cite[Page 289, III)]{Mat80}, the set 
$$
\mathcal{R}=\bigcup_{\mathfrak{a} \in \mathcal{S}}({\rm Supp}(\mathfrak{a}) \cap \mathcal{E})
$$ is finite, and writing $\mathcal{R}=\{\mathfrak{p}_1,\dots,\mathfrak{p}_m\}$ ($m \geq 1$), we have, for each $\mathfrak{a} \in \mathcal{S}$
%(\cite[Page 289, III)]{Mat80})
$$
\mathfrak{a}=\mathfrak{p}_1^{(e_1(\mathfrak{a}))}\cap \dots \cap \mathfrak{p}_m^{(e_m(\mathfrak{a}))} ,
$$ where 
$$(e_1(\mathfrak{a}),\dots,e_m(\mathfrak{a}))=(v_{\mathfrak{p}_1}(\mathfrak{a}),\dots,v_{\mathfrak{p}_m}(\mathfrak{a}))\in \mathbb{N}^m.
$$ 

For each 
$i=1,\ldots, m$, set
$$
\mathcal{S}_i=\left\lbrace \mathfrak{a}\in \mathcal{S} \mid e_i(\mathfrak{a})\geq 1\,\, \textrm{and}\,\,P(\underline{u}_{\mathfrak{a}},\underline{Y})\not\equiv 0 \pmod{\mathfrak{p}_i^{(e_i(\mathfrak{a}))}}\right\rbrace.
$$
Let 
$$
J=\{ j \in \{1,\dots,m \} \mid \mathcal{S}_j\neq \emptyset\}.
$$ 
For each $j \in J$, choose $\mathfrak{a}_j \in \mathcal{S}_j$ such that $$
e_j(\mathfrak{a}_j)=\min_{\mathfrak{a} \in \mathcal{S}_j}(e_j(\mathfrak{a})) \geq 1.
$$ By the weak approximation theorem  (Lemma \ref{lem:weakappro}), there exists $\underline{v} \in {\mathcal Z}^k$ such that
$$
\underline{v} \equiv \underline{u}_{\mathfrak{a}_j} \pmod{\mathfrak{p}_j^{(e_j(\mathfrak{a}_j))}}\quad \textrm{for all } j \in J.
$$  

To complete the proof, we show that 
$P(\underline{v},\underline{Y}) \not \equiv 0 \pmod{\mathfrak{a}}\hskip 1mm \textrm{for all }\mathfrak{a} \in \mathcal{S}$.
%this $\underline{v}$ works. 
Let $\mathfrak{a} \in \mathcal{S}$. Since $P(\underline{u}_{\mathfrak{a}},\underline{Y})\not\equiv 0 \pmod{\mathfrak{a}}$, there exists $j \in \{1,\dots,m\}$ with $e_j(\mathfrak{a}) \geq 1$ such that $P(\underline{u}_{\mathfrak{a}},\underline{Y})\not\equiv 0 \pmod{\mathfrak{p}_j^{(e_j(\mathfrak{a}))}}$; in other words, $\mathfrak{a} \in \mathcal{S}_j$. As 
$$P(\underline{v},\underline{Y}) \equiv P(\underline{u}_{\mathfrak{a}_j},\underline{Y}) \not\equiv 0 \pmod{\mathfrak{p}_j^{(e_j(\mathfrak{a}_j))}}
$$ and $e_j(\mathfrak{a}) \geq e_j(\mathfrak{a}_j)$, we obtain $$
P(\underline{v},\underline{Y}) \not\equiv 0 \pmod{\mathfrak{p}_j^{(e_j(\mathfrak{a}))}}.
$$ Therefore we can conclude that
\vskip 1mm

\hskip 55mm {$P(\underline{v},\underline{Y}) \not\equiv 0 \pmod{\mathfrak{a}}.
$} \end{proof}

\subsubsection{End of proof of Proposition \ref{lem:schcases}}\label{ssc:enproof36} The three cases of Lemma \ref{lem:main} are cases (1),(2),(3) below. Lemma \ref{lem:chinremthe} ensures the nonvanishing approximation property for the \hbox{(equivalent to ${\mathcal P}$)} primal\-ity type $\mathcal{P}'$ as chosen below in each case. It remains to check the \hbox{support finiteness property.}
\vskip 1mm
\noindent
(1) Let $\mathcal{Z}$ be a Krull domain. For $\mathcal{P}={\rm Prime}$, the support finiteness property holds since $\mathcal{Z}$ is a PDF ring by Remark \ref{rmk:finmanydiv}. For $\mathcal{P} \in \{  {\rm Irred}, {\rm Nonunit}\}$, take $\mathcal{P}'={\rm Irred}$ and $\mathcal{E}=\s1(\mathcal{Z})$. The support finiteness property then follows from Lemma \ref{lem:interrest}.  
\vskip 1mm
\noindent
(2) Let $\mathcal{Z}$ be a  PDF ring and $\mathcal{P}={\rm Prime}$. Take $\mathcal{P}'=\mathcal{P}$ and $\mathcal{E}=\mathcal{P}$. The support finiteness property holds from the definition of PDF ring.\label{seccase:ss}
\vskip 1mm
\noindent
(3) Let $\mathcal{Z}$ be a near UFD and $\mathcal{P} \in \{ {\rm Irred}, {\rm Prime}, {\rm Nonunit}\}$. Take $\mathcal{P}'={\rm Prime}$ and $\mathcal{E}={\rm Prime}$. The support finiteness property holds from the definition of near UFDs.

\section{Polynomial rings are locally Schinzel: Proof of Theorem \ref{thm:main1-2}} 
\label{sec:poly-rings}

Here we show Theorem \ref{thm:main1-2}. We begin with a lemma linking fixed divisors and ``actual'' divisors over polynomial rings.

\begin{lem}
\label{lem:fixeddiv}
Let $\underline{T}=(T_1,\dots,T_k)$ and $\underline{Y}=(Y_1,\dots,Y_n)$ with $k ,n\geq 1$. Let $\mathcal{Z}=\mathcal{R}[U]$ for an infinite domain $\mathcal{R}$.  Then every fixed divisor $\pi\in \mathcal{Z}$ of a polynomial $f \in \mathcal{Z}[\underline{T},\underline{Y}]$ w.r.t. $\underline{T}$ is a divisor of $f$ in $\mathcal{Z}[\underline{T}]$. 
\end{lem}

\begin{proof}
Without loss of generality, we may assume that $f$ belongs to $\mathcal{Z}[\underline{T}]$. Indeed, $\pi \in \mathcal{Z}$ is a fixed divisor of $f \in \mathcal{Z}[\underline{T},\underline{Y}]$ w.r.t. $\underline{T}$ if and only it is a fixed divisor w.r.t. $\underline{T}$ of each of its $\mathcal{Z}[\underline{T}]$-coefficients. 
\vskip 1mm
\noindent
{\bf Step 1: The case $k=1$.} Write
$$
f=a_d(U)T_1^d+\dots+a_1(U)T_1+a_0(U) \in \mathcal{Z}[T_1]
$$
and suppose $\pi \in \mathcal{Z}$ is a fixed divisor of $f$ w.r.t. $T_1$. Since $\mathcal{R}$ is infinite, there exists $s \in \mathcal{R}$ such that for the automorphism 
$$\varphi_s: U \in \mathcal{Z} \mapsto U-s \in \mathcal{Z},
$$ the polynomial $\varphi_s(\pi)$, viewed as a polynomial in $U$,  has no root in 
\begin{equation}\label{eq:roote}
\{0\} \cup \{w \in \overline{{\rm Frac}(\mathcal{R})} \mid w^i=1\quad\textrm{for some }i=1,\dots,d \}.
\end{equation} Clearly $\varphi_s(\pi)$ is a fixed divisor of $\varphi_s(f)$ w.r.t. $T_1$. Hence, we may assume without loss of generality that $\pi$ itself has no root in \eqref{eq:roote}. Since $\pi$ is fixed divisor, we have
$$
f(U^i)=a_d(U)U^{id}+\dots+a_1(U)U^i+a_0(U) \equiv 0 \pmod{\pi}\quad i=0,\dots,d.
$$
Using Vandermonde determinant, this yields
$$
\prod_{0 \leq i < j\leq d}(U^j-U^i) (a_0(U),\dots,a_d(U)) \equiv 0 \pmod{\pi}.
$$
 We claim that none of the factors $U^j-U^i$ is a zero divisor modulo $\pi$. Indeed, suppose 
$$U^j(U^{j-i}-1)g=(U^j-U^i)g= \pi r
$$ for some $g,r \in \mathcal{Z}$. Since $0 \leq j-i \leq d$ and $\pi$ has no root in \eqref{eq:roote}, it follows that $\pi$ and $U^j(U^{j-i}-1)$ are coprime in ${\rm Frac}(\mathcal{R})[U]$. Hence $U^{j}(U^{j-i}-1)$ divides $r$ in ${\rm Frac}(\mathcal{R})[U]$, and therefore also in $\mathcal{Z}=\mathcal{R}[U]$ since $U^j(U^{j-i}-1)$ is monic. Thus $U^j-U^i$ is a not a zero divisor modulo $\pi$. It follows that $\pi$ divides $a_0(U),\dots,a_d(U)$ as desired.
\vskip 1mm
\noindent
{\bf Step 2: The general case.} We proceed by induction on $k$. The case $k=1$ has been established. Assume now that $k \geq 2$ and let $\underline{T}'=(T_2,\dots,T_k)$, and suppose the statement holds for $1,\dots,k-1$ variables. Write
$$
f=b_\ell(\underline{T}')T_1^\ell+ \dots+b_1(\underline{T}')T_1+b_0(\underline{T}')\,\,\, (b_i(\underline{T}') \in \mathcal{Z}[\underline{T}']).
$$
Fix $(v_2,\dots,v_{k}) \in \mathcal{Z}^{k-1}$. Then $\pi$ is a fixed divisor of $f(T_1,v_2,\dots,v_d)$ w.r.t. $T_1$. By the one-variable case, $\pi$ divides each coefficient 
$$
b_j(v_2,\dots,v_k),\quad j=1,\dots,d
$$ Since $(v_2,\dots,v_k)$ is arbitrary, the induction hypothesis applied to the polynomials $b_j(\underline{T}')$ implies that $\pi$ divides each $b_j(\underline{T}')$. Hence $\pi$ divides $f$.
\end{proof}

Building on Lemma \ref{lem:fixeddiv}, the following is the crucial ingredient in proving (a more explicit version of) Theorem \ref{thm:main1-2}.

\begin{prop} \label{prop:R[U]-loc-Schinzel-ring}
Let $\mathcal R$ be an infinite domain, let $K:=\textrm{\rm Frac}(\mathcal{R})$, and set $\mathcal{Z} = \mathcal{R}[U]$. 
Let $k,n,s\ge 1$, $\underline{T} := (T_1,\dots, T_k)$,
$\underline T':=(T_2,\dots, T_k)$ and $\underline Y:=(Y_1,\dots, Y_n)$.\\
Let $\mathcal{P}\in \{\rm{Nonunit}\hskip 1pt \mathcal{Z}, \rm{Irred} \hskip 1pt \mathcal{Z}, \rm{Prime} \hskip 1pt \mathcal{Z}\}$ be a primality type, and let $P_1(\underline T, \underline Y), \dots, P_s(\underline T,\underline Y) \in {\mathcal Z}[\underline T,\underline Y]$ be polynomials satisfying the following.
\begin{itemize}
\item[i)] $P_1\cdots P_s$ has no divisor in $\mathcal{P}$.
\item[ii)] When viewed as elements of $K[U][\underline{T}, \underline{Y}]$, none of $P_1,\dots, P_s$ has a divisor in $K[U,T_1]\setminus K[U]$.
\end{itemize}
\iffalse
Then there exists a finite set $S\subset \mathcal R$, 
an integer $C_0>0$ and a polynomial $h(U)\in \mathcal{Z}\setminus\{0\}$ such that for all $\tau\in \mathcal{R}\setminus S$, for all integers $C>C_0$, and all $t_1(U)\in \mathcal{Z}$ of the form $t_1(U) = U^C h(U)\cdot g(U) + \tau$ ($g(U)\in \mathcal{Z}$ arbitrary), the following holds:
\fi
Then there exists a polynomial $h(U)\in \mathcal{Z}\setminus\{0\}$ such that for all sufficiently large integers $C$ (with the bound depending on $P_1,\dots, P_s$), there is a finite set $S\subset\mathcal{R}$ fulfilling the following:\\
For all $t_1(U)\in \mathcal{Z}$ of the form 
$$t_1(U)  = U^{2C} h(U) g(U) + U^C + \tau,$$
with $g(U)\in \mathcal{Z}$ arbitrary and $\tau\in \mathcal{R}\setminus S$, the polynomial
$\prod_{i=1}^s P_i(t_1(U),\underline T', \underline Y)\in \mathcal{Z}[\underline T', \underline Y]$ does not have a divisor in $\mathcal{P}$.
\end{prop}

\begin{proof}
Let $P:=P_1\cdots P_n$.  Let $c_1(U,T_1),\dots, c_N(U,T_1)\in \mathcal R[U,T_1]$ be the coefficients of $P$ (viewed as a polynomial in $\underline {T}', \underline Y$). By Assumption i), there is no $d\in \mathcal P$ 
%\mathcal P\mathcal R[U]
dividing all the $c_i$. 

To show the assertion that $\prod_{i=1}^s P_i(t_1(U),\underline{T}',\underline{Y})$ does not have a divisor in $\mathcal{P}$ (for $t_1(U)$ as given above), we distinguish the two cases of $d\in \mathcal{P}\cap \mathcal R$ and of $d\in \mathcal{P}\cap(\mathcal R[U]\setminus \mathcal R)$.
\vskip 1mm

\noindent
\emph{Case 1)} To treat the case of divisors $d\in \mathcal{P}\cap \mathcal{R}$,  choose an integer $C$ larger than the maximal $U$-degree of $c_i(U,T_1)$, $i=1,\dots, N$. For arbitrary $g(U)\in \mathcal{Z}$ and $\tau\in \mathcal{R}$, set 
\vskip 1mm

\centerline{$t_1(U) = U^{2C} g(U) + U^C + \tau$.}
\vskip 1mm

\noindent
Up to replacing $T_1$ by $T_1':=T_1-\tau$, we may reduce to the case $\tau=0$, and will therefore assume this in the following argument. We claim that $c_1(U,t_1(U)),\dots, c_N(U,t_1(U))$ have no common (non-unit, resp.\ irreducible, resp.\ prime, depending on the primality type $\mathcal{P}$) divisor $d\in \mathcal{R}$. Assume on the contrary the existence of such a divisor $d$. Write $c_i(U,T_1) = \sum_{j=0}^{d_i}\gamma_{i,j}(U) T_1^j$ with $\gamma_{i,j}\in \mathcal{Z}$, and expand $$c_i(U,t_1(U)) = \sum_{j=0}^{d_i}\gamma_{i,j}(U)\cdot (U^{2C} g(U) + U^C)^j.$$ 
 Note that, by choice of $C$, the terms of $\gamma_{i,0}(U)$ are exactly the ones of degree $<C$ in this expansion, whence $d$ must divide $\gamma_{i,0}$. We may thus replace $c_i$ by $c_i-\gamma_{i,0}(U)$ and now continue iteratively. I.e., we assume that $\gamma_{i,j}(U)$ is divisible by $d$ for all $j$ up to some bound $k<d_i$, and consider the remaining sum $\sum_{j=k+1}^{d_i} \gamma_{i,j}(U)\cdot (U^{2C} g(U) + U^C)^j$. Here, the lowest-degree terms form the sum $U^{C(k+1)}\gamma_{i,k+1}(U)$, allowing us to conclude that $d$ must divide  $\gamma_{i,k+1}(U)$ as well, and thus must eventually divide all the $\gamma_{i,j}(U)$. In conclusion, $d$ divides all the $c_i$. But then $d$ is a divisor of $P$, contradicting the assumptions of the theorem.

Note that this first case has required no exceptions among the values $\tau\in \mathcal{R}$ and no restrictions on the polynomial $h(U)$ from the assertion.
\vskip 1mm

\noindent
\emph{Case 2)} To treat the case of divisors $d\in \mathcal{P}\cap(\mathcal R[U]\setminus \mathcal R)$, denote by $C_1,\dots, C_N$ the affine plane curves over $\overline{K}$ given by the equations $c_1(U,T_1)=0,\dots, c_N(U,T_1)=0$, respectively. 
We will for the moment make the following extra assumption (*), to which the general case will be reduced:

%\begin{center}
%(*) The coefficients $c_1(U,T_1),\dots, c_N(U,T_1)$ %have no non-unit common divisor in $K[U]$.
%\end{center}

\vskip 1mm
\noindent
(*) \emph{The coefficients $c_1(U,T_1),\dots, c_N(U,T_1)$ have no non-unit common divisor in $K[U]$.}

\vskip 1mm

Under Assumption (*), it follows that the curves $C_1,\dots, C_N$ do not have a common irreducible component. Indeed, otherwise $c_1(U,T_1),\dots, c_N(U,T_1)$ would have an irreducible common divisor in the UFD $K[U,T_1]$. But if such a divisor has positive $T_1$-degree, it would divide one of the $P_i$, contradicting Assumption ii);
on the other hand, an irreducible divisor in $K[U]$ would contradict Assumption (*).

By B\'ezout's theorem, there are then only finitely many points  $A_1,\dots, A_m\in \mathbb{A}^2(\overline{K})$ in which all the $C_i$ intersect.
Let $h(U)\in \mathcal{Z}\setminus\{0\}$ be such that $h(\alpha)=0$ for all $\alpha$ occurring as the $U$-coordinate of some $A_i$, $i=1,\dots, m$. Concretely, one may choose $h(U)$ as the product of minimal polynomials of these $\alpha$ over $K$, multiplied by a suitable constant in $\mathcal{R}$ to clear denominators. 
Next, given any $C\in \mathbb{N}$, let $S (=S(C))\subset \mathcal{R}$ be the finite set of values $\beta\in\mathcal{R}$ such that some $A_i$, $i=1,\dots, m$, has $(U,T_1)$-coordinates $(\alpha, \alpha^C + \beta)$, $\alpha\in \overline{K}$.

For $\tau\in \mathcal{R}\setminus S$ and for
$\tilde{g}(U)\in \mathcal{Z}$ arbitrary, set $t_1(U) = \tilde{g}(U)h(U) + U^C + \tau$, and define the curve $C_{N+1}$ via $C_{N+1}: T_1 = t_1(U)$. By definition of $S$ and $\tau$, $C_{N+1}$ does not pass through any of the points $A_i$, $i=1,\dots, m$. This implies that $c_1(U,t_1(U)),\dots, c_N(U,t_1(U))$ have no common non-constant divisor $d\in \mathcal{R}[U]\setminus \mathcal{R}$, since indeed the roots $u$ of such a divisor in $\overline{K}$ would by definition induce simultaneous intersection points $(u,t_1(u))$ of the curves $C_1,\dots, C_{N+1}$.

\vskip 1mm

\noindent
\emph{General situation of 2)}. Clearly, the set of values $t_1(U)$ given in the assertion of the proposition fulfills the requirements of both the sets given in Cases 1) and 2). To finish the proof of the proposition, it thus remains to treat the general situation of 2), i.e., without assuming the additional  condition (*).  
\iffalse \textcolor{blue}{Isn't it a little confusing? To me, we rather reduce the general situation to that of (2) with condition (*) (?).}
\textcolor{red}{JK: I intended to say the general case, namely ``Assume there's a divisor in $\mathcal{P}\cap (\mathcal{Z}[U]\setminus\mathcal{Z})$ is going to be reduced to the case in which additionally (*) holds. Maybe one could alternatively say we are going to show that the extra assumption (*) always holds under the assumptions of the proposition. (I originally planned making this a separate lemma, but it's a bit problematic because the below argument also uses the ideas from Case 1), so is not truly separate.)}
\textcolor{blue}{[P: I think we agree. I just wanted to check with you. Please choose the wording you prefer.]}
\fi

Assume to this end that $\delta(U)\in K[U]$ is the greatest common divisor (necessarily of $T_1$-degree $0$ due to Assumption ii)) of the coefficients $c_1(U,T_1),\dots, c_N(U,T_1)$, viewed as elements of $K[U,T_1]$.
Set 
%$\tilde{c}_i : =\frac{c_i}{\delta}\in K[U,T_1]$. 
$\tilde{c}_i : =c_i/\delta\in K[U,T_1]$.
Since the $\tilde{c}_i$ satisfy Assumption (*), we may conclude from the above (applied with $K$ in place of $\mathcal{R}$) that the $\tilde{c}_i(U,t_1(U))$, $i=1,\dots, N$, have no nonunit common divisor in $K[U]$ as soon as $t_1(U)$ is of the form 
%\vskip 1mm

\begin{equation} 
\label{eq:poly_shape} 
t_1(U)= U^{2C}g(U)h(U) + U^C + \tau
\end{equation}

%\vskip 1mm

\noindent
(for $g(U)\in K[U]$ arbitrary, $h(U)\in K[U]$ suitable, and $\tau\in K$ away from a suitable finite set depending on $C$). In particular, there is an infinite subfamily of values 
$t_1(U)$ satisfying \eqref{eq:poly_shape} and additionally lying in $\mathcal R[U]$.\footnote{One should only multiply $h(U)\in K[U]$ by a suitable constant in $\mathcal{R}$ in order to replace it by a polynomial in $\mathcal{R}[U]$, which is possible due to $g(U)$ being arbitrary.} 

Fix  $t_1(U)\in \mathcal R[U]$ as above. Assume on the contrary that $c_1(U,t_1(U)),\dots,c_N(U,t_1(U))$ have a common divisor $d(U) \in \mathcal{R}[U] \setminus \mathcal{R}$. It follows from the $\tilde{c}_i(U,t_1(U))$, $i=1,\dots, N$ having no nonunit common divisor in $K[U]$ that $d(U)$ divides $\delta(U)$ in $K[U]$.
By Assumption i), there exists $1 \leq i_0 \leq N$ such that $c_{i_0}(U,T_1)$ is not divisible in ${\mathcal R}[U,T_1]$ by $d(U)$. 
As in Case 1), up to applying a shift in the variable $T_1$, we may and will assume $\tau=0$ for the following. 
Then
$$
c_{i_0}(U,t_1(U))=\sum_{j=0}^{d_{i_0}} \gamma_{i_0,j}(U)\left(U^{2C}g(U)h(U) + U^C\right)^j.
$$
For each $0 \leq j \leq d_{i_0}$, since $\delta(U)$  divides $\gamma_{i_0,j}(U)$ and $d(U)$ divides $\delta(U)$ in $K[U]$, we have 
$$
\tilde{\gamma}_{i_0,j}(U):=\frac{\gamma_{i_0,j}(U)}{d(U)} \in K[U].
$$
Then
$$
\sum_{j=0}^{d_{i_0}} \tilde{\gamma}_{i_0,j}(U)\left(U^{2C}g(U)h(U) + U^C\right)^j=\frac{c_{i_0}(U,t_1(U))}{d(U)} \in \mathcal{R}[U].
$$ By the same argument as Case 1), one shows inductively that, for every $0 \leq j \leq d_{i_0}$, one has 
$$\tilde{\gamma}_{i_0,j}(U) \in \mathcal{R}[U].
$$ It follows that $d(U)$ divides $c_{i_0}(U,T_1)$, a contradiction.
\end{proof}

\begin{proof}[Proof of Theorem \ref{thm:main1-2}]
By Proposition \ref{prop:schinzel+Hil=intH}, since $\mathcal{Z}=\mathcal{R}[U]$ is a Hilbertian ring, the integrally Hilbertian property is implied by the locally Schinzel property, i.e., it suffices to prove the latter.
 When ${\mathcal R}$ is a finite domain, it is a finite field, and ${\mathcal Z}={\mathcal R}[U]$ is a PID, hence a locally Schinzel domain as well. We may and will therefore assume that $\mathcal{R}$ is infinite. Let $\mathcal{Q}=\textrm{Frac}(\mathcal{Z})$, let $\mathcal{P}\in \{\rm{Nonunit} \mathcal{Z}, \rm{Irred} \mathcal{Z}, \rm{Prime} \mathcal{Z}\}$ be a primality type, and $P_1(\underline T, \underline Y), \dots, P_s(\underline T,\underline Y)$ be polynomials satisfying conditions {\rm (Prim}$/{\mathcal Q}[\underline T])$ and $\hbox{\rm (NoFixDiv}/{\mathcal Z}[\underline T])_{\mathcal P}$ from \S \ref{ssec:int-Hilb-ring}. Let $P:=P_1\cdots P_s$. Due to Condition $\hbox{\rm (NoFixDiv}/{\mathcal Z}[\underline T])_{\mathcal P}$, $P$ is in particular not divisible by any element in $\mathcal{P}$, and due to Condition {\rm (Prim}$/{\mathcal Q}[\underline T])$, none of $P_1,\ldots, P_s$ is divisible by any element in $\textrm{Frac}(\mathcal{R})[T_1,U] \setminus \textrm{Frac}(\mathcal{R})[U]$. 
 %Since $\textrm{Frac}(\mathcal{Z})[U][\underline T,\underline Y]$ is a UFD, in particular $P$  not divisible by any element in $\textrm{Frac}(\mathcal{Z})[T_1,U] \setminus \textrm{Frac}(\mathcal{Z})[U]$ either.  
 By Proposition \ref{prop:R[U]-loc-Schinzel-ring}, 
 there exist full arithmetic progressions 
 $\omega \mathcal{Z} + \alpha \subset\mathcal{Z}$
 of values $t_1(U)$ for which
 %for all but finitely many $\alpha\in \mathcal{Z}$, all but finitely many $C\in \mathbb{N}$ and all $g\in \mathcal{Z}[U]$, 
 the polynomial $P(t_1(U), \underline{T}',\underline{Y})\in \mathcal{Z}[\underline T', \underline{Y}]$ is not divisible by any element in $\mathcal{P}$. 
 Concretely, Proposition \ref{prop:R[U]-loc-Schinzel-ring} gives $\omega = U^{2C}h(U)$ for suitably large $C\in \mathbb{N}$ and suitable $h(U)\in \mathcal{Z}\setminus \{0\}$, and $\alpha = U^C + \tau$ for all $\tau\in \mathcal{R}$ away from a finite set (depending on $C$). 
 By Lemma \ref{lem:fixeddiv}, this implies that $P_1(t_1(U),\underline{T}',\underline{Y}),\dots, P_s(t_1(U), \underline{T}',\underline{Y})$ satisfy $\hbox{\rm (NoFixDiv}/{\mathcal Z}[\underline T'])_{\mathcal P}$.

Moreover, condition {\rm (Prim}$/{\mathcal Q}[\underline T])$ implies in particular that the coefficients w.r.t. $\underline{Y}$ of $P_1(\underline{T},\underline{Y}), \dots, P_s(\underline{T},\underline{Y})$, viewed as elements of $\mathcal{Q}(T_1)[\underline{T}']$, are coprime, thus satisfying Condition ii) of \cite[Proposition 3.1]{BD22}, applied with $\underline{a}=T_1$ and $\underline{x}=\underline{T}'$. Due to the equivalence ii)$\Leftrightarrow$iii) of  \cite[Proposition 3.1]{BD22}, there exists a finite subset $Z\subseteq \mathcal{Q}$ such that for all $t_1\in \mathcal{Q}\setminus Z$, the specializations of the above coefficients are coprime as elements of $\mathcal{Q}[\underline{T}']$, i.e., the polynomials $P_1(t_1,\underline{T}', \underline{Y}), \dots, P_s(t_1,\underline{T}', \underline{Y})$ satisfy {\rm (Prim}$/{\mathcal Q}[\underline T'])$. In particular, up to increasing the lower bound on the exponent $C\in \mathbb{N}$ above, we may assume that this conclusion holds for all the values $t_1(U)\in \omega\mathcal{Z}+\alpha$ exhibited above. As a consequence, $\mathcal{Z}$ is locally Schinzel.
\end{proof}

\section{Polynomial Schinzel rings} \label{sec:applications}

 \S \ref{ssec:schinzel-rings} introduces Schinzel rings (Definition \ref{def:schinzel}). This notion recasts in a common environment the celebrated Schinzel Hypothesis (thanks to Theorem \ref{thm:schinzel1=>Schinzelk}) and our results on polynomial rings. These are Theorem \ref{thm:main3} and Theorem \ref{thm:hypschistronggen}, which are more precise forms of Theorem \ref{thm:schinzel-main}.
 They are stated and commented on in \S \ref{SchH-general} and in \S \ref{ssec:hypschistronggen}. Their proofs are postponed to Section \ref{sec:proofs}.
The primality type is tacitly taken to be ${\mathcal P} = {\rm Nonunit}$ throughout the section.

\subsection{Schinzel rings} \label{ssec:schinzel-rings}
\subsubsection{Fields with a product formula}
\label{sssec:FDP}
Recall from \cite[\S 17.3]{FJ23} that a field ${\mathcal Q}$ equipped with a nonempty set $S$ of primes ${\frak p}$, with associated absolute value $|\hskip 2mm |_{\frak p}$, is said to satisfy the \emph{product formula} if for each ${\frak p}\in S$, there exists $\beta_{{\frak p}} >0$ such that:
\vskip 1mm

\noindent
(*) {\it For each $a\in {\mathcal Q}^\times$, the set $\{{\frak p}\in S \hskip 2pt | \hskip 2pt |a|_{\frak p} \not=1 \}$ is finite and $\prod_{{\frak p}\in S} |a|_{\frak p}^{\beta_{{\frak p}}} = 1$.}

\vskip 1mm

The field ${\mathbb Q}$ is a typical example: the product formula is: $\prod_{p} |a|_p \cdot |a| = 1$ for every $a\in {\mathbb Q}^\ast$, where $p$ ranges over all prime numbers, $|\cdot|_p$ is the $p$-adic absolute value and $|\cdot |$ is the standard absolute value. Rational function fields $\kappa(u_1,\ldots,u_r)$ in $r\geq 1$ variables over an arbitrary field $\kappa$, and finite extensions of fields with a product formula are other examples \cite[\S 17.3]{FJ23}. 

From a result of Weissauer, fields with a product formula, of characteristic $0$ or imperfect, are Hilbertian fields \cite[Theorem 17.3.3]{FJ23}. A domain ${\mathcal Z}$, of characteristic $0$ or imperfect, such that the fraction field ${\mathcal Q}$ has a product formula is a Hilbertian ring \cite[Theorem 4.6]{BDN20b}.

Associated to a product formula, relative to a set $S$ of primes, comes a natural height on ${\mathcal Q}$ -- the \emph{Weil height} -- which we denote by $H_S$ and  is defined as follows: for every $a\in {\mathcal Q}^\times$,
\vskip 1mm

\centerline{$H_S(a) = \prod_{{\frak p} \in S} \max(1,|a|_{\frak p})^{\beta_{\frak p}}$.}

\vskip 1mm
\noindent
Clearly $H_S(a^n) = H_S(a)^n$ ($n\in {\mathbb N}$) and 
$H_S(1/a) = H_S(a)$ if $a\not=0$. 
When ${\mathcal Q} = {\mathbb Q}$, $H_S$ is the usual absolute value. On the field $\kappa(\underline Y)$, each variable $Y_i$ induces a partial degree Weil height: $H_{i}(\cdot) = 2^{\deg_{u_i}(\cdot)}$, $i=1,\ldots, n$.

%The field ${\mathcal Q}$ is a priori assumed to be a \emph{field with a product formula}, relative to a given nonempty set $S$ of primes $\mathfrak{p}$ of ${\mathcal Q}$ (\S \ref{sssec:FDP}). 
%Definition is recalled in \S \ref{sssec:FDP}. For more on fields with a product formula, see \cite[\S 17.3]{FJ23}
%%%%%%%%%%%%%%%%%%%%%%%%%%%%%%%
%%%%%%%%%%%%%%%%%%%%%%%%%%%%%%%
%%%%%%%%%%%%%%%%%%%%%%%%%%%%%%%
%%%%%%%%%%%%%%%%%%%%%%%%%%%%%%%
%Our basic examples will be the field of rationals ${\mathbb Q}$ and the rational function fields $\kappa(\underline Y)$ in $n\geq 1$ variables over an arbitrary field $\kappa$. Other examples include finite extensions of fields with a product formula. 

\subsubsection{Definition of Schinzel rings}

\begin{defn} \label{def:schinzel} 
Let ${\mathcal Z}$ be domain such that the fraction field ${\mathcal Q}$ is equipped with several sets of primes $S_1,\ldots,S_n$ satisfying the product formula. Let $H_1,\ldots,H_n$ be the corresponding Weil heights. The ring  ${\mathcal Z}$ is called a \emph{Schinzel ring \hbox{w.r.t. the heights $H_1,\ldots,H_n$}} if the following holds. Let $\underline P$ be a finite set of polynomials $P_1,\ldots,P_s \in {\mathcal Z}[\underline T]$,  irreducible in $\mathcal{Q}[\underline T]$
and such that the product $P_1\cdots P_s$ has no fixed divisor in ${\mathcal Z}$ \hbox{w.r.t.} $\underline T$.
Let $A$ be a positive real number. Then the following subset of ${\mathcal Z}^k$ is Zariski-dense:
\vskip 1mm

\centerline{\hskip 6mm ${\mathcal S}_{\mathcal{Z}}(\underline P,A)
= \left\{\underline t = (t_1,\ldots,t_k)\in {\mathcal Z}^k \hskip 1pt \left\vert\hskip 1pt \begin{array}{c}
     P_1(\underline t), \ldots, P_s(\underline t) \ \hbox{\rm are irreducible in ${\mathcal Z}$}, \hbox{and} \hfill \hfill \\
     H_i(t_j) \geq A,\ \hbox{for all } j=1,\ldots,k \ \hbox{and}\  i=1,\ldots,n. \hfill \hfill \\
    \end{array}\right. \right\} 
    $.}
\end{defn}

The original Schinzel Hypothesis corresponds to the special case ${\mathcal Z}={\mathbb Z}$ and $k=1$, but is in fact equivalent to the full definition of ``${\mathbb Z}$ is a Schinzel ring'', as shown by Theorem \ref{thm:schinzel1=>Schinzelk} below, which already appeared in \cite[Lemma 2.3]{KK25} for $s=1$ and ${\mathcal Z}={\mathbb Z}$, and here is proved in \S \ref{ssec:pf-schinzel1=>Schinzelk}.
Also note that for ${\mathcal Z}={\mathbb Z}$, Definition \ref{def:schinzel} is unchanged if one drops the height condition in the definition of ${\mathcal S}_{\mathcal{Z}}(\underline P,A)$. This is not the case however for ${\mathcal Z} = \kappa[Y_1,\ldots,Y_n]$.

%\begin{thm}\label{thm:schinzel:1-equiv-k} The special case $k=1$ of Definition \ref{def:schinzel} is equivalent to the full case $k\geq 1$ if ${\mathcal Z}$ is a Krull domain.
%\end{thm}

%\subsubsection{One-parameter Schinzel implies several parameters}

\begin{thm} \label{thm:schinzel1=>Schinzelk}
Let $\mathcal{Z}$ be a Krull domain such that $\mathcal Q$ is equipped with several sets of primes $S_1,\dots,S_n$ satisfying the product formula. Assume that $\mathcal{Z}$ satisfies the special $k=1$ parameter case of Definition \ref{def:schinzel}. Then $\mathcal{Z}$ is a Schinzel ring in the full sense, \hbox{i.e.} for several parameters. 
\end{thm}

%: containing polynomials of arbitrarily large degree is strictly stronger than being infinite.
\vskip 1mm

In this context of Schinzel rings, we have the following results:
\vskip 0,5mm

\noindent
(a) (stated as Theorem \ref{thm:schinzel-main}). \emph{If ${\mathcal Z}$ is an integrally Hilbertian ring, then the polynomial ring ${\mathcal Z}[\underline Y]$ in $n\geq 1$ variables is a Schinzel ring \hbox{w.r.t.} the partial degrees Weil heights $\deg_{Y_i}(\cdot)$, $i=1,\ldots,n$}.
\vskip 1mm

\noindent
(b) (stated as Corollary \ref{cor:Z[Y_0,...,Y_n]}). \emph{If ${\mathcal Z}$ is an arbitrary domain, the polynomial ring ${\mathcal Z}[Y_0,Y_1,\ldots,Y_n]$ in $n+1 \geq 2$ variables is a Schinzel ring \hbox{w.r.t.} the partial degree Weil heights $\deg_{Y_i}(\cdot)$, $i=0,1,\ldots,n$.} 
\vskip 1mm

\noindent
Result (a) will be established as a weak form of Theorem \ref{thm:main3}, which is stated in \S \ref{thm:main3} and proved in \S \ref{ssec:Schinzel-H}. Result (b) is proved in \S \ref{ssec:pf-corollary1.9}.

\subsection{Effective polynomial Schinzel Hypothesis} \label{SchH-general}

 Theorem \ref{thm:main3} goes beyond the Schinzel ring property stated in Theorem \ref{thm:schinzel-main}. Its statement and proof are expressed in pure polynomial terms. 

% Theorem \ref{thm:main3} is a more precise and more general version of the Schinzel ring property  Theorem \ref{thm:schinzel-main}. Theorem \ref{thm:main3} is stated and commented in \S \ref{SchH-general}. Its proofs are postponed to \S \ref{}.
 
\subsubsection{Statement of Theorem \ref{thm:main3}}
Given an $n$-tuple $\underline d=(d_1,\ldots,d_n)$ of nonnegative integers, denote the number of monic monomials in $\underline Y$ of degree $\leq d_j$ in $Y_j$ ($j=1,\ldots,n$) by $\ell(\underline d)$:
\vskip 1mm

\centerline{$\ell(\underline d) = \prod_{j=1}^n (d_j+1)$,}
\vskip 1mm

\noindent
%The proof of Theorem \ref{thm:main3} will give a more precise conclusion than the mere existence of the polynomials  $M_1(\underline Y),\ldots, M_k(\underline Y)$. 
%For every $n$-tuple $\underline d = (d_1,\ldots, d_n)$ of nonnegative integers, 
and denote the set of polynomials $M \in {\mathcal Z}[\underline Y]$ such that $\deg_{Y_j}(M) \leq d_j$ 
($j= 1,\ldots,n$) by ${\mathcal P}ol_{{\mathcal Z},n,\underline{d}}$. We view it as the $\ell(\underline d)$-dimensional affine space over ${\mathcal Z}$: the coordinates correspond to the coefficients.

\begin{thm} \label{thm:main3}
Assume that ${\mathcal Z}$ is an integrally Hilbertian ring.
Let $P_1,\ldots,P_s \in {\mathcal Z}[\underline{T},\underline{Y}]$ be $s\geq 1$ polynomials, irreducible in $\mathcal{Q}[\underline{T},\underline{Y}]$, and such that the product $P_1\cdots P_s$ has no nonunit divisor in ${\mathcal Z}$.
For each $i=1,\ldots,k$, let ${\underline d}_i=(d_{i1},\ldots,d_{in})$ be an $n$-tuple of nonnegative integers. Assume that \vskip 1mm

\noindent
{\rm (1)} $\deg_{\underline Y}(P_l) \geq 1$ for $l=1,\ldots,s$ or ${\underline d}_i \not= (0,\ldots,0)$ for $i=1,\ldots,k$. 
\vskip 2mm 

\noindent
Consider the subset
\vskip 1mm

\centerline{${\mathcal S}_{{\mathcal Z}[\underline Y]}(\underline P,{\underline d}_1,\ldots,{\underline d}_n) \subset {\mathcal P}ol_{{\mathcal Z},n,\underline{d}_1} \times \cdots \times {\mathcal P}ol_{{\mathcal Z},n,\underline{d}_k}$} 
\vskip 2mm

\noindent
consisting of all $k$-tuples $(M_1(\underline Y), \ldots, M_k(\underline Y))$ such that   \vskip 1,5mm

\noindent
{\rm (2)} \hskip 10mm $P_i\hskip 2pt (M_1(\underline Y), \ldots, M_k(\underline Y),Y_1, \ldots, Y_n)$
is irreducible in ${\mathcal Z}[\underline Y]$, $i=1,\ldots, s$.
%\vskip 0,5mm

%\hskip 12mm and $\deg_{Y_j}(M_i) = d_{ij}$, $j=1,\ldots,n$, $i=1,\ldots, k$.
\vskip 2mm

\noindent
Then we have the following three conclusions:
\vskip 1mm

\noindent
{\rm (a)} Assume further that
\vskip 0,5mm

\noindent
{\rm (3.a)} \hskip 25mm $\ell({\underline d}_i)> \sum_{j=1}^s \deg_{T_i}(P_j)$, \hskip 3mm $i=1,\ldots,k$.
\vskip 0,5mm

\noindent
Then the subset ${\mathcal S}_{{\mathcal Z}[\underline Y]}(\underline P,{\underline d}_1,\ldots,{\underline d}_n)$ is Zariski-dense in ${\mathcal P}ol_{{\mathcal Z},n,\underline{d}_1} \times \cdots \times {\mathcal P}ol_{{\mathcal Z},n,\underline{d}_k}$.

%it is possible to substitute some polynomials $M_i(\underline Y)\in {\mathcal Z}[\underline Y]$ of degree $d_{ij}$ in $Y_j$ ($j=1,\ldots,n$) for each parameter $T_i$ ($i=1,\ldots,k$) in the original polynomials $P_1,\ldots,P_s$ in such a way that 

%\noindent
%Furthermore, we have these adjusted conclusions in the following special cases:
\vskip 2mm

\noindent
{\rm (b)} If ${\mathcal Z}$ is a {\it near UFD}, then the same holds with {\rm (3.a)}
%condition on $\ell(\underline d_1), \ldots, \ell(\underline d_k)$ 
replaced by 
\vskip 0,5mm

\noindent
{\rm (3.b)} \hskip 20mm $2^{\ell({\underline d}_i)} > \sum_{j=1}^s \deg_{T_i}(P_j)$, \hskip 3mm $i=1,\ldots,k$.
\vskip 2mm

\noindent
{\rm (c)} If $k=1$ and $\underline {d}_1 = \underline d = (d_1,\ldots,d_n)$, 
the same holds with {\rm (3.a)}
replaced by 
%the conclusion holds under this condition on $\underline {d}_1 = (d_1,\ldots,d_n)$:
\vskip 1mm

\noindent
{\rm (3.c)} \hskip 20mm $\sum_{j=1}^n d_j > \sum_{i=1}^s\deg_{\underline{Y}}(P_i)$.
\vskip 0,7mm
\noindent
and, if additionally ${\mathcal Z}$ is a {near UFD}, by 
\vskip 1mm

\noindent 
{\rm (3.c')}$ \hskip 20mm\sum_{j=1}^n d_j  > \max_{1 \leq i \leq s}\deg_{\underline{Y}}(P_i)$.

%\centerline{$\prod_{j=1}^n (d_j+1)> \sum_{i=1}^s \deg_{T}(P_i)$\hskip 5mm  or \hskip 5mm $\sum_{j=1}^n d_j \geq \sum_{i=1}^s\deg_{\underline{Y}}(P_i)+1$.}
%\vskip 2mm
%\noindent
%and, if ${\mathcal Z}$ is a {near UFD}, by
%\vskip 2mm
%\centerline{$2^{\prod_{j=1}^n (d_j+1)} > \sum_{i=1}^s \deg_{T}(P_i)$ \hskip 5mm or \hskip 5mm $\sum_{j=1}^n d_j  \geq \max_{1 \leq i \leq s}\deg_{\underline{Y}}(P_i)+1$.}
\end{thm}

\subsubsection{Remarks}
(a) The original Schinzel Hypothesis situation is excluded by assumption (1). Indeed, in that situation,  $P_1,\ldots,P_s$ are polynomials in $T$ and $\underline d=(0,\ldots,0)$, thus (1) fails.
%The conclusion  would then be that one can substitute constant polynomials $m\in {\mathcal Z}$ such that $P_1(m),\ldots,P_s(m)$ are irreducible in ${\mathcal Z}$. 
%For ${\mathcal Z}=\Z$, this is the conclusion of the original Schinzel Hypothesis.
%(if the product $P_1\cdots P_s$ were assumed to have no fixed divisor).
\vskip 1,5mm

\noindent
(b) \emph{The bounds in Theorem \ref{thm:main3}(3.b) are sharp}. Take $k=n=1$ for simplicity. Given any integer $d\geq 0$, we construct below an irreducible polynomial $P\in \Z[T,Y]$ such that $2^{\ell(d)} = \deg_T(P)$ and $P(M(Y),Y)$ is reducible in $\Z[Y]$ for all polynomials $M\in \Z[Y]$ of degree $d$. 
%Given an integer $\delta\geq 0$, 

Denote the subset of all polynomials in $\Z[Y]$ of degree $\leq d$ and with  coefficients $0$ or $1$ by ${\mathcal M}_d$. We have ${\rm card}({\mathcal M}_d)= 2^{d+1}$. Consider the following polynomial:
\vskip 1mm

\centerline{$\displaystyle P_0(T,Y) = \prod_{p(Y)\in {\mathcal M}_d} (T-p(Y))$.}
\vskip 1mm

\noindent
Then take for $P$ the polynomial  
\vskip 1mm

\centerline{$P(T,Y)= P_0(T,Y) + 2m$,}
\vskip 1mm

\noindent
where $m\in \Z$ is chosen such that $P$ is irreducible in $\Z[T,Y]$. To see that such an integer $m$ exists, apply Theorem \ref{thm:main3}(3.b) with $k=1$, $n=2$, $s=1$, with $P_1$ taken to be $Q(U,T,Y)= P_0(T,Y) +2 U$, with the parameter taken to be $U$ and ${\underline Y}=(T,Y)$, and with $\underline d=(0,0)$. The polynomial $Q$ is irreducible in $\Z[U,T,Y]$; condition (1) holds since $\deg_{T,Y}(Q) \geq 1$; and $\ell(0,0) = (0+1)(0+1)=1$ so $2^{\ell(0,0)}=2 > \deg_U(Q)=1$. Therefore one can indeed take $m\in \Z[T,Y]$ such that $\deg_T(m) = \deg_Y(m)=0$ and $Q(m,T,Y)$ irreducible in $\Z [T,Y]$. 
%(or, equivalently, is irreducible in $\Z [T,Y]$ and has no nonunit divisor in $\Z$). 

We have $\deg_T(P) =2^{d +1} = 2^{\ell(d)}$ and, by construction, for every polynomial $M(Y)\in \Z[Y]$ of degree $d$, we have $M(Y) \equiv p(Y) \pmod{2}$ for some $p(Y)\in {\mathcal M}_d$ and so the polynomial $P(M(Y),Y)$ is divisible by $2$. %thus showing that the conclusion of Theorem \ref{thm:main3} is false for $P$ and $d$.
%, \hbox{i.e.} when 
%\vskip 1mm
%
%\centerline{$2^{d+1} \leq \deg_T(P)$.}
\vskip 1,5mm

\noindent
(c) 
%This condition shows in particular they the $d_i$ should be suitably large either compared to the degrees in $T$ of the polynomials $P_1,\ldots,P_s$ or compared to their degrees in $\underline Y$. 
The bounds in Theorem \ref{thm:main3}(3.c) only depend on the degrees in $\underline Y$ of the polynomials $P_1,\ldots,P_s$. They can be interesting in certain situations, for example when $P_1,\ldots,P_s$ only depend on $\underline T$. The second bound (3.c') improves on the one given in \cite[Theorem 1.1]{BDN20b}.

\subsubsection{Principle of the proof}
%Theorem \ref{thm:main1} being valid for any number of parameters, we can follow a more direct strategy, compared to \cite{BDN20b}.
The full proof of Theorem \ref{thm:main3} is given in \S \ref{ssec:Schinzel-H}. The aimed explicitness of the statement may somehow hide the leading principle.
%which we present below.
%is somewhat hidden by the relative complexity due to the explicitness of the aimed statement. 
%The generality of Theorem , allows a more direct approach than in \cite{BDN20b}. 
Assume $k=1$ for simplicity.
In each of the polynomials $P_1,\ldots,P_s$ of Theorem \ref{thm:main3}, replace the parameter $T$ by the generic polynomial $M_{{\underline d}}(\underline Y)$ of degree $\leq d_{j}$ in $Y_j$ ($j=1,\ldots,n$):
\vskip 1mm

\centerline{$\displaystyle M_{{\underline d}}(\underline Y) = \sum_{\underline r=(r_1,\ldots,r_n)} \lambda_{\underline r} \hskip 1mm Y_1^{r_1} \cdots Y_n^{r_n}$;}
\vskip 1mm

\noindent
where the sum ranges over all $n$-tuples $\underline r=(r_1,\ldots,r_n)$ with $r_j\leq d_{j}$ ($j=1,\ldots,n$). The word ``generic'' means that the coefficients $\lambda_{\underline j}$ are indeterminates. 
The resulting polynomials
\vskip 1mm

\centerline{$F_i = P_i\hskip 2pt (M_{{\underline d}}(\underline Y),\underline{Y})\hskip 3mm (i=1,\ldots,s)$}
\vskip 1mm

\noindent
are polynomials in ${\mathcal Z}[\underline \Lambda, \underline Y]$, where $\underline \Lambda$ is the set consisting of all the indeterminates $\lambda_{\underline j}$. We will show that, if the integers $d_1,\ldots,d_n$ are suitably large,
the polynomials $F_1,\ldots, F_s$ are irreducible in ${\mathcal Q}[\underline \Lambda,\underline Y]$, of degree $\geq 1$ in $\underline Y$ (Lemma \ref{lemmaHS1}), and that their product $\prod_{i=1}^s F_i$ has no fixed divisor \hbox{w.r.t.} $\underline{\Lambda}$ (Lemma \ref{lemmaHS2}). We will then be able to use the integrally Hilbertian property to specialize the indeterminates from $\underline{\Lambda}$ and conclude the proof.

%This final section is mostly devoted to  Theorem \ref{thm:hypschistronggen}, which generalizes Theorem \ref{thm:schinzel-main}. The proof  occupies all of \S \ref{ssec:proof-thm5-1}.
%\S \ref{ssec:schinzel-rings} is devoted to Theorem \ref{thm:schinzel1=>Schinzelk}.

%\ref{thm:schinzel:1-equiv-k}

%The final \S \ref{ss:thm13} and \S \ref{ssec:coprime-Schinzel} are devoted to some consequences of our techniques over the ring $\Z$: Theorem \ref{lem-SH1} and its generalization given in Theorem \ref{thm:hypschistrong}, and some improvement on the so-called coprime Schinzel Hypothesis.
%\vskip 0,5mm
%Fix an integral domain ${\mathcal Z}$ with fraction field $\mathcal{Q}$ and two tuples of indeterminates  ${\underline T}=(T_1,\ldots,T_k)$ and ${\underline Y}=(Y_1,\ldots,Y_n)$. 
%In the whole Section \S %\ref{sec:applications}, 
%\ref{SchH-general}, \S \ref{ssec:Schinzel-H}, \S \ref{ss:thm13}, \S \ref{ssec:coprime-Schinzel}, 
%the primality type is tacitly taken to be ${\mathcal P} = {\rm Nonunit}$. 

\subsection{Local condition preservation} \label{ssec:hypschistronggen}
As indicated in Section \ref{sec:intro}, our central Schinzel-type result -- Theorem \ref{thm:schinzel-main} -- has a second part showing that, under appropriate assumptions, absence of fixed divisors can also be preserved by specialization. Theorem \ref{thm:hypschistronggen} is a more precise version.
%, which is stated only for $\mathbb{Z}$, holds more generally for Dedekind domains -- particularly for rings of integers of number fields -- and also for near UFDs.

\iffalse
\begin{thm}\label{thm:hypschistrong} \textcolor{red}{Generalize to several parameters!} Assume that $\mathcal{Z}$ is an integrally Hilbertian ring that is a near UFD or a Dedekind domain (\hbox{e.g.} the ring of integers of a number field). Let $P_1(T),\dots,P_s(T)\in \mathcal{Z}[T]$ be $s$ polynomials, irreducible in $\mathcal{Q}[T]$, and such that $P_1\cdots P_s$ has no fixed divisor \hbox{w.r.t.} $T$ in ${\rm Nonunit}\hskip 1pt \mathcal{Z}$. Let $(d_1,\dots,d_n) \in \mathbb{N}^n$ be a \underline{nonzero} tuple.  
Then it is possible to substitute some polynomial $M(\underline{Y})$ with $\deg_{Y_j}(M)=d_j$ ($j=1,\dots,n$) for the parameter $T$ in the original polynomials $P_1,\dots,P_s$ in such a way that the resulting polynomials $P_i(M(\underline{Y}),\underline{Y})$ ($i=1,\dots,s$) are irreducible in $\mathcal{Z}[\underline{Y}]$, and that the product 
$P_1(M(\underline{Y}),\underline{Y})\cdots P_i(M(\underline{Y}),\underline{Y})$
 has no fixed divisor \hbox{w.r.t.} $\underline{Y}$ in ${\rm Nonunit}\hskip 1pt \mathcal{Z}$.
\end{thm}
\fi

\begin{thm} \label{thm:hypschistronggen}
Assume that $\mathcal{Z}$ is an integrally Hilbertian ring that is a near UFD or a Krull domain. Let $P_1(\underline T,\underline{Y}),\dots,P_s(\underline T,\underline{Y})\in \mathcal{Z}[\underline T,\underline Y]$ be $s$ polynomials, irreducible in $\mathcal{Q}[\underline T,\underline{Y}]$, and such that $P_1\cdots P_s$ has no fixed divisor \hbox{w.r.t.} the $(k+n)$-tuple $(\underline T,\underline{Y})$ in ${\rm Nonunit}\hskip 1pt \mathcal{Z}$.
\vskip 0,5mm    
    
\noindent
{\rm (a)} Let $A$ be a positive real number. Then the subset ${\mathcal S}^\sharp_{\mathcal{Z}[\underline Y]}(\underline P,A)\subset {\mathcal Z}[\underline Y]^k$ of  all $k$-tuples $\underline{M}(\underline{Y})=(M_1(\underline{Y}),\dots,M_k(\underline{Y}))$ such that 
\vskip 1mm

%\noindent
{\rm (i)} $P_i(\underline M(\underline Y),\underline Y)$ is irreducible in ${\mathcal Z}[\underline Y]$, $i=1,\ldots,s$,
\vskip 0,5mm

%\noindent
{\rm (ii)} the product $P_1(\underline{M}(\underline{Y}),\underline{Y})\cdots P_s(\underline{M}(\underline{Y}),\underline{Y})$ has no fixed divisor \hbox{w.r.t.} $\underline{Y}$ in ${\rm Nonunit}\hskip 1pt \mathcal{Z}$,
\vskip 0,5mm

%\noindent
{\rm (iii)} $\deg_{Y_j}(M_i) > A$, $j=1,\dots,n$, $i=1,\ldots,s$,
\vskip 0,5mm

\noindent
is Zariski-dense.
\vskip 1mm

\noindent
{\rm (b)} For $k=n=1$,  {\rm (a)} holds with $M(Y)$ further required to be monic and  {\rm (iii)} replaced by
\vskip 0,5mm

{\rm (iii)'} $\deg_{Y}(M_i) = d_i$, $i=1,\ldots,s$, where $d_1,\ldots,d_s$ are arbitrary prescribed integers $d_i\geq 1$.
\end{thm}

\section{Proofs of Schinzel-type results} \label{sec:proofs}

\S \ref{ssec:Schinzel-H}, \S \ref{ssec:proof-thm5-1}, \S \ref{ssec:pf-corollary1.9}, \S \ref{ssec:pf-schinzel1=>Schinzelk} respectively provide the proofs of Theorem \ref{thm:main3}, Theorem \ref{thm:hypschistronggen}, Corollary \ref{cor:Z[Y_0,...,Y_n]}, Theorem \ref{thm:schinzel1=>Schinzelk}.

\subsection{Proof of Theorem \ref{thm:main3}} \label{ssec:Schinzel-H}

\subsubsection{Setup of the proof} 
Fix $s$ polynomials $P_1,\ldots, P_s \in {\mathcal Z}[\underline T,\underline Y]$.
For each index $i=1,\ldots,k$, let ${\underline{Q_i}}= (Q_{i0}, Q_{i1},\ldots,Q_{i\ell_i}$), with $Q_{i0}=1$,
be a $(\ell_i+1)$-tuple of nonzero polynomials in ${\mathcal Z}[\underline{Y}]$, distinct up to multiplicative constants 
in $\mathcal{Q}^\times$ and let $\underline{\lambda_i}=(\lambda_{i0}, \lambda_{i1},\ldots,\lambda_{i\ell_i})$ 
($\ell_i \geq 0$) be a corresponding tuple of indeterminates. We consider here a more general than necessary to prove Theorem \ref{thm:main3} for which one can take ${\underline{Q_i}}$ consisting of \emph{all monic monomials} of degree $\leq d_{ij}$ in $Y_j$ ($j=1,\ldots,n$), and then $\ell_i+1 =\ell(\underline d_i)$, $i=1,\ldots,k$ (see Remark \ref{rmk:momomials-vs-general}).  

Assume that 
all the indeterminates $\lambda_{i,l}$ , $i=1,\ldots,k$, $l=0,\ldots,\ell_i$, are algebraically independent and denote by $\underline{\Lambda}$ the tuple obtained by concatenating all tuples $\underline{\lambda_1}, \ldots, \underline{\lambda_k}$. Consider then the polynomials:

\vskip 1mm

\centerline{$\displaystyle M_{\underline {Q_i}}({\underline{\lambda_i}}, \underline Y) = \sum_{l=0}^{\ell_i} \lambda_{i,l} \hskip 1mm Q_{i,l}\in {\mathcal Z}[\underline \Lambda,\underline Y], \hskip 3mm i=1,\ldots,k$.}
\vskip 1mm

\noindent
Finally replace in each of the original polynomials $P_1,\ldots,P_s$ each parameter $T_i$ by $M_{\underline {Q_i}}({\underline{\lambda_i}}, \underline Y)$ ($i=1,\ldots,k$). 
The resulting polynomials
\vskip 1,5mm

\centerline{$F_i = P_i\hskip 2pt (M_{\underline {Q_1}}({\underline{\lambda_1}}, \underline Y), \ldots, M_{\underline {Q_k}}({\underline{\lambda_k}}, \underline Y),Y_1, \ldots, Y_n),\hskip 3mm i=1,\ldots,s$}
\vskip 1,5mm

\noindent
are polynomials in ${\mathcal Z}[\underline{\Lambda}, \underline Y]$. 

\subsubsection{Irreducibility of $F_1,\ldots,F_s$} The goal is Lemma \ref{lemmaHS1}. We start with Lemma \ref{lem:spec}. Fix $i\in \{1,\dots, s\}$ and, for simplicity, write $\underline{\lambda}$ for $\underline{\lambda}_i$, $\ell$ for $\ell_i$ and $\underline{Q}$ for $\underline{Q}_i$. Let $\underline U$ be a tuple of new variables and let $A=\mathcal{Q}[\underline{U}]$ (with $\underline U$ possibly empty in which case $A=\mathcal{Q}$).

\begin{lem}\label{lem:spec}
Let
$$P(T,\underline{Y})=P_{\rho}(\underline{Y}) T^{\rho} +\cdots+ P_{1}(\underline{Y}) T + P_{0}(\underline{Y})\,\ (P_j(\underline{Y}) \in A[\underline{Y}], j=1,\dots,\rho)
$$ be an irreducible polynomial in $A[T,\underline{Y}]$ of degree $\rho$. Assume that either $\deg_{\underline{Y}}(P) \geq 1$ or ($\rho \geq 1$ and $\ell \geq 1$). Then $G(\underline{\lambda},\underline{Y})=P(M_{\underline{Q}}(\underline{\lambda},\underline{Y}),\underline{Y})$ is irreducible in $A[\underline{\lambda},\underline{Y}]$ and of degree $\geq 1$ in $\underline{Y}$.
\end{lem}
\begin{proof}
The irreducibility conclusion follows similarly to \cite[Lemma 2.1]{BDN20b}. Consider the ring automorphism $\varphi:A[\underline{\lambda},\underline{Y}] \rightarrow A[\underline{\lambda},\underline{Y}]$ that is the identity on $A[\lambda_1,\ldots,\lambda_\ell,\underline{Y}]$ and maps $\lambda_0$ to 
the polynomial $\lambda_0 +  \sum_{i=1}^{\ell} \lambda_i Q_i (\underline{X})$. Since $P(\lambda_{0},\underline{Y})$ is irreducible in $A[\lambda_{0},\underline{Y}]$, it is also irreducible in 
$A[\underline{\lambda},\underline{Y}]$. Hence $G(\underline{\lambda},\underline{Y})=\varphi(P(\lambda_0,\underline{Y}))$ is irreducible in $A[\underline \lambda, \underline{Y}]$. 
Regarding the degree conclusion, assume first $\ell\geq 1$. As a polynomial in $\lambda_1$, the leading coefficient of $G$ is $P_{\rho}(\underline{Y}) Q_1(\underline{Y})^{\rho}$. If $\rho\geq 1$, it is of positive degree in $\underline{Y}$ since $Q_1(\underline Y)$ is. If $\rho=0$, then $\deg_{\underline{Y}}(P) \geq 1$, so $G(\underline{\lambda},\underline{Y})= P_{0}(\underline Y)=P(T,\underline{Y})$ is of positive degree in $\underline Y$. 
Finally consider the case $\ell =0$. We then have $G(\underline{\lambda},\underline{Y}) = P(\lambda_0,\underline{Y})$. The result follows since, in this case, we also have $\deg_{\underline Y}(P) \geq 1$.
\end{proof}

Fix now $l\in \{1,\ldots,s\}$ and denote $P_l$ and $F_l$ by $P$ and $F$ in the following lemma.

\begin{lem} 
\label{lemmaHS1} Assume that $P$ is irreducible in $\mathcal{Q}[\underline{T},\underline Y]$ and satisfies
\vskip 1mm

\noindent
{\rm (1)} $\deg_{\underline Y}(P) \geq 1$ or $\ell_i \geq 1$ for $i=1,\ldots,k$. 
\vskip 1,5mm

\noindent
Then the polynomial $F$ is irreducible in ${\mathcal Q}[\underline{\Lambda},\underline Y]$ and of degree $\geq 1$ in $\underline Y$.
%and have no divisor in ${\mathcal Z}$.
\end{lem}

%\begin{lem} 
%\label{lemmaHS1} Assume $P_1,\ldots, P_s$ are irreducible in $\mathcal{Q}[\underline{T},\underline Y]$ and satisfy
%\vskip 1mm

%\noindent
%{\rm (*)} $\deg_{\underline Y}(P_l) \geq 1$ for $l=1,\ldots,s$ or $\ell_i \geq 1$ for $i=1,\ldots,k$. 
%\vskip 1,5mm

%\noindent
%Then the polynomials $F_1,\ldots,F_s$ are irreducible in ${\mathcal Q}[\underline{\Lambda},\underline Y]$ and of degree $\geq 1$ in $\underline Y$.
%and have no divisor in ${\mathcal Z}$.
%\end{lem}

%irreducible in $\mathcal{Q}[\underline{X},Y]$, primitive \hbox{w.r.t.} $D$, and of degree $\rho_i \geq 1$ in $Y$.

\begin{proof} Set ${\underline T}^\prime = (T_2,\ldots,T_k)$ and view  $P$ as in ${\mathcal Z}[{\underline T}^\prime][T_1,{\underline Y}]$ (with ${\underline T}^\prime$ the empty tuple if $k=1$). We distinguish two cases.
\vskip 1mm

\noindent
\noindent
{\it 1st case:} $\deg_{\underline Y}(P) \geq 1$.
Apply Lemma \ref{lem:spec} to the polynomial $P$, $T=T_1$ , $\underline Q=\underline{Q_1}$ and $A$ taken to be ${\mathcal Q}[{\underline T}^\prime]$ (i.e., $\underline{U}=\underline{T}'$).
%; note that assumption (*) for these data holds since $\deg_{({\underline T}^\prime,{\underline Y})}(P_l) \geq \deg_{{\underline Y}}(P_l)$ ($l=1,\ldots,s$). 
The resulting polynomial, say $F_1$, is a polynomial in ${\mathcal Z}[{\underline T}^\prime][{\underline{\lambda_1}}, {\underline Y}]$, which is irreducible in ${\mathcal Q}[{\underline T}^\prime][ {\underline{\lambda_1}},{\underline Y}]$ and such that $\deg_{{\underline Y}}(F_1) \geq 1$. The same lemma can in turn be applied to this polynomial $F_1$, viewed in ${\mathcal Q}[{\underline{\lambda_1}},{\underline T}^{\prime\prime}][T_2,{\underline Y}]$, where ${\underline T}^{\prime\prime}=(T_3,\ldots,T_k)$, and with $T=T_2$ and $\underline Q=\underline{Q_2}$. Proceeding inductively with each of the parameters $T_1,\ldots,T_k$ leads to the conclusion that $F$ is irreducible in ${\mathcal Q}[\underline{\Lambda},\underline Y]$ and of degree $\geq 1$ in $\underline Y$.

\vskip 1mm
\noindent
{\it 2nd case:} $\deg_{\underline Y}(P)=0$. Due to assumption (1), we then have $\ell_i \geq 1$ for $i=1,\ldots,k$. The polynomial $P$ is in ${\mathcal Q}[\underline T]$ and is irreducible in ${\mathcal Q}[\underline T]$. View it in ${\mathcal Q}[{\underline T}^\prime][T_1]$. If $\deg_{T_1}(P) = 0$, there is no specialization of $T_1$ to be performed, or, in other words, any specialization of $T_1$ to  $M_{\underline {Q_1}}({\underline{\lambda_1}},\underline Y)$ leaves the polynomial $P$ unchanged. As $P$ is not a constant polynomial, there is an index $i_0\in \{1,\ldots,s\}$ such that $\deg_{T_{i_0}}(P)\geq 1$, and we may choose $i_0$ to be minimal. After specializations of $T_1,\ldots,T_{i_0-1}$ that have left $P$ unchanged, we wish to specialize $T_{i_0}$. Lemma \ref{lem:spec} can be applied to the polynomial $P$, viewed in ${\mathcal Q}[T_{i_0+1},\ldots,T_k] [T_{i_0},\underline Y]$, with $T=T_{i_0}$ , $\underline Q=\underline{Q_{i_0}}$ and $\ell=\ell_{i_0}\geq 1$, with as a result a polynomial, say $F_{i_0} \in {\mathcal Z}[T_{i_0+1},\ldots,T_k] [\underline{\lambda_{i_0}},\underline Y]$, irreducible in ${\mathcal Q}[T_{i_0+1},\ldots,T_k] [\underline{\lambda_{i_0}},\underline Y]$ and of degree $\geq 1$ in $\underline Y$. Then, for the next parameter $T_{i_0+1}$, the argument of the 1st case can be applied to $F_{i_0}$, 
viewed in ${\mathcal Q}[T_{i_0+2},\ldots,T_k,\lambda_{i_0}] [{T_{i_0+1}},\underline Y]$, and then again the same argument can be applied for the remaining parameters to finally lead to the desired conclusion that $F$ is irreducible in ${\mathcal Q}[\underline{\Lambda},\underline Y]$ and of degree $\geq 1$ in $\underline Y$.
\end{proof}

\subsubsection{Partial conclusion}
\label{ssec:general-conclusion} Assume that all our polynomials $P_1,\ldots,P_s$ satisfy condition (1) from Lemma \ref{lemmaHS1}, \hbox{i.e.}, we have:

\vskip 1mm

\noindent
{\rm (1/{\rm all})} $\deg_{\underline Y}(P_l) \geq 1$ for $l=1,\ldots,s$ or $\ell_i \geq 1$ for $i=1,\ldots,k$. 
\vskip 1,5mm

\noindent
Then, from Lemma \ref{lemmaHS1},
the polynomials $F_1,\ldots,F_s$ satisfy conditions (Irred$/{\mathcal Q}(\underline \Lambda)$) and (Prim$/{\mathcal Q}[\underline \Lambda]$) from Section \ref{ssec:int-Hilb-ring}. By definition of ``integrally Hilbertian'', one  
obtains this conclusion.

\begin{prop} \label{prop:generalSchH}
Assume that ${\mathcal Z}$ is integrally Hilbertian.
%\hbox{w.r.t} the primality type {\rm Nonunit}. 
Let $P_1,\ldots,P_s\in {\mathcal Z}[\underline T, \underline Y]$, irreducible in ${\mathcal Q}[\underline T, \underline Y]$ and satisfying condition {\rm (1/{\rm all})}.
If in addition, the polynomials $F_1,\ldots,F_s$ satisfy condition $\hbox{\rm (NoFixDiv}/{\mathcal Z}[\underline \Lambda])$ \emph{from Section \ref{ssec:int-Hilb-ring}}, then the set of $k$-tuples of polynomials
\vskip 1,5mm

\centerline{$M_{\underline {Q_i}}({\underline{\theta_i}}, \underline Y) = \sum_{l=0}^{\ell_i} \theta_{i,l} \hskip 1mm Q_{i,l}\in {\mathcal Z}[\underline Y], \hskip 3mm (i=1,\ldots,k)$}
\vskip 1mm

\centerline{(corresponding to some $(\ell_i+1)$-tuples $(\theta_{i,0}, \ldots, \theta_{i,\ell_i}) \in {\mathcal Z}^{\ell_i+1}$)}

\vskip 1,5mm

\noindent
such that
\vskip 1,5mm

\noindent
{each polynomial $P_i\hskip 2pt (M_{\underline {Q_1}}({\underline{\theta_1}}, \underline Y), \ldots, M_{\underline {Q_k}}({\underline{\theta_k}}, \underline Y),\underline Y)$ is irreducible in ${\mathcal Z}[\underline Y]$, $i=1,\ldots,s$,}
\vskip 2mm

\noindent
is Zariski-dense in ${\mathcal P}ol_{{\mathcal Z},n,\underline{d}_1} \times \cdots \times {\mathcal P}ol_{{\mathcal Z},n,\underline{d}_k}$.

%it is possible to substitute some polynomials
%\vskip 1,5mm

%\centerline{$M_{\underline {Q_i}}({\underline{\theta_i}}, \underline Y) = \sum_{l=0}^{\ell_i} \theta_{i,l} \hskip 1mm Q_{i,l}\in {\mathcal Z}[\underline Y], \hskip 3mm (i=1,\ldots,k)$}
%\vskip 1mm

%\centerline{(for some elements $\theta_{i,l} \in {\mathcal Z}$) \hskip 20mm}

%\vskip 1,5mm

%\noindent
%for each parameter $T_i$ ($i=1,\ldots,k$) in the polynomials $P_1,\ldots,P_s$ in such a way that the resulting polynomials
%\vskip 1mm

%\centerline{$P_i\hskip 2pt (M_{\underline {Q_1}}({\underline{\theta_1}}, \underline Y), \ldots, M_{\underline {Q_k}}({\underline{\theta_k}}, \underline Y),\underline Y),\hskip 3mm i=1,\ldots,s$}
%\vskip 1mm

%\noindent
%are irreducible in ${\mathcal Z}[\underline Y]$.
\end{prop}

\begin{rmk} \label{rmk:momomials-vs-general}
In \S \ref{sssec:no-fixed-divisor}, we show that condition $\hbox{\rm (NoFixDiv}/{\mathcal Z}[\underline \Lambda])$ holds in the situation of Theorem \ref{thm:main3} for which the lists $\underline{Q_1},\ldots, \underline{Q_k}$ consist of all monomials of bounded degrees, and so can conclude the proof.
%, thanks to 
% Proposition \ref{prop:generalSchH}. 
It may be interesting to consider other situations corresponding to other lists of monomials; one needs then to check condition $\hbox{\rm (NoFixDiv}/{\mathcal Z}[\underline \Lambda])$ for these lists. An extreme example of this is when $\underline{Q_i}$ consists of the single constant monomial $Q_{i0}=1$ ($i=1,\ldots,k$). Condition  $\hbox{\rm (NoFixDiv}/{\mathcal Z}[\underline \Lambda])$ is then assumption $\hbox{\rm (NoFixDiv}/{\mathcal Z}[\underline T])$ from \S \ref{ssec:int-Hilb-ring} and Proposition \ref{prop:generalSchH} is Definition \ref{def:integrally-hilbertian}.
%Theorem \ref{thm:main1}.
\end{rmk}

\subsubsection{The No Fixed Divisor condition $\hbox{\rm (NoFixDiv}/{\mathcal Z}[\underline \Lambda])$}
\label{sssec:no-fixed-divisor}
The situation is that of Theorem \ref{thm:main3}. For each $i=1,\ldots,k$, an $n$-tuple ${\underline d}_i=(d_{i1},\ldots,d_{in})$ of nonnegative integers is given and $\underline{Q_i}$ is the tuple of all monic monomials of degree at most some given integer $d_{ij}$ in $Y_j$ (in some order with $Q_{i0}=1$);
we then have $\ell_i+1 = \ell(\underline d_i)=  \prod_{j=1}^n (d_{ij}+1)$. In particular, $\ell_i \geq 1$ exactly corresponds to $\underline d_i\not= (0,\ldots,0)$, so condition {\rm ((1/{\rm all}))} from \S \ref{ssec:general-conclusion} corresponds to condition (1) from Theorem \ref{thm:main3}.

Set $\Pi = P_1 \cdots P_s$ and use the abbreviation 
$\underline M_{\underline Q}( \underline{\Lambda},\underline Y)$ for the $k$-tuple 
\vskip 2mm
\centerline{$\underline M_{\underline Q}( \underline{\Lambda},\underline Y)=(M_{\underline {Q_1}}({\underline{\lambda_1}}, \underline Y), \ldots, M_{\underline {Q_k}}({\underline{\lambda_k}}, \underline Y))$.}
\vskip 2mm

\noindent
Then we have $F_1\cdots F_s= \Pi\hskip 1pt (\underline M_{\underline Q}(\underline{\Lambda},\underline Y),\underline Y)$. Similarly, given some ring $A$ and some ring morphism ${\mathcal Z}\rightarrow A$, for every $\underline{\Theta}=(\underline{\theta_1},\ldots,\underline{\theta_k)}\in \prod_{i=1}^k A^{\ell_i+1}$, we denote by $\underline M_{\underline Q}( \underline{\Theta},\underline Y)$ the $k$-tuple of polynomials in $A[\underline Y]$ obtained from $\underline M_{\underline Q}( \underline{\Lambda},\underline Y)$ by specializing the indeterminates in $\underline{\Lambda}$ to $\underline{\Theta}$.

\begin{lem} \label{lemmaHS2}  
Assume that $P_1,\ldots,P_s\in {\mathcal Z}[\underline T, \underline Y]$ are irreducible in ${\mathcal Q}[\underline T, \underline Y]$, the product $P_1\cdots P_s$ has no nonunit divisor in ${\mathcal Z}$ and condition {\rm (1)} from Theorem \ref{thm:main3} holds. Then, under the various assumptions 
%{\rm (3-1)--(3-4)} 
%on the ${\underline d_i}$ 
of statements {\rm (a)},{\rm (b)},{\rm (c)} of \hbox{\rm Theorem \ref{thm:main3}}, the polynomials $F_1,\ldots,F_s$ satisfy condition $\hbox{\rm (NoFixDiv}/{\mathcal Z}[\underline \Lambda])$, \hbox{i.e.}, the product $F_1\cdots F_s$  has no fixed divisor \hbox{w.r.t.} $\underline{\Lambda}$ in ${\rm Nonunit}\hskip 1pt {\mathcal Z}$. 
\end{lem}

%In particular, the product $P_1\cdots P_s$ has no fixed divisor \hbox{w.r.t.} $\underline{T}$ in ${\rm Nonunit}\hskip 1pt {\mathcal Z}[\underline Y]$. Thus the local condition of the Schinzel Hypothesis is a consequence of the assumptions of Theorem \ref{thm:main2}.

\begin{proof} Let $p$ be a nonunit of ${\mathcal Z}$, $p\not=0$. By assumption, the polynomial $\Pi(\underline T,\underline Y)$ is nonzero in $({\mathcal Z}/p{\mathcal Z})[\underline T,\underline Y]$. We start with case (b) of Theorem \ref{thm:main3}.
\vskip 1mm

\noindent
\emph{Case {\rm (b)}: ${\mathcal Z}$ is a near UFD and {\rm (3.b)} holds: $2^{\ell({\underline d}_i)} > \sum_{j=1}^s \deg_{T_i}(P_j)$} ($i=1,\ldots,k$). As ${\mathcal Z}$ is a near UFD, one may assume that $p$ is a prime. For each $i=1,\ldots,k$, the subset
\vskip 2mm

\centerline{$\{M_{\underline {Q_i}}({\overline{\underline{\theta_i}}}, \underline Y) \hskip 2pt | \hskip 2pt \overline{\underline{\theta_i}} \in ({\mathcal Z}/p{\mathcal Z})^{\ell_i+1}\} \subset ({\mathcal Z}/p{\mathcal Z})[\underline Y]$,}
\vskip 2mm

\noindent
is of cardinality $|{\mathcal Z}/p{\mathcal Z}|^{\ell_i+1} \geq 2^{\ell_i+1}$. 
Assumpiion (3.b) rewrites as $2^{\ell_i+1}> \deg_{T_i}(\Pi)$, $i=1,\ldots,k$. Since ${\mathcal Z}/p{\mathcal Z}$ is a domain, it follows that there exists a $k$-tuple 
\vskip 1mm

\centerline{$\underline M_{\underline Q}( \overline{\underline{\Theta}},\underline Y)=(M_{\underline {Q_1}}({\overline{\underline{\theta_1}}}, \underline Y), \ldots, M_{\underline {Q_k}}({\overline{\underline{\theta_k}}}, \underline Y)) \in (({\mathcal Z}/p{\mathcal Z})[\underline Y])^k$}
\vskip 1mm

\noindent
such that
$\Pi\hskip 1pt (\underline M_{\underline Q}(\overline{\underline{\Theta}},\underline Y),\underline Y)$ is nonzero in $({\mathcal Z}/p{\mathcal Z})[\underline Y]$.
 Lifting the $\overline{\underline{\theta_i}}$ to some elements of ${\mathcal Z}$ provides a point ${\underline{\Theta}} = (\underline {\theta_1},\ldots,\underline{\theta_k})\in \prod_{i=1}^k {\mathcal Z}^{\ell_i+1}$ such that 
 \vskip 1mm
 
 \centerline{$\underline M_{\underline Q}( \underline{\Theta},\underline Y)=(M_{\underline {Q_1}}({\underline{\theta_1}}, \underline Y), \ldots, M_{\underline {Q_k}}(\underline{\theta_k}, \underline Y)) \in ({\mathcal Z}[\underline Y])^k$}
 
 \vskip 1,5mm

 \noindent
  has the property that $\Pi\hskip 1pt (\underline M_{\underline Q}(\underline{\Theta},\underline Y),\underline Y)$ is not divisible by $p$.  
Conclude that $\underline{\Theta}$ is a special value of $\underline{\Lambda}$ that makes the polynomial $\Pi\hskip 1pt (\underline M_{\underline Q}(\underline{\Lambda},\underline Y),\underline Y)$ nonzero modulo $p$, \hbox{i.e.} $p$ is not a fixed divisor of $F_1\cdots F_s$ \hbox{w.r.t.} $\underline{\Lambda}$.
\vskip 2mm

\noindent
\emph{Case {\rm (a)}: ${\mathcal Z}$ is  integrally Hilbertian and \hbox{\rm (3.a)} holds: $\ell({\underline d}_i) > \sum_{j=1}^s \deg_{T_i}(P_j)$ ($i=1,\ldots,k$)}. The ring ${\mathcal Z}/p{\mathcal Z}$ is no longer integral in general. We will however be able to adjust the argument by working with a smaller subset of tuples of polynomials $\underline M_{\underline Q}(\underline{\Theta},\underline Y)$
than the one used in Case (b).

Specifically, for each $i=1,\ldots,k$, consider the following subset:
\vskip 1mm

\centerline{${\mathcal S}_i=\{ Q_{il} \hskip 2pt | \hskip 2pt  l=0,\ldots,\ell_i\} \subset {\mathcal Z}[\underline Y]$.}

\vskip 1mm

\noindent
We claim that the corresponding elements, regarded in $({\mathcal Z}/p{\mathcal Z})[\underline Y]$, have this property: any difference between two distinct such elements is not a zero divisor in $({\mathcal Z}/p{\mathcal Z})[\underline Y]$. Namely, if $P\in ({\mathcal Z}/p{\mathcal Z})[\underline Y]$ is nonzero and $Q$ is the largest monomial in $P$ for the lexicographical order, then the largest monomial in  the polynomial $Q_{il} P$ is $Q_{il} Q$. For two different $Q_{il}, Q_{i l^\prime} \in {\mathcal S}_i$, we have $Q_{il} Q\not= Q_{i l^\prime} Q$ and so $Q_{il} P\not= Q_{i l^\prime} P$, which proves the claim. Deduce next that the product of all nonzero differences $(Q_{il}- Q_{i l^\prime})$ with $Q_{il}, Q_{i l^\prime} \in {\mathcal S}_i$ is not a zero divisor in $({\mathcal Z}/p{\mathcal Z})[\underline Y]$.

 It follows from this conclusion that a nonzero polynomial in one variable of degree $\leq \ell_i$ and with coefficients in $({\mathcal Z}/p{\mathcal Z})[\underline Y]$ cannot vanish at every one of the $\ell_i+1$ elements $Q_{il}$ with $ l=0,\ldots,\ell_i$. Indeed the determinant of the linear system corresponding to the $\ell_i+1$ vanishing conditions -- a Van Der Monde determinant -- is precisely the product of differences in $({\mathcal Z}/p{\mathcal Z})[\underline Y]$ considered above. By construction it is not a zero divisor in $({\mathcal Z}/p{\mathcal Z})[\underline Y]$, and so only the zero polynomial could be a  solution to the system, which is excluded.

Conclude: as $\ell_i=\ell(\underline d_i) - 1 \geq \deg_{T_i}(\Pi)$ for each $i=1,\ldots,k$, there exist $l_i\in \{0,\ldots,\ell_i\}$ ($i=1,\ldots,k$) such that the $k$-tuple $\underline Q_{\underline l}=(Q_{1l_1},\ldots,Q_{kl_k}) \in ({\mathcal Z}[\underline Y])^k$ satisfies the following.
\vskip 1mm

\centerline{$\Pi\hskip 1pt (\underline Q_{\underline l},\underline Y)$ is not divisible by $p$.}
\vskip 1mm

\noindent
Conclude as in case (b) that $p$ is not a fixed divisor of $F_1\cdots F_s$ \hbox{w.r.t.} $\underline{\Lambda}$.
\vskip 2mm

\noindent
\emph{Case {\rm (c)}: ${\mathcal Z}$ integrally Hilbertian, $k=1$ and \hbox{\rm (3.c)} holds
$\sum_{j=1}^n d_j > \sum_{i=1}^s \deg_{\underline Y}(P_i)$.}
Write $T$ for $T_1$, $\underline \lambda$ for $\underline{\lambda_1}$ and $\ell$ for $\ell(\underline d)-1$. We may assume that $r=\deg_T(\Pi) \geq 1$. Set
\vskip 1mm

\centerline{$\Pi= P_1 \dots P_s=Z_0(\underline{Y})+Z_1(\underline{Y})T+\dots+Z_r(\underline{Y})T^r$.}
\vskip 1mm

\noindent
Assume on the contrary that $F_1\cdots F_s$ has a fixed divisor \hbox{w.r.t.} $\underline{\lambda}$, \hbox{i.e.}, there is a nonunit $p\in \mathcal{Z}$, $p\not=0$, such that 
\vskip 1mm

\centerline{$(F_1\cdots F_s)(\underline{m},\underline{Y})\equiv 0 \pmod{p}$\hskip 2mm for all $\underline{m} \in \mathcal{Z}^{\ell+1}$.}
\vskip 1mm

\noindent
In particular, for $\underline m$ corresponding to the monomial $Y_1^{d_1}\dots Y_n^{d_n}$, we obtain:
\vskip 1,5mm

\centerline{$Z_0(\underline{Y})+Z_1(\underline{Y})Y_1^{d_1}\dots Y_n^{d_n}+Z_2(\underline{Y})Y_1^{2d_1}\dots Y_n^{2d_n}+\dots+Z_r(\underline{Y})Y_1^{r d_1}\dots Y_n^{r d_n}\equiv 0 \pmod{p}$.}
\vskip 1,5mm

%Consider the unique $\underline{m}_0 \in \mathcal{Z}^{N_{\underline{d}}}$ such that $M_{\underline{d}}(\underline{m}_0,\underline{Y})=Y_1^{d_1}\dots Y_n^{d_n}$. Then

%\begin{equation}\label{eq:condgmo1}
%Z_0(\underline{Y})+Z_1(\underline{Y})Y_1^{d_1}\dots Y_n^{d_n}+Z_2(\underline{Y})Y_1^{2d_1}\dots Y_n^{2d_n}+\dots+Z_r(\underline{Y})Y_1^{r d_1}\dots Y_n^{r d_n}\equiv 0 \pmod{p}.
%\end{equation}

\noindent
We claim that for every $0 \leq i<j \leq r$ such that $Z_i(\underline{Y})$ and $Z_j(\underline{Y})$ are nonzero, the nonzero monomials appearing in $Z_i(\underline{Y})Y_1^{id_1}\dots Y_n^{id_n}$ are different from the ones in $Z_j(\underline{Y})Y_1^{jd_1}\dots Y_n^{j d_n}$. Indeed for nonzero monomials $A$ and $B$ in $Z_i(\underline{Y})Y_1^{id_1}\dots Y_n^{id_n}$ and $Z_j(\underline{Y})Y_1^{jd_1}\dots Y_n^{j d_n}$ respectively, we have:
\begin{align*}
\deg_{\underline{Y}}(A) &\leq \deg_{\underline{Y}}(Z_i)+i\left(\sum_{i=1}^n d_i\right)\\
&\leq \deg_{\underline{Y}}(\Pi)+i\left(\sum_{i=1}^n d_i\right)\\
&= \sum_{i=1}^s\deg_{\underline{Y}}(P_i)+i\left(\sum_{i=1}^n d_i\right)\\
&<\left(\sum_{i=1}^n d_i\right)+i\left(\sum_{i=1}^n d_i\right) \quad \textrm{(by assumption)}\\
&\leq j\left(\sum_{i=1}^n d_i\right)\leq \deg_{\underline{Y}}(B).
\end{align*}
Thus $A \neq B$.
Hence, for each $i=1,\ldots, r$, we have $Z_i(\underline{Y})Y_1^{i\cdot d_1}\dots Y_n^{i \cdot d_n}\equiv 0 \pmod{p}$, and so $Z_i(\underline{Y}) \equiv 0 \pmod{p}$. Therefore $p$ divides $\Pi$, a contradiction.

%Therefore, $P_1\dots P_s$ is not primitive \hbox{w.r.t.} $\mathcal{Z}$, a contradiction. Consequently, $G(\underline{\lambda},\underline{Y})$ has no fixed divisor w.r.t. $\underline{\lambda}$ in ${\rm Nonunit}\hskip 1pt\mathcal{Z}$.
\vskip 2mm

Finally, assume as in the special case of Theorem \ref{thm:main3}(c) that \emph{$\mathcal{Z}$ is a near UFD and that \hbox{\rm (3.c')} holds: $\sum_{j=1}^nd_j> \max_{1 \leq i \leq s}\deg_{\underline{Y}}(P_i)$.} If $p$ is a fixed divisor of $F_1\cdots F_s$ \hbox{w.r.t.} $\underline \lambda$, then $\Pi(Y_1^{d_1}\dots Y_n^{d_n},\underline{Y})\equiv 0 \pmod{p}$. Without loss of generality, one may assume that $p$ is prime. Since $(\mathcal{Z}/p\mathcal{Z})[\underline{Y}]$ is an integral domain, there exists $1 \leq j \leq s$ such that 
\vskip 1mm

\centerline{$P_{j}(Y_1^{d_1}\dots Y_n^{d_n},\underline{Y})\equiv 0 \pmod{p}$.}
\vskip 1mm

\noindent
In the second part of the argument above, with $P_j$ replacing $P_1\cdots P_s$, we conclude that $P_j \equiv 0 \pmod{p}$, a contradiction.
\end{proof}
 %(P_j,F_j(\underline{m_0},\underline{Y}))$ in place of $(P_1\dots P_s,G(\underline{m_0},\underline{Y}))$, 
 %Consequently, $G(\underline{\lambda},\underline{Y})$ has no fixed divisor w.r.t. $\underline{\lambda}$

\subsubsection{End of proof of Theorem \ref{thm:main3}}
\label{ssec:end-of-pf}
Lemmas \ref{lemmaHS1} and \ref{lemmaHS2} guarantee that Proposition \ref{prop:generalSchH} can be applied under the assumptions of Theorem \ref{thm:main3}. Its conclusion yields the requested conclusion of Theorem \ref{thm:main3}.

% in fact obtain more.
%For every $n$-tuple $\underline d = (d_1,\ldots, d_n)$ of nonnegative integers, denote the set of polynomials $M \in {\mathcal Z}[\underline Y]$ such that $\deg_{Y_j}(M) \leq d_j$ 
%($j= 1,\ldots,n$), by ${\mathcal P}ol_{{\mathcal Z},n,\underline{d}}$. It is an affine space over ${\mathcal Z}$: the coordinates correspond to the coefficients. The integrally Hilbertian definition finally applied in the proof of Theorem \ref{thm:main3} leads to the conclusion that 
%\vskip 1mm

%\noindent
%(*) \emph{the sets of $k$-tuples $(M_1(\underline Y),\ldots,M_k(\underline Y))$ of polynomials in ${\mathcal Z}[\underline Y]$ with $\deg_{Y_j}(M_i) \leq d_{ij}$ ($i=1,\ldots,k$, $j=1,\ldots,n$) is Zariski-dense in the affine space ${\mathcal P}ol_{{\mathcal Z},n,\underline{d_1}}\times \cdots \times {\mathcal P}ol_{{\mathcal Z},n,\underline{d_k}}$.}

\subsubsection{Final comments} \hskip 1mm
\vskip 0,5mm

\noindent
(a) For statement (b) and second part of statement (c) of Theorem \ref{thm:main3}, for which ${\mathcal Z}$ is assumed to be a near UFD, only the condition (\ref{def:near-UFD}-2) -- every nonunit has at least one prime divisor -- is used in the proof. Thus this sole condition can more generally replace the near UFD assumption.
\vskip 1mm

\noindent
(b) The argument for Case (b) also works in the general case of an integrally Hilbertian ring 
%(Case (a)) 
\emph{provided that} the local condition is strengthened to assume that the product $P_1\cdots P_s$ has no divisor in ${\rm Spec}^\ast \hskip 1pt {\mathcal Z}$ (instead of just ${\rm Nonunit} \hskip 1pt {\mathcal Z}$). Hence the better bound from Case (b) also holds in the general Case (a) under this stronger local condition.

The argument should be adjusted as follows. Starting from a non unit $p\in {\mathcal Z}$, $p\not=0$, consider a maximal ideal ${\frak p}\subset {\mathcal Z}$ containing $p$. By the stronger local condition, the polynomial $\Pi(\underline T,\underline Y)$ is nonzero in $({\mathcal Z}/{\frak p})[\underline T,\underline Y]$. Then,
one can work as we did in Case (b) with ${\frak p}$ replacing $p{\mathcal Z}$ to construct a point ${\underline{\Theta}} = (\underline {\theta_1},\ldots,\underline{\theta_k})\in \prod_{i=1}^k {\mathcal Z}^{\ell_i+1}$ such that 
 \vskip 1mm
 
 \centerline{$\underline M_{\underline Q}( \underline{\Theta},\underline Y)=(M_{\underline {Q_1}}({\underline{\theta_1}}, \underline Y), \ldots, M_{\underline {Q_k}}(\underline{\theta_k}, \underline Y)) \in ({\mathcal Z}[\underline Y])^k$}
 
 \vskip 1,5mm

 \noindent
 $\Pi\hskip 1pt (\underline M_{\underline Q}(\underline{\Theta},\underline Y),\underline Y)$ is not divisible by ${\mathcal P}$, and so \emph{a fortiori} not by $p$.

\subsection{Proof of Theorem \ref{thm:hypschistronggen}} \label{ssec:proof-thm5-1}
\noindent
\vskip 1mm
\noindent
\emph{\textbf{General setting.}}
Since the primality type ${\rm Nonunit}\hskip 1pt \mathcal{Z}$ is equivalent to:
$$
\mathcal{P}=
\begin{cases}
{\rm Irred}\hskip 1pt \mathcal{Z} & \textrm{if\,\,} \mathcal{Z}\,\,\textrm{is a Krull domain},\\
{\rm Prime}\hskip 1pt \mathcal{Z} & \textrm{if\,\,} \mathcal{Z}\,\,\textrm{is a near UFD},
\end{cases}
$$
as before, we can carry out the proof using the primality $\mathcal{P}$. Recall from Proposition \ref{lem:schcases} that ${\mathcal Z}$ satisfies the properties (NVA) and (SF) \hbox{w.r.t.} ${\mathcal P}$ for the following choice of ${\mathcal E}$ in condition \eqref{eq:exismaide} of Definition \ref{def:schproide}:
$$
\mathcal{E}=
\begin{cases}
\s1(\mathcal{Z}) & \textrm{if\,\,} \mathcal{Z}\,\,\textrm{is a Krull domain},\\
{\rm Prime}\hskip 1pt \mathcal{Z} & \textrm{if\,\,} \mathcal{Z}\,\,\textrm{is a near UFD}.
\end{cases}
$$ 

\noindent
Without loss of generality, we may assume that $\deg_{\underline T}(P_1 \cdots P_s) \geq 1$. Write
$$
P_1 \cdots P_s=a_0\underline{T}^{\underline w_0}\underline{Y}^{\underline z_0}+ \dots+a_t\underline{T}^{\underline w_t}\underline{Y}^{\underline z_t},
$$ with $a_0 \neq 0$, and
$$
(\underline w_0, \underline z_0)> \dots > (\underline{w}_t,\underline{z}_t)
$$ in lexicographic order $T_1 \succ T_2 \succ \cdots \succ T_k \succ Y_1 \succ \cdots \succ Y_n$. Choose nonzero tuples $\underline{d}_1,\dots,\underline{d}_k \in \mathbb{N}^n$
such that, for every $c=1,\dots,t$,
\begin{equation}\label{eq:condZg}
\sum_{i=1}^n\left(\sum_{r=1}^k w_{c,r} d_{r,i}+z_{c,i}\right)< \sum_{i=1}^n\left(\sum_{r=1}^k w_{0,r} d_{r,i}+z_{0,i}\right).
\end{equation}
In case (b): $k=n=1$ of Theorem \ref{thm:hypschistronggen}, it is easy to see that $\underline{d}_1 \in \N$ can be any integer $\geq 1$. In the general case (a), we choose $\underline{d}_1,\dots, \underline{d}_k \in \N^n \cap ]A,+\infty[^n$ so that this condition is satisfied {(such a choice is easily seen to be possible).}

Set $\Delta:=\sum_{i=1}^n\left(\sum_{\tau=1}^k w_{0,r} d_{r,i}+z_{0,i}\right)$. Define
$$
\begin{cases}
\mathcal{S}=({\rm Supp}(a_0) \cap \mathcal{E}) \cup (\Gamma_\Delta \cap \mathcal{E})\\
\Omega=\Sigma(\mathcal{S})_{\mathcal{E}} \cap \mathcal{P}
\end{cases}
$$
By Lemma \ref{lem:interrestgen} and the support finiteness property, $\mathcal{S}$ is finite; hence $\Omega$ is finite. Since $P_1\cdots P_s$ has no fixed divisor \hbox{w.r.t.} $(\underline{T},\underline{Y})$ in $\mathcal{P}$, by the nonvanishing approximation property, there exists $(\underline{\theta},\underline{\alpha}) \in \mathcal{Z}^{k+n}$ such that 
\begin{equation}\label{eq:nonvaneee}
(P_1\cdots P_s)(\underline{\theta},\underline{\alpha}) \not \equiv 0 \pmod{\mathfrak{p}},
\end{equation} for all $\mathfrak{p} \in \Omega$. Let $\omega \in \mathcal{Z}$ be a generator of the principal ideal $\prod_{\mathfrak{p} \in \Omega} \mathfrak{p}$, which is contained in $\bigcap_{\mathfrak{p} \in \Omega} \mathfrak{p}$.
\begin{claim}\label{claim:invardiv}
\begin{enumerate}
\item If an ideal $\mathfrak{q} \in \mathcal{E}$ divides $a_0  \omega^\ell$,  then $\mathfrak{q} \in \mathcal{S}$.\label{finl1}
\item If an ideal $\mathfrak{p} \in \mathcal{P}$ divides $a_0 \omega^\ell$ for some $\ell \geq 0$, then $\mathfrak{p} \in \Omega$.\label{finl2}
\end{enumerate}
\end{claim}
\begin{proof}[Proof of Claim \ref{claim:invardiv}]
 
We have: 
$$
 ({\rm Supp}(a_0) \cup {\rm Supp}(\omega)) \cap \mathcal{E}
=({\rm Supp}(a_0) \cap \mathcal{E}) \bigcup \left(\bigcup_{\mathfrak{p} \in \Omega} {\rm Supp}(\mathfrak{p}) \cap \mathcal{E}\right)\subset \mathcal{S}.
$$ 

\noindent
Assertions (\ref{finl1}) and (\ref{finl2}) follow then immediately from the definitions; for (\ref{finl2}), note that 
$$
{\rm Supp}(\mathfrak{p}) \cap \mathcal{E} \subset ({\rm Supp}(a_0) \cup {\rm Supp}(\omega)) \cap \mathcal{E} \subset \mathcal{S}.
$$
\end{proof}

For $r=1,\dots,k$, consider all monomials $Q_{r,0},\dots,Q_{r,\ell_r}$ of degree $\leq d_{r,j}$ in $Y_j$ ($j=1,\dots,n$), where $Q_{r,0}=1$ and $Q_{r,\ell_r}=Y_1^{d_{r,1}}\cdots Y_n^{d_{r,n}}$. For each $r=1,\dots, k$, fix a tuple $\underline{\lambda}_r=(\lambda_{r,0},\dots,\lambda_{r,\ell_r-1})$ of indeterminates, and define
\begin{equation}\label{eq:genericRr}
R_r(\underline{\lambda}_r,\underline{Y})=Y_1^{d_{r,1}}Y_2^{d_{r,2}}\cdots Y_n^{d_{r,n}}+\sum_{i=0} ^{\ell_r-1} \lambda_{r,i} Q_{r,i}.
\end{equation}
be the corresponding generic polynomial. Set
$$
\underline{\lambda}=(\underline{\lambda}_1,\dots,\underline{\lambda}_k).
$$
For each $i=1,\dots,s$, define
$$
F_i(\underline{\lambda},\underline{Y})=P_i\left(\theta_1+ \omega \cdot R_1(\underline{\lambda}_1,\underline{Y}),\dots,\theta_k+ \omega \cdot R_k(\underline{\lambda}_k,\underline{Y}),\underline{Y}\right).
$$
\begin{claim}\label{claim:specterm}
The coefficient of $\prod_{i=1}^nY_i^{\sum_{r=1}^k w_{0,r} d_{r,i}+z_{0,i}}$ in $F_1 \cdots F_s$ equals $a_0 \omega^{\Gamma}$, where $\Gamma:=\sum_{i=1}^n w_{0,i}$.
\end{claim}
\begin{proof}
This easily follows from expanding
$F_1\cdots F_s
$ using \eqref{eq:condZg}.
\end{proof}
\vskip 1mm
\noindent
\emph{\textbf{Intermediate step.}}
We wish to show that the polynomials $F_1,\ldots,F_s$ satisfy the conditions $({\rm Irred}/\mathcal{Q}(\underline{\lambda}))$, $({\rm Prim}/\mathcal{Q}[\underline{\lambda}])$ and (${\rm NoFixDiv}/\mathcal{Z}[\underline{\lambda}]_{\mathcal{P}}$).
\begin{claim}\label{claim:irrfigencas}
For every $i=1,\dots,s$, the polynomial $F_i$ is irreducible in $\mathcal{Q}[\underline{\lambda},\underline{Y}]$.
\end{claim}
\begin{proof}[Proof of Claim \ref{claim:irrfigencas}]
The proof is the same as that of Lemma \ref{lemmaHS1}. {Fix $i\in \{1,\ldots,s\}$.} Consider the ring automorphism $\varphi:\mathcal{Q}[\underline{\lambda},\underline{Y}] \rightarrow \mathcal{Q}[\underline{\lambda},\underline{Y}]$ that is the identity on $\mathcal{Q}[\underline{\lambda} \setminus \{\lambda_{r,0}\mid r=1,\dots,k \},\underline{Y}]$ and, {for $r=1,\ldots,k$, maps $\lambda_{r,0}$,} to 
the polynomial $\theta_r+\omega \cdot R_r(\underline{\lambda}_r,\underline{Y})$. Since $P_i(\lambda_{1,0},\dots,\lambda_{k,0})$ is irreducible in $\mathcal{Q}[\lambda_{1,0},\dots,\lambda_{k,0},\underline{Y}]$, it remains irreducible in $\mathcal{Q}[\underline{\lambda},\underline{Y}]$. Therefore, 
$$F_i(\underline{\lambda},\underline{Y})=\varphi(P_i(\lambda_{1,0},\dots,\lambda_{k,0}))
$$ is also irreducible in $\mathcal{Q}[\underline \lambda, \underline{Y}]$. 
\end{proof}
\begin{claim}\label{claim:nofixprdivgenca}
The product $F_1\cdots F_s$ has no fixed divisor \hbox{w.r.t.} $\underline{\lambda}$ in $\mathcal{P}$.
\end{claim}
\begin{proof}[Proof of Claim \ref{claim:nofixprdivgenca}]
Assume on the contrary that $F_1 \cdots F_s$ has a fixed divisor $\mathfrak{p}\in \mathcal{P}$ \hbox{w.r.t.} $\underline{\lambda}$, that is, $$
(F_1\cdots F_s)(\underline{m},\underline{Y})\equiv 0 \pmod{\mathfrak{p}},
$$
for all $\underline{m} \in \mathcal{Z}^{\ell}$. In particular, $(F_1 \cdots F_s)(\underline{0},\underline{Y}) \equiv 0 \pmod{\mathfrak{p}}$. By Claim \ref{claim:specterm}, $\mathfrak{p}$ divides $a_0 \omega^\Gamma$. Therefore, by Claim \ref{claim:invardiv}, $\mathfrak{p} \in \Omega$ and $\mathfrak{p}$ divides $\omega$. This implies that $(P_1\dots P_s)(\underline{\theta},\underline{\alpha}) \equiv (F_1\dots F_s)(\underline{0},\underline{\alpha}) \equiv 0 \pmod{\mathfrak{p}}$, contradicting \eqref{eq:nonvaneee}. Consequently, $F=F_1 \cdots F_s$ has no fixed divisor \hbox{w.r.t.} $\underline{\lambda}$ in $\mathcal{P}$.
\end{proof}
\vskip 1mm
\noindent
\emph{\textbf{End of proof of Theorem \ref{thm:hypschistronggen}.}}
Since $\mathcal{Z}$ is integrally Hilbertian, for 
$$\underline{v}=(\underline{v}_1,\dots,\underline{v}_k) 
$$ in a Zariski-dense set $\mathcal{H} \subset  \mathcal{Z}^{\ell_1+\dots+\ell_k}$, the polynomials
$$
G_i(\underline{Y})=F_i(\underline{v},\underline{Y})=P_i(\theta_1+\omega \cdot R_1(\underline{v}_1,\underline{Y}),\dots, \theta_k+\omega \cdot R_k(\underline{v}_k,\underline{Y}),\underline{Y}) \in \mathcal{Z}[\underline{Y}]\quad i=1,\dots,s
$$ are irreducible in $\mathcal{Z}[\underline{Y}]$. For $\underline{v} \in \mathcal{H}$, Let 
$$
\underline{M}_{\underline{v}}(\underline{Y})=(M_{\underline v, 1}(\underline{Y}),\dots,M_{\underline v,k}(\underline{Y}))=(\theta_1+\omega \cdot R_1(\underline{v},\underline{Y}),\dots,\theta_k+\omega \cdot R_k(\underline{v},\underline{Y})) \in \mathcal{Z}[\underline{Y}]^k.
$$ To finish the proof of Theorem \ref{thm:hypschistronggen}, it remains to show the following.
\begin{claim}\label{claim:nothavifd}
The product $G_1\cdots G_s$ has no fixed divisor \hbox{w.r.t.} $\underline{Y}$ in $\mathcal{P}$.
\end{claim}
\begin{proof}[Proof of Claim \ref{claim:nothavifd}]
Assume on the contrary that $G_1\cdots G_s$ has a fixed divisor $\mathfrak{p} \in \mathcal{P}$ \hbox{w.r.t.} ${\underline Y}$. 

First we wish to show that ${\rm Supp}(\mathfrak{p}) \cap \mathcal{E} \subset \mathcal{S}$. Let $\mathfrak{q} \in {\rm Supp}(\mathfrak{p}) \cap \mathcal{E}$. Note that $\mathfrak{q}$ is a prime ideal and is a fixed divisor of $G_1\cdots G_s$ \hbox{w.r.t.} $\underline Y$. By Lemma \ref{lem:zero}, either $|\mathcal{Z}/\mathfrak{q}| \leq \max_{1 \leq j \leq n} \deg_{Y_j}(G_1\cdots G_s)$ or $\mathfrak{q} \in {\rm Div}_{\mathcal{Z}}(G_1 \cdots G_s)$. In the first case, 
$$|\mathcal{Z}/\mathfrak{q}| \leq \max_{1 \leq j \leq n} \deg_{Y_j}(G_1\cdots G_s) \leq \sum_{i=1}^n\left(\sum_{r=1}^k m_{0,r} d_{r,i}+z_{0,i}\right)=\Delta,$$
so $\mathfrak{q} \in \mathcal{S}$. Suppose now we are in the second case, that is, $\mathfrak{q}$ divides $G_1 \cdots G_s$. By Claim \ref{claim:specterm}, $\mathfrak{q}$ divides $a_0 \omega^\Gamma$. Therefore, by Claim \ref{claim:invardiv}, we have $\mathfrak{q} \in \mathcal{S}$. Consequently, ${\rm Supp}(\mathfrak{p}) \cap \mathcal{E} \subset \mathcal{S}$. 

The conclusion above implies that $\mathfrak{p} \in \Omega$. In particular, $\mathfrak{p}$ divides $\omega$. Therefore
$$
(P_1\cdots P_s)(\underline{\theta},\underline{\alpha}) \equiv (G_1\cdots G_s)(\underline{\alpha}) \equiv 0 \pmod{\mathfrak{p}},
$$
contradicting \eqref{eq:nonvaneee}. Consequently, $G_1 \cdots G_s$ has no fixed divisor in $\mathcal{P}$ w.r.t. $\underline{Y}$.
\end{proof}

%\begin{defn} 
%Let ${\mathcal Z}$ be domain such that the fraction field ${\mathcal Q}$ is equipped with several sets $S_1,\ldots,S_n$ satisfying the product formula. Let $H_1,\ldots,H_n$ be the corresponding Weil heights. The ring  ${\mathcal Z}$ is called a \emph{Schinzel ring \hbox{w.r.t. the heights $H_1,\ldots,H_n$}} if the following holds. Let $\underline P$ be a finite set of polynomials $P_1,\ldots,P_s \in {\mathcal Z}[\underline T]$,  irreducible in $\mathcal{Q}[\underline T]$
%and such that the product $P_1\cdots P_s$ has no fixed divisor in ${\mathcal Z}$ \hbox{w.r.t.} $\underline T$.
%Let $A$ be a positive real number. Then the following subset of ${\mathcal Z}^k$ is Zariski-dense:
%\vskip 1mm
%
%\centerline{\hskip 6mm ${\mathcal %S}_{\mathcal{Z}}(\underline P,A)
%= \left\{\underline t = (t_1,\ldots,t_k)\in {\mathcal Z}^k \hskip 1pt \left\vert\hskip 1pt \begin{array}{c}
%     P_1(\underline t), \ldots, P_s(\underline t) \ \hbox{\rm are irreducible in ${\mathcal Z}$}, \hbox{and} \hfill \hfill \\
%     H_i(t_j) \geq A,\ \hbox{for all } j=1,\ldots,k \ \hbox{and}\  i=1,\ldots,n. \hfill \hfill \\
%    \end{array}\right. \right\} 
%    $.}
%\end{defn}

%\begin{cor} 
%Let ${\mathcal Z}$ be an arbitrary domain. Then the polynomial ring ${\mathcal Z}[Y_0, Y_1,\ldots,Y_n]$ with $n\geq 1$ is a Schinzel ring \hbox{w.r.t.} the partial degree Weil heights $H_0, H_1,\ldots,H_n$. 
%\end{cor}

\subsection{Proof of Corollary \ref{cor:Z[Y_0,...,Y_n]}} \label{ssec:pf-corollary1.9}
Consider the ring ${\mathcal Z}[Y_0, Y_1,\ldots,Y_n]$ ($n\geq 1$) equipped with the $(n+1)$ Weil heights induced by the partial degrees $\deg_{Y_i}(\cdot)$, $i=0,1,\ldots,n$. Set 
\vskip 0,5mm
\centerline{$\underline Y^+=(Y_0,Y_1,\ldots,Y_n)$ and $\underline Y=(Y_1,\ldots,Y_n)$.} 
\vskip 0,5mm

\noindent
According to Definition \ref{def:schinzel}, what is to be proved is this:
\vskip 1mm

\noindent
(**) \emph{given $s\geq 1$ polynomials $P_1,\ldots,P_s \in {\mathcal Z}[\underline Y^+][\underline T]$, irreducible in $\mathcal{Q}(\underline Y^+)[\underline T]$
and such that the product $P_1\cdots P_s$ has no nonunit divisor in ${\mathcal Z}[\underline Y^+]$, and given any constant $A>0$, the subset ${\mathcal S}_{\mathcal{Z}}(\underline P,A)\subset {\mathcal Z}[\underline Y^+]^k$ consisting of all $k$-tuples $(M_1(\underline Y^+),\ldots,M_k(\underline Y^+))$ 
of polynomials such that}
\vskip 1mm
%\noindent
\noindent
{\rm (i)} \emph{$P_i(\underline Y^+, M_1(\underline Y^+),\ldots,M_k(\underline Y^+))$ is irreducible in ${\mathcal Z}[\underline Y^+]$ ($i=1,\ldots,s$)},
\vskip 1mm 

\noindent
{\rm (ii)} $\deg_{Y_i}(M_j) >A$, $i=0,1,\ldots,n$, $j=1,\ldots,k$.
\vskip 1mm

\noindent
\emph{is Zariski-dense.}
\vskip 1mm

The idea is to apply Theorem \ref{thm:main3} in the special case that the integrally Hilbertian ring there is ${\mathcal Z}[Y_0]$. Except for requirement (ii) with $i=0$, Theorem \ref{thm:main3} in this special case, gives a more precise conclusion than (**) for which the degrees $\deg_{Y_i}(M_j)$ ($i=1,\ldots,n$, $j=1,\ldots,k$) can be any integers $d_{ij}$ provided that these integers are suitably large. The whole desired conclusion (**) follows then from a strengthening relative to the variable $Y_0$ of Theorem \ref{thm:main3} 
for ${\mathcal Z}[Y_0]$. 
%combined with Proposition \ref{prop:R[U]-loc-Schinzel-ring} (explicit form of Theorem \ref{thm:main1-2}), and . 

More specifically, Theorem \ref{thm:main3} asserts that the set, denoted there by 

\vskip 0,5mm

\centerline{${\mathcal S}_{{\mathcal Z}[Y_0][\underline Y]}(\underline P,{\underline d}_1,\ldots,{\underline d}_n)$} 
\vskip 1mm
%\subset {\mathcal P}ol_{{\mathcal Z},n,\underline{d}_1} \times \cdots \times {\mathcal P}ol_{{\mathcal Z},n,\underline{d}_k}

\noindent
is Zariski-dense (under appropriate assumptions on ${\underline d}_1,\ldots,{\underline d}_n)$. This subset is viewed as a subset of some affine space ${\mathbb A}^N$ of $N$-tuples (by viewing the polynomials $M_i (\underline Y^+)$ as the tuples of their coefficients). The strengthening we need is to be able to guarantee that, for any prescribed real number $A>0$, if the coordinates 
(in ${\mathcal Z}[Y_0]$) of the $N$-tuples are also required to be of degree $> A$  in $Y_0$, the resulting 
subset of ${\mathcal S}_{{\mathcal Z}[Y_0][\underline Y]}(\underline P,{\underline d}_1,\ldots,{\underline d}_n)$ remains Zariski-dense.

As Proposition \ref{prop:generalSchH} recapitulates, the proof of Theorem \ref{thm:main3} shows that the 
$N$-tuples in question (those in the set ${\mathcal S}_{{\mathcal Z}[Y_0][\underline Y]}(\underline P,{\underline d}_1,\ldots,{\underline d}_n)$) come from an application of the 
integral Hilbertian property to some polynomials. Thus, the question becomes 
whether one can guarantee that the integral Hilbertian property of ${\mathcal Z}[Y_0]$  
(which is Theorem \ref{thm:schinzel-main}) is true in this more refined version --- that the "good" 
specializations can be required to be of degree 
$> A $ in $Y_0$. 

%REMARK ?
This final check should be done in two steps because integral Hilbertianity of ${\mathcal Z}[Y_0]$ is proved via the locally Schinzel property. Regarding the latter, the needed extra requirement is indeed guaranteed, by Proposition \ref{prop:R[U]-loc-Schinzel-ring}, which produces, as required for the locally Schinzel property, a certain arithmetic progression $(\omega \ell + \alpha)_{\ell \in {\mathcal Z}[Y_0]}$, with $\omega, \alpha \in {\mathcal Z}[Y_0]$, and shows further that  $\omega$ can be chosen of degree $>A$ in $Y_0$. It remains to check that in the proof of Proposition \ref{prop:schinzel+Hil=intH} (deducing integral Hilbertianity from the locally Schinzel property), the produced ``good'' tuples can also be taken of degree $>A$ in $Y_0$: this is clear as each of their coordinates can be picked in an infinite subset of some arithmetic progression as above.

%REMARK ?
%I think that we could even require that the partial degrees in $Y_0$ be *any*
%suitably large multiple of $\ell$ ($= 1$ or $char(Z)$). This is what Theorem 1.2 of 
%[BDN20] says when $Z$ is a field. The approach there is slightly different 
%and does not go through the intermediate "locally Schinzel" stage. We could 
%probably adjust our approach to obtain this slightly better conclusion, with 
%some work... I'm not sure this is worth it, especially since I'd rather keep 
%Corollary 1.9 phrased in terms of Schinzel rings (which do not need this bonus).

\subsection{Proof of Theorem \ref{thm:schinzel1=>Schinzelk}} \label{ssec:pf-schinzel1=>Schinzelk}
Denote by $H_1,\dots,H_n$ the Weil heights corresponding to the sets of primes $S_1,\dots,S_n$ respectively. Let $P_1(\underline{T}),\dots,P_s(\underline T) \in \mathcal{Z}[\underline T]$ be polynomials that are irreducible in $\mathcal{Q}[\underline T]$ and such that
$
P_1 \cdots P_s
$ has no fixed divisor w.r.t. $\underline{T}$. As in the proof of Theorem \ref{thm:hypschistronggen}, for each $r=1,\dots,k$, consider the tuple of variables $\underline{\lambda}_r=(\lambda_{r,0},\dots,\lambda_{r,\ell_r-1})$ (with $\ell_r \geq 1$), and define the polynomial 
$$
M_r(\underline{\lambda}_r,Y)=\theta_r+ \omega \cdot R_r(\underline{\lambda}_r,Y),
$$ where $R_r$ is as in \eqref{eq:genericRr} and $\omega$ as defined in the general setting, just before Claim \ref{claim:invardiv}. Set $\underline{\lambda}=(\underline{\lambda}_1,\dots,\underline{\lambda}_k)$, and define
$$
\underline{M}(\underline{\lambda},Y)=(M_1(\underline{\lambda}_1,Y),\dots,M_k(\underline{\lambda}_k,Y)).
$$

As in the proof of Theorem \ref{thm:hypschistronggen}, for 
$\underline{m}=(\underline{m}_1,\dots,\underline{m}_k)$ in a Zariski-dense subset  $ \mathcal H \subset \mathcal{Z}^{\ell_1} \times \dots \times \mathcal{Z}^{\ell_r}=\mathcal{Z}^{\ell_1+ \dots+ \ell_r}$, each polynomial
$
P_j(\underline{M}(\underline{m},Y)) 
$ is irreducible in $\mathcal{Q}[Y]$, and the product
$$
\prod_{j=1}^s P_j(\underline{M}(\underline m ,Y))
$$ has no fixed divisor w.r.t. $Y$. 

For each $\underline{m} \in \mathcal{H}$, there exists a constant $c_{\underline{m}} >0$ such that
\begin{equation}\label{eq:minweilheigh}
H_i(M_r(\underline{m},x)) \geq c_{\underline{m}}H_i(x)^{\deg_Y(M_{r}(\underline{m},Y))},
\end{equation} for all $r=1,\dots,k$, all $i=1,\dots,n$, and all $x \in \mathcal{Z}$.

Fix $A>0$. Choose $B>0$ such that $H_i(x) \geq B$ implies $c_{\underline{m}}H_i(x)^{\deg_Y(M_{r}(\underline{m},Y))} \geq A$ for all $r=1,\dots,k$ and all $i=1,\dots,n$. Since $\mathcal{Z}$ is {is assumed to satisfy the $1$-parameter case of the Schinzel property}, for each $\underline{m} \in \mathcal{H}$ there exists an infinite set $K_{\underline m,B} \subset \mathcal{Z} \cap \bigcap_{i=1}^n \{  H_i \geq B \}$ such that for every $x \in K_{\underline m,B}$, each polynomial $P_j(\underline{M}(\underline{m},x))$ is irreducible in $\mathcal Z$. 

 For each $\underline{m} \in \mathcal H$, note that the set
$$
L_{\underline{m},A}:=\{x \in K_{\underline{m},B} \mid H_i(M_r(\underline{m}_r,x) \geq A,\quad\textrm{for all } r=1,\dots,k\,\,\textrm{and }i=1,\dots,n\}
$$ is infinite. It remains to prove that
$$
R_A:=\bigcup_{\underline{m} \in \mathcal{H}} E_{\underline{m},A} \subset \mathcal{Z}^k \cap \bigcap_{i=1}^n \{ H_i \geq A\}
$$ is Zariski-dense, where
$$
E_{\underline{m},A}=\{\underline{M}(\underline{m},x) \mid x \in L_{\underline{m},A}\}.
$$
For that, let $Q(\underline{X}) \in \mathcal{Q}[\underline{X}]$ be a polynomial that vanishes on $R_A$. For every $\underline{m} \in \mathcal H$, since $L_{\underline{m},A}$ is infinite, it follows that
$$
Q(\underline{M}(\underline{m},Y))
$$ 
is zero as a polynomial in $Y$. As $\mathcal H$ is Zariski-dense, we deduce that
$$
Q(\underline{M}(\underline{\lambda},Y))=0.
$$ Since $0=\varphi^{-1}(Q(\underline{M}(\underline{\lambda},Y)))=Q(\lambda_{1,0},\dots,\lambda_{k,0})$--where $\varphi$ is as in the proof of Claim \ref{claim:irrfigencas})--we obtain $Q=0$. Consequently $R_A$ is Zariski-dense.

\bibliography{main-2}
\bibliographystyle{alpha}
\end{document}